\documentclass[12pt,reqno,fleqn]{amsart}
\usepackage{amssymb}
\usepackage{amscd}
\usepackage{amsxtra}
\usepackage[mathscr]{eucal}

\usepackage{hyperref}

\theoremstyle{plain}
\newtheorem{theorem}{Theorem}[section]
\newtheorem{lemma}[theorem]{Lemma}
\newtheorem{proposition}[theorem]{Proposition}
\newtheorem{corollary}[theorem]{Corollary}
\newtheorem*{proposition*}{Proposition}

\theoremstyle{definition}
\newtheorem{conjecture}[theorem]{Conjecture}
\newtheorem{remark}[theorem]{Remark}
\newtheorem{definition}[theorem]{Definition}
\newtheorem{example}[theorem]{Example}

\theoremstyle{plain}
\newtheorem{Theorem}{Theorem}[section]

\newtheorem{Conjecture}[Theorem]{Conjecture}
\newtheorem{Proposition}[Theorem]{Proposition}

\numberwithin{equation}{section}

\newcommand{\fA}{{\mathfrak A}}
\newcommand{\fa}{{\mathfrak a}}

\newcommand{\fB}{{\mathfrak B}}

\newcommand{\cC}{{\mathcal C}}

\newcommand{\nCl}{{\mathrm{Cl'_{\fa}}^{+}}}

\newcommand{\cD}{{\mathcal D}}

\newcommand{\ve}{{\varepsilon}}

\newcommand{\fF}{{\mathcal F}}
\newcommand{\bF}{{\mathbf F}}

\newcommand{\cG}{{\mathcal G}}

\newcommand{\cH}{{\mathcal H}}

\newcommand{\cHom}{{\mathcal Hom}}
\newcommand{\wh}[1]{{\widehat{#1}}}
\newcommand{\ov}[1]{{\overline{#1}}}

\newcommand{\1}{{\mathbf 1}}  

\newcommand{\vp}{{\varpi}}

\newcommand{\cL}{{\mathcal L}}

\newcommand{\cM}{{\mathcal M}}

\newcommand{\cN}{{\mathcal N}}

\newcommand{\cP}{{\mathcal P}}
\newcommand{\fp}{{\mathfrak p}}
\newcommand{\sP}{{\mathscr P}}
\newcommand{\vphi}{{\varphi}}
\newcommand{\vpi}{{\varpi}}
\newcommand{\bpsi}{{\Psi}}
\newcommand{\bQ}{{\mathbf Q}}

\newcommand{\br}{{\mathbf r}}
\newcommand{\bR}{{\mathbf R}}
\newcommand{\bcR}{{K \mathcal R}}
\newcommand{\cR}{{\mathcal R}}

\newcommand{\cS}{{\mathscr S}} 
\newcommand{\ti}[1]{{\tilde{#1}}}

\newcommand{\bZ}{{\mathbf Z}}

\DeclareMathOperator{\ab}{ab}

\DeclareMathOperator{\Cl}{Cl}
\DeclareMathOperator{\co}{co}

\DeclareMathOperator{\Coker}{Coker}

\DeclareMathOperator{\Det}{Det}
\DeclareMathOperator{\disc}{disc}

\DeclareMathOperator{\epi}{epi}

\DeclareMathOperator{\Gal}{Gal}
\DeclareMathOperator{\GL}{GL}
\DeclareMathOperator{\Hom}{Hom}

\DeclareMathOperator{\Image}{Im}
\DeclareMathOperator{\Irr}{Irr}
\DeclareMathOperator{\Ind}{Ind}
\DeclareMathOperator{\Ker}{Ker}
\DeclareMathOperator{\LC}{LC}

\DeclareMathOperator{\loc}{loc}
\DeclareMathOperator{\Map}{Map}

\DeclareMathOperator{\nrd}{nrd}

\DeclareMathOperator{\rag}{rag}

\DeclareMathOperator{\Stab}{Stab}
\DeclareMathOperator{\Supp}{Supp}

\DeclareMathOperator{\Tr}{Tr}

\begin{document}
\title[Relative Galois module structure]{On the relative Galois module
structure of rings of integers in tame extensions} 
\author{A. Agboola \and L. R. McCulloh}
\address{Department of Mathematics \\
         University of California \\
         Santa Barbara, CA 93106. }
\email{agboola@math.ucsb.edu}
\address{Department of Mathematics \\
         University of Illinois \\
         1409 W. Green Street \\
         Urbana, IL 61801.}
\email{mcculloh@math.uiuc.edu}
\date{July10, 2018. Final version. To appear in Algebra and Number Theory.}

\begin{abstract}
Let $F$ be a number field with ring of integers $O_F$ and let $G$ be a
finite group. We describe an approach to the study of the set of
realisable classes in the locally free class group $\Cl(O_FG)$ of
$O_FG$ that involves applying the work of the second-named author in
the context of relative algebraic $K$ theory. For a large class of
soluble groups $G$, including all groups of odd order, we show
(subject to certain mild conditions) that the set of realisable
classes is a subgroup of $\Cl(O_FG)$. This may be viewed as being a
partial analogue in the setting of Galois module theory of a classical
theorem of Shafarevich on the inverse Galois problem for soluble
groups.
\end{abstract}

\maketitle

\tableofcontents


\section{Introduction}

Suppose that $F$ is a number field with ring of integers $O_F$, and
let $G$ be a finite group. If $F_\pi/F$ is any tame Galois $G$-algebra
extension of $F$, then a classical theorem of E. Noether implies that
the ring of integers $O_\pi$ of $F_\pi$ is a locally free
$O_FG$-module, and so determines a class $(O_\pi)$ in the locally free
class group $\Cl(O_FG)$ of $O_FG$. Hence, if we write $H^1_t(F,G)$ for
the pointed set of isomorphism classes of tame $G$-extensions of $F$,
then we obtain a map of pointed sets
\[
\psi: H^1_t(F,G) \to \Cl(O_FG); \quad [\pi] \mapsto (O_{\pi}).
\]
Even when $G$ is abelian, so that $H^1_t(F, G)$ is actually a group,
this map is almost never a group homomorphism. We say that an element
$c \in \Cl(O_FG)$ is \textit{realisable} if $c = (O_\pi)$ for some
tame Galois $G$-algebra extension $F_\pi/F$, and we write $\cR(O_FG)$
for the collection of realisable classes in $\Cl(O_FG)$.  These
classes are natural objects of study, and they have arisen in a number
of different contexts in Galois module theory. The problem of
describing $\cR(O_FG)$ for a given $G$ may be viewed as being a loose
analogue of the inverse Galois problem in the setting of arithmetic
Galois module theory.

When $G$ is abelian, the second-named author has given a complete
description of $\cR(O_FG)$ by showing that it is equal to the kernel
of a certain Stickelberger homomorphism on $\Cl(O_FG)$ (see
\cite{Mc1}). In particular, he has shown that $\cR(O_FG)$ is in fact a
group. In subsequent unpublished work (\cite{Mc2}, \cite{McS}) he
showed that, for arbitrary $G$, the set $\cR(O_FG)$ is always
contained in the kernel of this Stickelberger homomorphism, and he
raised the question of whether or not $\cR(O_FG)$ is in fact always
equal to this kernel. This question has inspired research by a number
of different authors, and we refer the reader to e.g. \cite{BS},
\cite{BGS}, \cite{FB}, and to the bibliographies of these papers, for
further information concerning previous work on this problem.

In this paper we shall describe a new approach to studying this
topic that involves combining the methods introduced by the
second-named author in \cite{Mc1} and \cite{Mc2} with techniques
involving relative algebraic $K$-theory and categorical twisted forms
introduced by D. Burns and the first-named author in \cite{AB}. This
enables us to both clarify certain aspects of the theory of realisable
classes and to establish new results. Although our perspective is
somewhat different, it should be stressed that many of the
main ideas that we use are in fact already present in some form in
\cite{Mc1} and \cite{Mc2}.

Let us now describe the contents of this paper in more detail. In
Section \ref{S:galres} we recall some basic facts concerning principal
homogeneous spaces, Galois algebras and resolvends; these play a key
role in everything that follows. Next, we assemble a number of
technical results explaining how resolvends may be used to compute
discriminants of rings of integers in Galois $G$-extensions. We also
discuss how certain Galois cohomology groups may be expressed in terms
of resolvends in a manner that is very useful for calculations in
class groups and $K$-groups.  In Section \ref{S:det} we explain how
determinants of resolvends may be represented in terms of certain
character maps, and we recall an approximation theorem of A. Siviero
(which is in turn a variant of \cite[Theorem 2.14]{Mc1}).

We begin Section \ref{S:reshom} by outlining the results we need about
twisted forms and relative algebraic $K$-groups from \cite{AB}. Each
tame $G$-extension $F_{\pi}/F$ of $F$ has an associated resolvend
isomorphism
\[
\br_G: F_{\pi} \otimes_F F^c \simeq F^cG
\]
of $F^cG$-modules, and this may be used to construct a categorical
twisted form which is represented by an element $[O_{\pi},O_FG;\br_G]$
in a certain relative algebraic $K$-group $K_0(O_FG,F^c)$. The group
$K_0(O_FG,F^c)$ admits a natural surjection onto the locally free
class group $\Cl(O_FG)$, sending $[O_{\pi},O_FG;\br_G]$ to
$(O_{\pi})$, and so there is a map of pointed sets
\[
\bpsi: H^1_t(F,G) \to K_0(O_FG,F^c); \quad [\pi] \mapsto
     [O_{\pi},O_FG;\br_G]
\]
which is a refinement (more precisely, a lifting) of the map $\psi$
above.

Crucial to our approach is the fact that each of the constructions
that we have just described admits a local variant. Let $v$ be any
place of $F$, and write $H^1_t(F_v,G)$ for the pointed set of
isomorphism classes of tame $G$-extensions of $F_v$.  Then there is a
localisation homomorphism
\[
\lambda_v: K_0(O_FG, F^c) \to K_0(O_{F_v}G, F^c_v)
\]
as well as a map of pointed sets
\[
\bpsi_v: H^1_t(F_v,G) \to K_0(O_{F_v}G,F_v^c); \quad [\pi_v] \mapsto
     [O_{\pi_v},O_{F_v}G; \br_G].
\]

The following result reflects the fact that $[O_{\pi},O_FG;\br_G]$ is
a much finer structure invariant than $(O_{\pi})$ (see Proposition
\ref{P:kernel} below):

\begin{Proposition} \label{P:A}
The kernel of $\bpsi$ is finite.
\end{Proposition}

Let $G'$ denote the derived subgroup of $G$. We may identify $H^1(F,
G')$ with a subset of $H^1(F, G)$ via the exact sequence $0 \to G' \to
G \to G^{ab} \to 0$. Proposition \ref{P:A} is proved by showing that
$\Ker(\bpsi)$ is a subset of the pointed set $H^{1}_{fnr}(F, G')$ of
isomorphism classes of $G'$-Galois $F$-algebras that are unramified at
all finite places of $F$; this last set is finite because there are
only finitely many unramified extensions of $F$ of bounded degree. If
$G$ is abelian, the map $\bpsi$ is injective (see Proposition
\ref{P:gab}). In many cases one can show that $\Ker(\bpsi) =
H^{1}_{fnr}(F, G')$, but we do not know whether this equality always
holds.

Write $\bcR(O_FG)$ for the image of $\bpsi$, i.e. for the collection
of realisable classes of $K_0(O_FG, F^c)$. The central conjecture of
this paper gives a precise description of $\bcR(O_FG)$ in terms of a
local-global principle for the relative algebraic $K$-group $K_0(O_FG,
F^c)$. This may be described as follows. 

For each place $v$ of $F$, let $H^{1}_{nr}(F_v, G)$ denote the subset
$H^1_t(F_v,G)$ consisting of isomorphism classes of unramified
$G$-extensions of $F_v$. We define a pointed set of ideles
$J(H^1_t(F,G))$ of $H^1_t(F,G)$ to be the restricted direct product
over all places $v$ of the sets $H^1_t(F_v,G)$ with respect to
the subsets $H^{1}_{nr}(F_v,G)$ (see Definition \ref{D:psid}). The
natural maps $H^1_t(F,G) \to H^1_t(F_v,G)$ for each $v$ induce a map
$H^1_t(F,G) \to J(H^1_t(F,G))$. We also define a group of ideles
$J(K_0(O_FG,F^c))$ of $K_0(O_FG,F^c)$ to be the restricted direct
product over all places of $F$ of the groups $K_0(O_{F_v}G,F^c_v)$
with respect to the subgroups $K_0(O_{F_v}G, O_{F^{c}_{v}})$ (see
Definition \ref{D:kiddes}).  We show that the maps $\lambda_v$ above
induce an injective localisation map
\[
\lambda: K_0(O_FG,F^c) \to J(K_0(O_FG,F^c))
\]
(see Proposition \ref{P:locinj}), and that the maps $\bpsi_v$ induce
an idelic version
\[
\bpsi^{id}: J(H^1_t(F,G)) \to J(K_0(O_FG,F^c))
\]
of the map $\bpsi$ (see Definition \ref{D:psid}). We conjecture that
$\bcR(O_FG)$ has the following description (see Conjecture \ref{C:cc}
below):

\begin{Conjecture} \label{C:conj}
$\bcR(O_FG) = \lambda^{-1}(\Image(\bpsi^{id}))$.
\end{Conjecture}

In other words, our conjecture predicts that an element $x$ lies in the
image of $\bpsi$ if and only if $\lambda_v(x)$ lies in the image of
$\bpsi_v$ for every place $v$ of $F$. We remark that it
follows directly from the definitions that
\[
\bcR(O_FG) \subseteq \lambda^{-1}(\Image(\bpsi^{id})).
\]

We point out that, in contrast to $\cR(O_FG)$, it is not difficult to
show that if $G$ is non-trivial, then $\bcR(O_FG)$ is never a subgroup
of $K_0(O_FG, F^c)$ (cf. \cite[Remarks 6.13(i)]{AB}, \cite[Remark
  2.10(iii)]{AB2}). Nevertheless, by applying the methods of
\cite{Mc1} and \cite{Mc2} in the present context, we show that
Conjecture \ref{C:conj} implies both an affirmative answer to the
second-named author's question concerning $\cR(O_FG)$ as well as a
positive solution to the inverse Galois problem for $G$ over $F$ (see
Theorems \ref{T:idgp}, \ref{T:congp} and \ref{T:fieldres} below):

\begin{Theorem} \label{T:B}
If Conjecture \ref{C:conj} holds, then $\cR(O_FG)$ is a subgroup of
$\Cl(O_FG)$. Furthermore, if $c \in \cR(O_FG)$, then there exist
infinitely many $[\pi] \in H^1_t(F, G)$ such that $F_{\pi}$ is a field
and $(O_{\pi}) = c$. The extensions $F_{\pi}/F$ may be chosen to have
ramification disjoint from any finite set $S$ of places of $F$. In
particular, the inverse  Galois problem for $G$ admits a positive
solution over $F$.
\end{Theorem}

In order to orient the reader, we shall now briefly indicate the main
ideas involved in the proof of Theorem \ref{T:B}. 

We begin by observing that the long exact sequence of relative
algebraic $K$-theory yields a sequence
\[
K_1(F^cG) \xrightarrow{\partial^1} K_0(O_FG, F^c)
\xrightarrow{\partial^0} \Cl(O_FG) \to 0.
\]
Hence, in order to show that $\cR(O_FG) = \Image(\psi)$ is a subgroup
of $\Cl(O_FG)$, it suffices to show that $\partial^1(K_1(F^cG)) \cdot
\Image(\bpsi) $ is a subgroup of $K_0(O_FG, F^c)$.

To do this, we first show that it suffices to prove that
\[
\lambda(\partial^1(K_1(F^cG))) \cdot \Image(\bpsi^{id})
\]
is a subgroup of $J(K_0(O_FG, F^c))$. Once this is done, it is not
hard to show that $\partial^1(K_1(F^cG)) \cdot \Image(\bpsi)$ is equal
to the kernel of the homomorphism
\[
K_0(O_FG, F^c) \xrightarrow{\lambda} J(K_0(O_FG, F^c)) \to
\frac{J(K_0(O_FG, F^c))}{\lambda[\partial^{1} (K_1(F^cG))] \cdot
\Image(\bpsi^{id})},
\]
and so is indeed a subgroup of $K_0(O_FG, F^c)$ (see Theorem
\ref{T:congp} below). The crux of the proof of the first part of
Theorem \ref{T:B} therefore consists of showing that
$\lambda(\partial^1(K_1(F^cG))) \cdot \Image(\bpsi^{id})$ is a
subgroup of $K_0(O_FG, F^c)$.

This is accomplished as follows.  Write $G(-1)$ for the
group $G$ (viewed as a set) endowed with an action of $\Omega_F$ via
the inverse cyclotomic character. Although in general this is only an
action on $G$ as a set (rather than via automorphisms of $G$), the
induced action on conjugacy classes of $G$ does induce an action on
the centre $Z(F^c[G])$ of the group ring $F^cG$. We write
$Z(F^c[G(-1)])$ to denote $Z(F^c[G])$ endowed with this action. We set
\[
\Lambda(FG):= Z(F^c[G(-1)])^{\Omega_F},
\]
and we write $\Lambda(O_FG)$ for the (unique) $O_F$-maximal order in
$\Lambda(FG)$. For each place $v$ of $F$, we define $\Lambda(F_vG)$
and $\Lambda(O_{F_v}G)$ in an analogous manner. We write
$J(\Lambda(FG))$ for the restricted direct product over all places of
$F$ of the groups $\Lambda(F_vG)^{\times}$ with respect to the
subgroups $\Lambda(O_{F_v}G)^{\times}$.

Let $\Irr(G)$ denote the set of irreducible characters of $G$.
Motivated by an analysis of normal integral basis generators of tame
local extensions, we define a Stickelberger pairing
\[
\langle-,-\rangle_G: \Irr(G) \times  G \to \bQ.
\]
(Loosely speaking, this may be viewed as being a monodromy-type
pairing that encodes ramification data associated to tame extensions
of local fields in a uniform manner (cf. Definition \ref{D:modstick}
below).) We then use this pairing to construct a $K$-theoretic
transpose Stickelberger homomorphism
\[
K\Theta^t: J(\Lambda(FG)) \to J(K_0(O_FG, F^c)).
\]
The homomorphism $K\Theta^t$ is closely related to the map
$\bpsi^{id}$ in the following way. We show that even though the map
$\bpsi_v$ is just a map of pointed sets, the image
$\bpsi_v(H^{1}_{nr}(F_v, G))$ of the restriction of $\bpsi_v$ to
$H^{1}_{nr}(F_v, G)$ is in fact a subgroup of $K_0(O_{F_v}G,
F^{c}_{v})$ for each $v$.  Using an approximation theorem for
$J(\Lambda(FG))$, we show further that, for a suitable choice of
auxiliary ideal $\fa$ of $O_F$, the homomorphism
$K\Theta^t$ may be used to construct a homomorphism
\[
\Theta^{t}_{\fa}: \nCl(\Lambda(O_FG)) \to
\frac{J(K_0(O_FG, F^c))}{\lambda[\partial^1(K_1(F^cG))] \cdot
  \prod_{v} \bpsi_v(H^{1}_{nr}(F_v, G))},
\]
where $\nCl(\Lambda(O_FG))$ is a certain finite quotient of
$J(\Lambda(FG))$. We prove that
\[
\Image(\Theta^{t}_{\fa}) = \Image(\ov{\bpsi^{id}}),
\]
where $\ov{\bpsi^{id}}$ denotes the composition of $\bpsi^{id}$ with
the obvious quotient map
\[
J(K_0(O_FG, F^c)) \to \frac{J(K_0(O_FG, F^c))}{\lambda[
  \partial^1(K_1(F^cG))] \cdot \prod_{v} \bpsi_v(H^{1}_{nr}(F_v, G))}.
\]
We then show that this in turn implies that
\begin{equation} \label{E:idimage}
\lambda(\partial^1(K_1(F^cG))) \cdot \Image(K\Theta^t) = 
\lambda(\partial^1(K_1(F^cG))) \cdot \Image(\bpsi^{id}).
\end{equation}
In particular, this proves that the right-hand side of
\eqref{E:idimage} is a subgroup of $J(K_0(O_FG, F^c))$, as claimed.
This completes our outline of the proof of the first part of Theorem
\ref{T:B}.

The strategy of the proof of the second part of Theorem \ref{T:B} may
be very roughly described as follows. Suppose that $x \in
\lambda^{-1}(\Image(\bpsi^{id}))$.  By using the map $K\Theta^t$
together with a suitable approximation theorem on $J(K_0(O_FG, F^c))$,
we show that there are infinitely many $y \in
\lambda^{-1}(\Image(\bpsi^{id}))$ such that (i) $\partial^0(y) =
\partial^0(x)$, and (ii) each $y$ corresponds via Conjecture
\ref{C:conj} to an element $[\pi_y] \in H^1_t(F, G)$ which is ramified
(away from $S$) in such a way that $\pi_y \in \Hom(\Omega_F, G)$ is
forced to be surjective. This in turn implies that $F_{\pi_y}$ is a
field (rather than just a Galois algebra), and so the inverse Galois
problem for $G$ admits a positive solution over $F$.
\smallskip

Let us now turn to our results concerning the validity of Conjecture
\ref{C:conj}.
\smallskip

When $G$ is abelian, we obtain the following refinement of
\cite[Theorem 6.7]{Mc1} (see Theorem \ref{T:ab} below):

\begin{Theorem} \label{T:C}
Conjecture \ref{C:conj} is true if $G$ is abelian.
\end{Theorem}

By combining our methods with work of Neukirch, we are able to
establish a variant of Conjecture \ref{C:conj} for a large class of
soluble groups, including all groups of odd order (see Theorems
\ref{T:domco} and \ref{T:rodd} below). We thereby obtain the following
result, which may be viewed as being a partial
analogue of a classical theorem of Shafarevich (see \cite{S}) on the
inverse Galois problem for soluble groups in the context of arithmetic
Galois module theory. (See Theorem \ref{T:odd} of the main text.)

\begin{Theorem} \label{T:D}
Suppose that $G$ is of odd order and that $(|G|, h_F)=1$, where
$h_F$ denotes the class number of $F$. Suppose also that $F$ contains
no non-trivial $|G|$-th roots of unity. Then $\cR(O_FG)$ is a subgroup
of $\Cl(O_FG)$.  If $c \in \cR(O_FG)$, then there exist infinitely
many $[\pi] \in H^1_t(F, G)$ such that $F_{\pi}$ is a field and
$(O_{\pi}) = c$. The extensions $F_{\pi}/F$ may be chosen to have
ramification disjoint from any finite set $S$ of places of $F$.
\end{Theorem}

While it is perhaps conceivable that it might be possible to remove
the hypothesis $(|G|, h_F) = 1$ of Theorem \ref{T:D} using methods
similar to those of the present paper (although we do not as yet know
how to do this), the same probably cannot be said of the condition
concerning the number of roots of unity in $F$. This latter hypothesis
is forced upon us because our proof makes crucial use of a lifting
theorem of Neukirch (see Section \ref{S:neu}) where such hypotheses
are unavoidable (cf. the last paragraph of the Introduction of
\cite{Neu}). It would be interesting to determine whether or not the
methods of \cite{S} can be used to prove a result similar to Theorem
\ref{T:D} for all soluble groups.

The results and techniques introduced in this paper suggest a number
of different avenues of further investigation. For example, our
methods may also be applied in the context of the relative Galois
module structure of the square root of the inverse different as
studied by C. Tsang (see \cite{Ts16}, \cite{Ts17}), and it seems
reasonable to expect that an analogue of Theorem \ref{T:D} holds in
this setting. Applying the methods of \cite{A12} to the study of
counting and equidistribution problems involving cohomological classes
in relative algebraic $K$-groups should lead to new results concerning
similar problems for number fields, generalising certain aspects of
e.g.  \cite{W89} and \cite{Malle}. Our techniques may also be applied
in the setting of global function fields (see \cite{AB} and
\cite{AB01}), and it would be of interest to further investigate the
connection between the approach adopted here and that taken in
e.g. \cite{Ch94} (cf. for example, \cite[Section 4]{AB}).

Here is an outline of the rest of this paper. In Section
\ref{S:local}, we explain a hitherto unpublished result of the
second-named author that describes how resolvends of normal integral
bases of tamely ramified extensions of non-archimedean local fields
admit certain \textit{Stickelberger factorisations} (see Definition
\ref{D:stickfac}); this is a non-abelian analogue of a version of
Stickelberger's factorisation of abelian Gauss sums. A somewhat
analogous (but much simpler) framework over $\bR$ is described in
Section \ref{S:local2}.

In Section \ref{S:stickpair}, we recall the definition and properties
of the Stickelberger pairing. We also give a new character-theoretic
description of this pairing (see Proposition \ref{P:stickpair}) as
well as an application of this description (see Corollary \ref{C:stick}).

We construct a $K$-theoretic version of the transpose Stickelberger
homomorphism in Section \ref{S:stickhom}, and we also briefly describe
an alternative approach to defining the Stickelberger pairing and
establishing its basic properties. In Section \ref{S:modray} we
construct transpose Stickelberger homomorphisms $\Theta^{t}_{\fa}$ on
modified narrow ray class groups $\nCl(\Lambda(O_FG))$. These are used
in Section \ref{S:idgp} to prove Theorem \ref{T:idgp}, thereby
completing the proof of the first part of Theorem \ref{T:B}.

In Section \ref{S:ker} we prove Proposition \ref{P:A}, and we explain
how a weaker form of Conjecture \ref{C:conj} implies that every
realisable class in $\Cl(O_FG)$ may be realised (in infinitely many
ways) by rings of integers of tame field (and not merely Galois
algebra) $G$-extensions of $F$. This proves the second part of Theorem
\ref{T:B}.

We give a proof of Theorem \ref{T:C} in Section \ref{S:ab}. In Section
\ref{S:neu}, we describe work of Neukirch on the solution to an
embedding problem that is required for the proof of Theorem
\ref{T:D}. This proof is completed in Section \ref{S:odd} via showing
that a suitable variant of Conjecture \ref{C:conj} holds for a large
class of soluble groups (see Definition \ref{D:R} and Theorems
\ref{T:rgreat} and \ref{T:domco}).

We are very grateful indeed to Andrea Siviero for his extremely
detailed comments on an earlier draft of this paper, and to Ruth
Sergel for her most perceptive remarks at a critical stage of this
project. We heartily thank Nigel Byott and Cindy Tsang for the many
very helpful comments, questions and corrections that we received from
them, and Bob Guralnick for a very helpful remark concerning the proof
of Corollary \ref{C:stick} below. We are also extremely grateful to
the anonymous referees whose very careful reading of the manuscript
led us to correct and considerably strengthen our original results,
and to significantly improve our exposition.
\medskip

\noindent{}{\bf Notation and conventions.} 

For any field $L$, we write $L^c$ for an algebraic closure of $L$, and
we set $\Omega_L:= \Gal(L^c/L)$. If $L$ is a number field or a
non-archimedean local field (by which we shall always mean a finite
extension of $\bQ_p$ for some prime $p$), then $O_L$ denotes the ring
of integers of $L$. If $L$ is an archimedean local field, then we
adopt the usual convention of setting $O_L = L$.
\smallskip

Throughout this paper, $F$ will denote a number field. For each place
$v$ of $F$, we fix an embedding $F^c \to F_{v}^{c}$, and we view
$\Omega_{F_v}$ as being a subgroup of $\Omega_F$ via this choice of
embedding. We write $I_v$ for the inertia subgroup of $\Omega_{F_v}$
when $v$ is finite.

The symbol $G$ will always denote a finite group upon which $\Omega_F$
acts trivially. If $H$ is any finite group, we write $\Irr(H)$ for the
set of irreducible $F^c$-valued characters of $H$ and $R_H$ for the
corresponding ring of virtual characters. We write $\1_H$ (or
simply $\1$ if there is no danger of confusion) for the trivial
character in $R_H$. If $h \in H$, then we write $c(h)$ for the
conjugacy class of $h$ in $H$ and $\cC(H)$ for the set of conjugacy
classes of $H$. We denote the derived subgroup of $H$ by $H'$.
\smallskip

If $L$ is a number field or a local field, and $\Gamma$ is any group
upon which $\Omega_L$ acts continuously, we identify $\Gamma$-torsors
over $L$ (as well as their associated algebras, which are Hopf-Galois
extensions associated to $A_{\Gamma}:= (L^c\Gamma)^{\Omega_{L}}$) with
elements of the set $Z^1(\Omega_L, \Gamma)$ of $\Gamma$-valued
continuous $1$-cocycles of $\Omega_L$ (see \cite[I.5.2]{Se} and
Section \ref{S:galres} below). If $\pi \in Z^1(\Omega_L, \Gamma)$,
then we write $L_\pi/L$ for the corresponding Hopf-Galois extension of
$L$, and $O_\pi$ for the integral closure of $O_L$ in $L_\pi$.  (Thus
$O_{\pi} = L_{\pi}$ if $L$ is an archimedean local field.)  Each such
$L_{\pi}$ is a principal homogeneous space (p.h.s.) of the Hopf
algebra $\Map_{\Omega_L}(\Gamma, L^c)$ of $\Omega_L$-equivariant maps
from $\Gamma$ to $L^c$. It may be shown that if $\pi_1, \pi_2 \in
Z^1(\Omega_L,\Gamma)$, then $L_{\pi_1} \simeq L_{\pi_2}$ if and only
if $\pi_1$ and $\pi_2$ differ by a coboundary. The set of isomorphism
classes of $\Gamma$-torsors over $L$ may be identified with the
pointed cohomology set $H^1(L,\Gamma):=H^1(\Omega_L,\Gamma)$.  We
write $[\pi] \in H^1(L,\Gamma)$ for the class of $L_{\pi}$ in
$H^1(L,\Gamma)$. If $L$ is a number field or a non-archimedean local
field we write $H^1_t(L,\Gamma)$ for the subset of $H^1(L,\Gamma)$
consisting of those $[\pi] \in H^1(L,\Gamma)$ for which $L_{\pi}/L$ is
at most tamely ramified. If $L$ is an archimedean local field, we set
$H^1_t(L,G) = H^1(L, G)$. We denote the subset of $H^1_t(L,\Gamma)$
consisting of those $[\pi] \in H^1_t(L,\Gamma)$ for which $L_{\pi}/L$
is unramified at all (including infinite) places of $L$ by
$H^{1}_{nr}(L,\Gamma)$. (So, with this convention, if $L$ is an
archimedean local field, we have $H^{1}_{nr}(L, \Gamma) = 0$.) If $L$
is a number field, we write $H^{1}_{fnr}(F, \Gamma)$ for the subset of
$H^1_t(F, \Gamma)$ consisting of those $[\pi] \in H^1_t(F, \Gamma)$
for which $L_{\pi}/L$ is unramified at all finite places of $L$.
\smallskip

If $A$ is any algebra, we write $Z(A)$ for the centre of $A$. If $A$
is semisimple, we write
\[
\nrd: A^{\times} \to Z(A)^{\times},\quad \nrd: K_1(A) \to
Z(A)^{\times}
\]
for the reduced norm maps on $A^{\times}$ and $K_1(A)$ respectively
(cf. \cite[Chapter II, \S1]{Fr}). If $A$ is an $R$-algebra for some
ring $R$, and $R \to R_1$ is an extension of $R$, we write $A_{R_1}:=
A \otimes_{R} R_1$ to denote extension of scalars from $R$ to $R_1$.
\smallskip

If $S_1$ and $S_2$ are sets, we sometimes use the notation $S_1
\xrightarrow{\epi} S_2$ to denote a surjective map from $S_1$ to
$S_2$.


\section{Principal homogeneous spaces and resolvends} \label{S:galres}

In this section we shall describe some basic facts concerning
principal homogeneous spaces and resolvends. 

Throughout this section, the symbol $L$ denotes either a number field
or a local field.

\subsection{Principal homogeneous spaces} \label{SS:galal}
\cite[Section 1]{Mc1}, \cite[Section 1]{B}. 
Let $\Gamma$ be any finite group upon which $\Omega_L$ acts
continuously on the left, and write $Z^1(\Omega_L, \Gamma)$ for the
set of $\Gamma$-valued continuous $\Omega_L$ $1$-cocycles. If $\pi \in
Z^1(\Omega_L,\Gamma)$, then we write $^{\pi} \Gamma$ for the set
$\Gamma$ endowed with the following modified action of $\Omega_L$: if
\[
\Gamma \to {^{\pi} \Gamma}; \quad \gamma \mapsto \ov{\gamma}
\]
is the identity map on the underlying sets, then
\[
\ov{\gamma}^{\omega} = \ov{\pi(\omega) \cdot \gamma^{\omega}}
\]
for each $\gamma \in \Gamma$ and $\omega \in \Omega_L$. The group
$\Gamma$ acts on $^{\pi} \Gamma$ via right multiplication.

We define an associated $L$-algebra $L_\pi$ by
\[
L_\pi := \Map_{\Omega_L}(^\pi \Gamma, L^c);
\]
this is the algebra of $L^c$-valued functions on $^\pi \Gamma$ that
are fixed under the action of $\Omega_L$. The Hopf algebra
\[
A = A_L:=(L^c \Gamma)^{\Omega_L}
\]
acts on $L_\pi$ via the rule
\[
(\alpha \cdot a)(\gamma) = \sum_{g \in \Gamma} \alpha_g \cdot a (\gamma \cdot g)
\]
for all $\gamma \in \Gamma$ and $\alpha = \sum_{g \in \Gamma} \alpha_g
\cdot g \in A$. The algebra $L_\pi$ is a principal homogeneous space
(p.h.s. for short) of the Hopf algebra
\begin{equation} \label{E:Bdef}
B := \Map_{\Omega_L}(\Gamma, L^c).
\end{equation}
It may be shown that every p.h.s. of $B$ is isomorphic to an algebra
of the form $L_\pi$ for some $\pi$, and so every such p.h.s. may be
viewed as being a subset of the $L^c$-algebra $\Map(\Gamma, L^c)$. It
is easy to check that
\[
L_\pi \otimes_L L^c = L^c \Gamma \cdot \ell_{\Gamma},
\]
where $\ell_\Gamma \in \Map(\Gamma, L^c)$ is defined by
\[
\ell_{\Gamma}(\gamma) =
\begin{cases}
1   &\text{if $\gamma = 1$;} \\
0   &\text{otherwise.}
\end{cases}
\]

This implies that $L_\pi$ is a free, rank one $A$-module.

The Wedderburn decomposition of $L_\pi$ may be described as
follows. For any $\ov{\gamma} \in {^{\pi} \Gamma}$, write
$\Stab(\ov{\gamma})$ for the stabiliser of $\ov{\gamma}$ in
$\Omega_L$, and set 
\[
L(\ov{\gamma}):= (L^{c})^{\Stab(\ov{\gamma})}.
\]
Then
\[
L_\pi \simeq \prod_{\Omega_L \backslash ^{\pi} \Gamma} L(\ov{\gamma}),
\]
where $\Omega_L \backslash ^{\pi} \Gamma$ denotes the set of
$\Omega_L$-orbits of $^{\pi} \Gamma$, and the product is taken over a
set of orbit representatives. In general, the field $L(\ov{\gamma})$
is not normal over $L$. However, if $\Omega_L$ acts trivially on
$\Gamma$, then $Z^1(\Omega_L, \Gamma) = \Hom(\Omega_L, \Gamma)$, and
for each $\ov{\gamma} \in ^{\pi}\Gamma$, we have
\begin{equation}  \label{E:weddcomp}
L(\ov{\gamma}) = (L^c)^{\Ker(\pi)}=: L^{\pi},
\end{equation}
with $\Gal(L^{\pi}/L) \simeq \pi(\Omega_L)$. In this case, we have
that
\begin{equation} \label{E:weddiso}
L_\pi \simeq \prod_{\Gamma /\pi(\Omega_L)} L^{\pi},
\end{equation}
and this isomorphism depends only upon the choice of a transversal of
$\pi(\Omega_L)$ in $\Gamma$. 

\begin{remark} \label{R:galalg}
For most of this paper we shall only need to consider the case in
which $\Omega_L$ acts trivially on $\Gamma$; in this situation $A =
L\Gamma$, and $L_{\pi}$ is a $\Gamma$-Galois $L$-algebra. A notable
exception to this will occur in Section \ref{S:local}, when we take
$L$ to be a non-archimedean local field, and we construct a canonical
subextension of a tame extension $L_{\pi}/L$ (see Definitions
\ref{D:gnr} and \ref{D:gnr1}). This canonical sub-extension is
complementary to the maximal unramified sub-extension of $L_{\pi}/L$,
and is not usually a Galois algebra extension of $L$. It is however, a
p.h.s. of a Hopf algebra of the form \eqref{E:Bdef} associated to a
certain group $\Gamma$ equipped (as a set) with a non-trivial
$\Omega_L$-action.  \qed
\end{remark}

\subsection{Resolvends} \label{SS:resolvends}
\cite[Section 1]{Mc1}, \cite[Section 2]{B}. 

Since every p.h.s. of $B$ may be viewed as being a subset of
$\Map(\Gamma, L^c)$, it is natural to consider the Fourier transforms
of elements of $\Map(\Gamma, L^c)$. These arise via the
\textit{resolvend map}
\[
\br_\Gamma:  \Map(\Gamma, L^c) \to L^c \Gamma; \qquad a \mapsto
\sum_{s \in \Gamma} a(s) s^{-1}.
\]
The map $\br_\Gamma$ is an isomorphism of left $L^c\Gamma$-modules,
but not of algebras, because it does not preserve multiplication. It
is easy to show that for any $a \in \Map(\Gamma, L^c)$, we have that
$a \in L_\pi$ if and only if $\br_\Gamma(a)^{\omega} = \br_\Gamma(a)
\cdot \pi(\omega)$ for all $\omega \in \Omega_L$. It may also be shown
that an element $a \in L_\pi$ generates $L_\pi$ as an $A$-module if
and only if $\br_\Gamma(a) \in (L^c \Gamma)^{\times}$. Two elements
$a_1, a_2 \in \Map(\Gamma, L^c)$ with $\br_\Gamma(a_1),
\br_\Gamma(a_2) \in (L^c\Gamma)^{\times}$ generate the same p.h.s.  as
an $A$-module if and only if $\br_\Gamma(a_1) = b \cdot
\br_\Gamma(a_2)$ for some $b \in A^{\times}$. If $a$ is any generator
of $L_\pi$ as an $A$-module, then a $\Gamma$-valued $\Omega_L$
$1$-cocycle that represents the class $[\pi]$ of $\pi$ in the pointed
cohomology set $H^1(L,\Gamma)$ is given by
\[
\omega \mapsto \br_{\Gamma}(a)^{-1} \cdot \br_{\Gamma}(a)^{\omega}.
\]

We define pointed sets (where in each case the distinguished element
is afforded by $1 \in A^{\times}_{L^c} = (L^c\Gamma)^{\times}$):
\begin{align*}  
&H(A):= \left\{ \alpha \in A^{\times}_{L^c}:
  \alpha^{-1} \cdot \alpha^{\omega} \in \Gamma \quad \forall \omega \in \Omega_L
  \right\}; \\
&\cH(A):= H(A)/\Gamma = \{ \alpha \cdot \Gamma: \alpha \in H(A) \},
\end{align*}
and we write $r_\Gamma(a) \in \cH(A)$ for the image in $\cH(A)$ of
$\br_\Gamma(a) \in H(A)$. The element $r_\Gamma(a)$ is referred to as
the \textit{reduced resolvend} of $a$. If $\fA$ is any $O_L$-order in
$A$, then we define $H(\fA)$ and $\cH(\fA)$ in a similar manner. Hence
we have
\[
H(\fA) = \fA_{O_{L^c}} \cap H(A), \quad \cH(\fA) = H(\fA)/\Gamma.
\]

Write $L^t$ for the maximal, tamely ramified extension of $L$.  We
set
\begin{align*}
&H_t(A):= \left\{ \alpha \in H(A):
  \alpha^{\omega} = \alpha \quad \forall \omega \in \Omega_{L^t}
  \right\}; \\
&\cH_t(A):= H_t(A)/\Gamma = \{ \alpha \cdot \Gamma: \alpha
  \in H_t(A) \}, 
\end{align*}
and we define $H_t(\fA)$ and $\cH_t(\fA)$ analogously for any
$O_{L}$-order $\fA$ in $A$.
\medskip

We shall now give a characterisation of the set $H(A)$  that avoids any
explicit mention of Galois action. This is a non-abelian version of a
description of $H(A)$ in terms of primitive elements of quotients of
groups of units in Hopf algebras in the abelian case (see
\cite[Theorem 6.4]{AB}).

In order to do this, we first note that there are
$\Omega_L$-equivariant homomorphisms of algebras
\[
\Delta, i_1, i_2: A_{L^c} \to A_{L^c} \otimes_{L^c} A_{L^c}
\]
induced by the maps
\[
\Delta(\gamma) = \gamma \otimes \gamma, \quad i_1(\gamma) = \gamma
\otimes 1, \quad i_2(\gamma) = 1 
\otimes \gamma
\]
for $\gamma \in \Gamma$.

We define a map of pointed sets
\[
\sP: A_{L^c}^{\times} \to (A_{L^c} \otimes_{L^c}
A_{L^c})^{\times}; \quad x \mapsto 
\Delta(x) \cdot [i_1(x) \cdot i_2(x)]^{-1}.
\]

It is easy to verify that
\[
\sP(x_1 \cdot x_2) = \Delta(x_1) \cdot \sP(x_2) \cdot [i_1(x_1) \cdot
  i_2(x_1)]^{-1}.
\]
As $\sP(\gamma) = 1$ for each $\gamma \in \Gamma$, it follows that
$\sP$ induces a map of pointed sets (which we denote by the same
symbol):
\[
\sP: A_{L^c}^{\times}/\Gamma \to (A_{L^c} \otimes_{L^c}
A_{L^c})^{\times}. 
\]

\begin{theorem} \label{T:prim}
Let $x \in A_{L^c}^{\times}$. Then $x \in H(A)$ if and only if
$\sP(x) \in (A \otimes_L A)^{\times}$.
\end{theorem}

\begin{proof}
Suppose that $x \in H(A)$. Then if $\omega \in \Omega_L$, we have
\[
x^{\omega} = x \cdot \gamma_{\omega}
\]
for some $\gamma_{\omega} \in \Gamma$. Hence 
\begin{align*}
[\Delta(x)(i_1(x)i_2(x))^{-1}]^{\omega} &= \Delta(x)(\gamma_{\omega}
  \otimes \gamma_{\omega}) [i_1(x) (\gamma_{\omega} \otimes 1) i_2(x) (1 \otimes
    \gamma_{\omega})]^{-1} \\
&= \Delta(x) (\gamma_{\omega} \otimes \gamma_{\omega})(1 \otimes
  \gamma_{\omega})^{-1} i_2(x)^{-1} (\gamma_{\omega} \otimes 1)^{-1} i_1(x)^{-1}
  \\
&=\Delta(x)(\gamma_{\omega} \otimes \gamma_{\omega})(1 \otimes
  \gamma_{\omega})^{-1}(\gamma_{\omega} \otimes 1)^{-1} i_2(x)^{-1} i_1(x)^{-1}
  \\
&= \Delta(x)[i_1(x)i_2(x)]^{-1}.
\end{align*}
This shows that
\[
\sP(x) \in [(A_{L^c} \otimes_{L^c}
  A_{L^c})^{\times}]^{\Omega_L} = (A \otimes_L A)^{\times}.
\]

Suppose conversely that $\sP(x) \in (A \otimes_L A)^{\times}$, and
that $x^{\omega} = x \cdot u_{\omega}$ for each $\omega \in \Omega_L$.
We wish to show that $u_{\omega} \in \Gamma$.  As the maps $\Delta$,
$i_1$, and $i_2$ are $\Omega_L$-equivariant, we have that
\[
\Delta(x)^{\omega} = \Delta(x) \cdot \Delta(u_{\omega}),\quad
i_1(x)^{\omega} = i_1(x) \cdot i_1(u_{\omega}), \quad i_2(x)^{\omega}
= i_2(x) \cdot i_2(u_{\omega}),
\]
and a straightforward computation shows that
\[
\sP(x)^{\omega} = \Delta(x) \cdot \cP(u_{\omega}) \cdot [i_1(x) \cdot
  i_2(x)]^{-1}.
\]
As $\sP(x) = \sP(x)^{\omega}$, this implies that $\sP(u_{\omega}) =
1$, i.e. that
\[
\Delta(u_{\omega}) = i_1(u_{\omega}) \cdot i_2(u_{\omega}).
\]
It now follows that $u_{\omega} \in \Gamma$ via an argument identical to
that given in \cite[Theorem 6.4]{AB}.
\end{proof}

Let $F$ be a number field. Our next result shows that the pointed set
$H(A_F)$ of resolvends satisfies a Hasse principle.

\begin{proposition}  \label{P:reshasse}
Let $F$ be a number field, and suppose that $x \in
(F^c\Gamma)^{\times}$. Then $x \in H(A_F)$ if and only if $\loc_v(x)
\in H(A_{F_v})$ for every finite place $v$ of $F$.
\end{proposition}

\begin{proof}
We first observe that the map $\sP$ commutes with localisation,
i.e. for each finite place $v$ of $F$, we have
\begin{equation} \label{E:primloc}
\loc_v(\sP(x)) = \sP(\loc_v(x))
\end{equation}
for all $x \in (F^c\Gamma)^{\times}$. Hence we have
\begin{align*}
x \in H(A_F) &\iff \sP(x) \in (A_F \otimes_F A_F)^{\times} \quad \text{(from
Theorem \ref{T:prim});} \\
&\iff \loc_v(\sP(x)) \in (A_{F_v} \otimes_{F_v} A_{F_v})^{\times} \quad
\text{for each finite $v$;} \\
&\iff \sP(\loc_v(x)) \in (A_{F_v} \otimes_{F_v} A_{F_v})^{\times} \quad
\text{for each finite $v$ (from \eqref{E:primloc});} \\
&\iff \loc_v(x) \in H(A_{F_v}) \quad \text{for each finite $v$ (from Theorem
  \ref{T:prim})}.
\end{align*}

\end{proof}

\begin{remark}  \label{R:altreshasse}
It is also possible to give a proof of Proposition \ref{P:reshasse}
directly from the definition of $H(A_F)$. The standard such proof
that was known to the authors is valid only for abelian groups
$\Gamma$; we are grateful to an anonymous referee for explaining how
this proof may be modified so as to hold for arbitrary finite groups.

Suppose that $x \in A_{F^c}^{\times}$ is such that, for each finite
place $v$ of $F$, we have $\loc_v(x) \in H(A_{F_v})$.  We wish to
show that $x \in H(A_F)$.

Let $E/F$ be any finite Galois extension such that $\Omega_E$ fixes
$x$. Then the action of $\Omega_F$ on $x$ factors through the action
of the finite group $D:= \Gal(E/F)$. Hence, to prove the desired
result, it suffices to show that for any $\delta \in D$, we have
$x^{\delta} = x \cdot \gamma_{\delta}$, with $\gamma_{\delta} \in
\Gamma$. 

Let $\cG_F$ denote the subgroup of $\Omega_F$ generated by the
subgroups $\Omega_{F_v}$ as $v$ runs over the finite places of $F$.
As each element of $\Omega_F$ is conjugate to an element of
$\Omega_{F_v}$ for some $v$, it follows via the Chebotarev density
theorem that the image $\ov{\cG}_F$ of $\cG_F$ in $D$ has non-trivial
intersection with every conjugacy class of $D$.  A lemma of Jordan now
implies that $\ov{\cG}_F$ must be equal to the whole of $D$ (see
\cite[p. 435, Theorem 4']{Se2003}).  The result we seek now follows at
once. \qed
\end{remark}


\section{Resolvends and cohomology} \label{S:resco}

Recall that $F$ is a number field and $G$ is a finite group upon which
$\Omega_F$ acts trivially. In this section, we explain, following
\cite[\S2]{Mc1}, how resolvends may be used to compute discriminants
of rings of integers of $G$-Galois extensions of $F$, and to describe
certain Galois cohomology groups.

For each $[\pi] \in H^1(F,G)$, the standard trace map
\[
\Tr: \Map(G,F^c) \to F^c
\]
induces a trace map 
\[
\Tr: F_{\pi} \to F
\]
via restriction. This in turn yields an associated, non-degenerate
bilinear form $(a,b) \mapsto \Tr(ab)$ on $F_\pi$. If $M$ is any full
$O_F$-lattice in $F_\pi$, then we set
\[
M^*:= \{ b \in F_\pi | \Tr(b \cdot M) \subseteq O_F \}
\]
and
\[
\disc(O_\pi /O_F) := [ O_{\pi}^{*}: O_{\pi}]_{O_F},
\]
where the symbol $[-:-]_{O_F}$ denotes the $O_F$-module index. We see
from the isomorphism \eqref{E:weddiso} that we have
\[
\disc(O_\pi/O_F) = \disc(O_{F^{\pi}}/O_F)^{[G:\pi(\Omega_{F})]},
\]
where $\disc(O_{F^{\pi}}/O_F)$ denotes the usual discriminant of the
  number field $F^{\pi}$ over $F$, and so it follows that 
\[
\disc(O_{\pi}/O_F) = O_F
\]
if and only if $F_\pi/F$ is unramified at all finite places of $F$.

\begin{definition} We write $[-1]$ for the maps induced on
  $\Map(G,F^c)$ and $F^cG$ by the map $g \mapsto g^{-1}$ on $G$. \qed
\end{definition}

\begin{lemma} \label{L:keycalc}
Suppose that $a, b \in F_\pi$ for some $[\pi] \in H^1(F,G)$. Then
\[
\br_G(a) \cdot \br_G(b)^{[-1]} = \sum_{s \in G} \Tr(a^s b) \cdot
s^{-1} \in FG.
\]
\end{lemma}

\begin{proof} This may be verified via a straightforward calculation
  (see e.g.  \cite[(1.6)]{Mc}, and note that the calculation given there is
  valid for an arbitrary finite group G).
\end{proof}

\begin{corollary} \label{C:dual}
Suppose that $F_{\pi} = FG \cdot a$. Then we have:

(i) $\br_G(a)^{-1} = \br_G(b)^{[-1]}$, where $b \in F_{\pi}$
satisfies $\Tr(a^s b^t) = \delta_{s,t}$.

(ii) $(O_FG \cdot a)^* = O_FG \cdot b$.

(iii) $[(O_FG \cdot a)^*: O_FG \cdot a]_{O_F} = [O_FG: O_FG \cdot
  \br_G(a) \cdot \br_G(a)^{[-1]}]_{O_F}$.

(iv) $\br_G(a) \in (O_{F^c}G)^{\times}$ if and only if $O_{\pi} = O_FG
\cdot a$ and $\disc(O_{\pi}/O_F) = O_F$.

Analogous results hold if $F$ is replaced by $F_v$ for any finite
place $v$ of $F$.
\end{corollary}

\begin{proof} Exactly as in \cite[2.10 and 2.11]{Mc1}.
\end{proof}

\begin{lemma}  \label{L:H90}
Suppose that $L$ is either a number field or a local field. Then

(i) $H^1(L, (L^cG)^{\times}) = 1$;

(ii) $H^1(L, Z(L^cG)^{\times}) = 1$.

\end{lemma}

\begin{proof}
For each $\chi \in \Irr(G)$, write $d(\chi)$ for the degree of
$\chi$, and $M_{d(\chi)}(L^c)$ for the algebra of $d(\chi) \times
d(\chi)$-matrices over $L^c$. Then the Wedderburn isomorphism of
algebras
\[
L^cG \simeq \bigoplus_{\chi \in \Irr(G)} M_{d(\chi)}(L^c)
\]
yields isomorphisms of groups
\[
(L^cG)^{\times} \simeq \bigoplus_{\chi \in \Irr(G)} \GL_{d(\chi)}(L^c),
\quad Z(L^cG)^{\times} \simeq \bigoplus_{\chi \in \Irr(G)} (L^c)^{\times}.
\]
Let $\chi_1,\ldots, \chi_m \in \Irr(G)$ be a set of representatives of
$\Omega_L \backslash \Irr(G)$. Write $\Stab(\chi_i)$ for the
stabiliser of $\chi_i$ in $\Omega_L$, and set $L[\chi_i]:=
(L^c)^{\Stab(\chi_i)}$. There are isomorphisms of
$\Omega_L$-modules 
\[
(L^cG)^{\times} \simeq \bigoplus_{i=1}^{m}
\Ind_{\Omega_{L[\chi_i]}}^{\Omega_{L}} (\GL_{d(\chi_i)}(L^c)), \quad
Z(L^cG)^{\times} \simeq \bigoplus_{i=1}^{m}
\Ind_{\Omega_{L[\chi_i]}}^{\Omega_{L}} (L^c)^{\times}.
\]
We have
\begin{align*}
H^1(L, (L^cG)^{\times}) &\simeq H^1(L, \bigoplus_{i=1}^{m}
\Ind_{\Omega_{L[\chi_i]}}^{\Omega_{L}} \GL_{d(\chi_i)}(L^c)) \\
&\simeq \bigoplus_{i=1}^{m} H^1(L[\chi_i], \GL_{d(\chi_i)}(L^c)) \\
&= 1,
\end{align*}
where the second isomorphism follows via Shapiro's Lemma and the third
is standard consequence of Hilbert's Theorem 90. This proves (i). The
proof of (ii) is very similar.
\end{proof}

Recall that two pointed sets $S_1$ and $S_2$ are said to be \textit{isomorphic}
if there is a bijection of sets
\[
f: S_1 \to S_2
\]
with $f(x_1) = f(x_2)$, where $x_i$ is the distinguished element of
$S_i$, $(i=1,2)$.

A sequence
\[
\cdots \rightarrow S_{i-1} \xrightarrow{f_{i}} S_i \xrightarrow{f_{i+1}}
S_{i+1} \rightarrow \cdots
\]
of pointed sets is said to be \textit{exact} if there is an equality of sets
\[
\Image(f_{i}) = f_{i+1}^{-1}(x_{i+1}),
\]
where $x_{i+1}$ is the distinguished element of $S_{i+1}$.

\begin{theorem} \label{T:reskey} 
(a) There is an exact sequence of pointed sets
\begin{equation} \label{E:res1}
1 \to G \to (FG)^{\times} \to \cH(FG) \to H^1(F,G) \to 1.
\end{equation}

(b) For each finite place $v$ of $F$, recall that $H^{1}_{nr}(F_v,G)$
denotes the subset of $H^1(F_v,G)$ consisting of those $[\pi_v] \in
H^1(F_v,G)$ for which the associated $G$-Galois extension
$F_{\pi_v}/F_v$ is unramified. Then there is an exact sequence of
pointed sets
\begin{equation} \label{E:res2}
1 \to G \to (O_{F_v}G)^{\times} \to \cH(O_{F_v}G) \to
H^{1}_{nr}(F_v,G) \to 1.
\end{equation}

(c) There are exact sequences of pointed
sets
\begin{equation} \label{E:res3}
1 \to G \to (FG)^{\times} \to \cH_t(FG) \to H_t^1(F,G) \to 1,
\end{equation}
and
\begin{equation} \label{E:res4}
1 \to G \to (F_vG)^{\times} \to \cH_t(F_vG) \to H_t^1(F_v,G) \to 1
\end{equation}
for each place $v$ of $F$.
\end{theorem}

\begin{proof} When $G$ is abelian, parts (a) and (b) are proved in
\cite[pages 268 and 273]{Mc1} by considering the $\Omega_F$ and
$\Omega_{F_v}$-cohomology of the exact sequences of abelian groups
\begin{equation} \label{E:tors1}
1 \to G \to (F^cG)^{\times} \to (F^cG)^{\times}/G \to 1
\end{equation}
and 
\[
1 \to G \to (O_{F^c_v}G)^{\times} \to (O_{F^c_v}G)^{\times}/G \to 1
\]
respectively.  If $G$ is non-abelian, and these exact sequences are
viewed as exact sequences of pointed sets instead, then a similar
proof of part (a) also holds, as is pointed out in \cite[page
  268]{Mc1}: taking $\Omega_F$-cohomology of the exact sequence
\eqref{E:tors1} of pointed sets yields an exact sequence
\begin{equation} \label{E:tors2}
1 \to G \to (FG)^{\times} \to \cH(FG) \to H^1(F,G) \to H^1(F, (F^cG)^{\times}),
\end{equation}
and since $H^1(F, (F^cG)^{\times}) =1$ (see Lemma \ref{L:H90}(i)), 
\eqref{E:res1} immediately follows.

Alternatively, we could also argue directly (as is done in \cite{Mc1})
that the map $\cH(FG) \to H^1(F,G)$ in \eqref{E:tors2} is
surjective. Let us briefly describe the argument given in
\cite{Mc1}. Suppose that $[\pi] \in H^1(F,G)$, and let $a \in F_{\pi}$
be a normal basis generator of $F_\pi/F$. Set $\alpha = \br_G(a)$;
then the coset $\alpha \cdot G \in \cH(FG)$ lies in the pre-image of
$[\pi]$, and so it follows that \eqref{E:tors2} is indeed surjective
on the right, as claimed.

Part (b) follows from Corollary \ref{C:dual}(iv) (cf. the
proof of (2.12) on \cite[page 273]{Mc1}). 

The proof of (c) is very similar to that of (a). Let $F^t$ and $F_v^t$
denote the maximal tamely ramified extensions of $F$ and $F_v$
respectively, and set $\Omega_F^t:= \Gal(F^t/F)$, $\Omega_{F_v}^{t}:=
\Gal(F_v^t/F_v)$. Then (c) follows via considering the $\Omega_F^t$
and $\Omega_{F_v}^{t}$-cohomology of the exact sequences of pointed
sets
\[
1 \to G \to (F^tG)^{\times} \to (F^tG)^{\times}/G \to 1
\]
and 
\[
1 \to G \to (F_v^tG)^{\times} \to (F_v^tG)^{\times}/G \to 1
\]
respectively, using the direct argument given in \cite[page 268]{Mc1}
that we have described above.
\end{proof}

Suppose that $L$ is a number field or a local field. Recall that
$Z(LG)$ denotes the centre of $LG$. Before stating
our next result, we note that the reduced norm map
\[
\nrd: (LG)^{\times} \to Z(LG)^{\times}
\]
induces an injection $G^{ab} \to Z(LG)^{\times}$. (More explicitly, if
we identify $Z(L^cG)^{\times}$ with $\prod_{\chi \in \Irr(G)}
(L^{c})^{\times}$ via the Wedderburn decomposition of $L^cG$ (cf. the
proof of Lemma \ref{L:H90}), then the injection $G^{ab} \to
Z(L^G)^{\times}$ is induced by the map $G \to Z(L^cG)^{\times}$ given
by $g \mapsto [(\det(\chi))(g)]_{\chi}$, where $\det(\chi)$ is the abelian
character of $G$ defined below in Definition \ref{D:abdet}. See also
\eqref{E:globrag}.) In what follows, we shall identify $G^{ab}$ with
its image in $Z(LG)^{\times}$ under this map. We set
\begin{align*}
&H(Z(LG)):= \left\{ \alpha \in Z(L^cG)^{\times}:
  \alpha^{-1} \cdot \alpha^{\omega} \in G^{ab} \quad \forall \omega \in \Omega_L
  \right\}; \\
&\cH(Z(LG)):= H(Z(LG))/ G^{ab} = 
\{ \alpha \cdot G^{ab}: \alpha \in H(Z(LG)) \}.
\end{align*}
We define $H(Z(\fA))$ and $\cH(Z(\fA))$ analogously for any
$O_L$-order $\fA$ in $LG$.

\begin{proposition} \label{P:abres}
Let $L$ be a number field or a local field. Then there is an exact
sequence of abelian groups:
\begin{equation} \label{E:res5}
1 \to G^{ab} \to Z(LG)^{\times} \to \cH(Z(LG)) \to H^1(L, G^{ab}) \to 1.
\end{equation}

\end{proposition}

\begin{proof} This follows at once from taking $\Omega_L$ 
 cohomology of the exact sequence of abelian groups
\[
1 \to G^{ab} \to Z(L^cG)^{\times} \to Z(L^cG)^{\times}/G^{ab}
\to 1,
\]
arising from the injection $G^{ab} \to Z(L^{c}G)^{\times}$ induced by
the reduced norm map $\nrd: (LG)^{\times} \to Z(LG)^{\times}$ as
described above, and noting that $H^1(\Omega_L, Z(L^cG)^{\times}) =
1$, via Lemma \ref{L:H90}(ii).
\end{proof}

It is easy to see that the group $(LG)^{\times}$ acts on the pointed
set $\cH(LG)$ by left multiplication. Write $(LG)^{\times} \backslash
\cH(LG)$ for the quotient set afforded by this action.  It follows
from Theorem \ref{T:reskey} and Proposition \ref{P:abres} that there
are isomorphisms
\[
H^1(L, G) \xrightarrow{\sim} (LG)^{\times} \backslash \cH(LG)
\]
and
\[
H^1(L, G^{ab}) \xrightarrow{\sim} Z(LG)^{\times} \backslash
\cH(Z(LG))
\]
of pointed sets and abelian groups respectively, and that the
following diagram commutes:
\begin{equation} \label{E:resdiag}
\begin{CD}
H^1(L,G) @>{\sim}>>  (LG)^{\times} \backslash \cH(LG) \\
@VVV                                   @VV{\nrd}V \\
 H^1(L, G^{ab}) @>{\sim}>>       Z(LG)^{\times} \backslash
 \cH(Z(LG)).
\end{CD}
\end{equation}
(Here the left-hand vertical arrow is induced by the quotient
map $G \to G^{ab}$, while the right-hand vertical arrow is induced by
the reduced norm map $\nrd: (L^cG)^{\times} \to
Z(L^cG)^{\times}$. )

We shall need the following result in Section \ref{S:cc}.

\begin{proposition} \label{P:group}
Let $F$ be a number field. For each finite place $v$ of $F$, the image
of the map
\[
\nrd: (O_{F_v}G)^{\times} \backslash \cH(O_{F_v}G) \to
Z(O_{F_v}G)^{\times} \backslash \cH(Z(O_{F_v}G))
\]
of pointed sets is in fact a group.
\end{proposition}

\begin{proof} Just as in the case of \eqref{E:resdiag}, we see from
  the exact sequences \eqref{E:res2} and \eqref{E:res5} that there is
  a commutative diagram
\begin{equation} \label{E:unramresdiag}
\begin{CD}
H_{nr}^{1}(F_v,G) @>{\sim}>>  (O_{F_v}G)^{\times} \backslash \cH(O_{F_v}G) \\
@VVV                                   @VV{\nrd}V \\
 H_{nr}^{1}(F_v, G^{ab}) @>>>       Z(O_{F_v}G)^{\times} \backslash
 \cH(Z(O_{F_v}G)) \\
@VV{\cap}V                          @VV{\cap}V \\
H^{1}(F_v, G^{ab})  @>{\sim}>> Z(F_vG)^{\times} \backslash \cH(Z(F_vG)).
\end{CD}
\end{equation}

The middle horizontal arrow of \eqref{E:unramresdiag} is therefore
injective, and its image is a subgroup of $Z(O_{F_v}G)^{\times}
\backslash \cH(Z(O_{F_v}G))$.  Hence, to prove the desired result, it
suffices to show that the map $H_{nr}^{1}(F_v, G) \to H_{nr}^{1}(F_v,
G^{ab})$ is surjective. This is in turn an immediate consequence of
the fact that the Galois group $\Gal(F^{nr}_{v}/F_v)$ is profinite
free on a single generator.
\end{proof}


\section{Determinants and character maps} \label{S:det}

In this section we shall describe how determinants of resolvends
may be represented in terms of certain character maps.

Let $L$ be a number field or a local field. 

Suppose that $\Gamma$ is any finite group upon which the absolute
Galois group $\Omega_L$ of $L$ acts (possibly trivially). Then
$\Omega_L$ also acts on the ring $R_{\Gamma}$ of virtual characters of
$\Gamma$ according to the following rule: if $\chi \in \Irr(\Gamma)$
and $\omega \in \Omega_L$, then, for each $\gamma \in \Gamma$, we have
$\chi^{\omega}(\gamma) = \omega(\chi(\omega^{-1}(\gamma)))$.

We begin by recalling some well-known facts and definitions concerning
determinant maps (see e.g. \cite[Chapter II]{Fr} or \cite[Chapter
  I]{Fr1}).

\begin{definition} \label{D:Det}
For each element $a$ of $\GL_n(L^cG)$, we define an element 
\begin{equation} \label{E:Detdef}
\Det(a)
\in \Hom(R_G, (L^{c})^{\times}) \simeq Z(L^{c}G)^{\times}
\end{equation}
in the following way: if $T$ is any representation of $G$ over $L^c$
 with character $\phi$, then we set
\[
\Det(a)(\phi):= \det(T(a)).
\]
It may be shown that this definition depends only upon the character
$\phi$, and not upon the choice of representation $T$. The map
\[
\Det: \GL_{n}(L^cG) \to \Hom(R_{G},(L^{c})^{\times})
\]
is $\Omega_L$-equivariant, and so induces a map
\[
\Det: \GL_n(LG) \to \Hom_{\Omega_{L}}(R_{G}, (L^c)^{\times}).
\]
\qed
\end{definition}

\begin{remark} \label{R:nrd}
The map $\Det$ in \eqref{E:Detdef} above is essentially the same as
the reduced norm map.  Let
\begin{equation} \label{E:nrd}
\nrd: (L^cG)^{\times} \to Z(L^cG)^{\times}
\end{equation}
denote the reduced norm. Then \eqref{E:nrd} induces an isomorphism
\begin{equation} \label{E:nrdiso}
\nrd: K_1(L^cG) \xrightarrow{\sim} Z(L^cG)^{\times} \simeq \Hom(R_G,
(L^c)^{\times})
\end{equation}
(see e.g. \cite[Theorem 45.3]{CR2}). Suppose now that $\phi$ is any
$L^c$-valued character of $G$, and let $a \in (L^cG)^{\times}$. Then
we have that
\[
\Det(a)(\phi) = \nrd(a)(\phi)
\]
(see \cite[Chapter I, Proposition 2.7]{Fr1}). \qed
\end{remark}

\begin{definition} \label{D:abdet}
Suppose that $\chi \in \Irr(G)$. We define an abelian character
$\det(\chi)$ of $G$ as follows. Let $T$ be any representation of 
$G$ over $L^c$ affording $\chi$. For each element $g \in G$, we
set
\[
(\det(\chi))(g) = \Det(T(g)).
\]
Then $\det(\chi)$ is independent of the choice of $T$, and may
be viewed as being a character of $G^{ab}$. We extend $\det$
to a homomorphism $R_G \to (G^{ab})^{\wedge}$, where $(G^{ab})^{\wedge}$
denotes the group of characters of $G^{\ab}$, by defining
\[
\det \left( \sum_{\chi \in \Irr(G)} a_\chi \chi \right)  = \prod_{\chi \in 
\Irr(G)}
(\det(\chi))^{a_\chi},
\]
and we set 
\[
A_G:= \Ker(\det).
\]
Hence we have an exact sequence of groups
\begin{equation} \label{E:aug}
0 \to A_G \to R_G \xrightarrow{\det} (G^{ab})^{\wedge} \to 0.
\end{equation}
\qed
\end{definition}

Applying the functor $\Hom(-, (L^c)^{\times})$ 
to \eqref{E:aug}, we
obtain an exact sequence
\begin{equation} \label{E:globrag}
0 \to G^{ab} \to \Hom(R_G, (L^{c})^{\times}) \xrightarrow{\rag}
\Hom(A_G, (L^c)^{\times}) \to 0,
\end{equation}
which is surjective on the right because $(L^c)^{\times}$ is
divisible. It follows that there are $\Omega_L$-equivariant
isomorphisms
\begin{equation} \label{E:zres}
\Hom(A_G, (L^c)^{\times}) \simeq \Hom(R_G, (L^{c})^{\times})/G^{ab}
\simeq Z(L^cG)^{\times}/G^{ab}.
\end{equation}
In what follows, we shall sometimes identify $\Hom(A_G,
(L^c)^{\times})$ with $Z(L^cG)^{\times}/G^{ab}$ via \eqref{E:zres}
without explicit mention.

Taking $\Omega_L$-cohomology of \eqref{E:globrag}
yields an  exact sequence
\begin{equation} \label{E:globcoh}
0 \to G^{ab} \to \Hom_{\Omega_L}(R_G, (L^c)^{\times}) \xrightarrow{\rag}
\Hom_{\Omega_L}(A_G, (L^c)^{\times}) \to H^1(L, G^{ab}) \to 1,
\end{equation}
which is surjective on the right via Lemma \ref{L:H90}(ii).

\begin{definition} \label{D:symp}
Let $R^s_G$ denote the (additive) subgroup of $R_G$ generated by the
symplectic characters of $G$. Thus, $R^s_G$ is generated by the
irreducible symplectic characters of $G$, together with elements of
the form $\chi + \ov{\chi}$, where $\chi \in R_G$ and $\ov{\chi}$
denotes the complex conjugate of $\chi$. All virtual characters lying
in $R^s_G$ are real-valued.

If $F$ is a number field, and $v$ is a real place of $F$, we write
$\Hom^{+}_{\Omega_{F_v}}(R_G, (F_{v}^{c})^{\times})$ for those
elements $f \in \Hom_{\Omega_{F_v}}(R_G, (F_{v}^{c})^{\times})$ for
which $f(\eta)>0$ for all $\eta \in R^s_G$. Note that if $f \in
\Hom_{\Omega_{F_v}}(R_G, (F_{v}^{c})^{\times})$ and $\chi \in R_G$,
then we automatically have
\[
f(\chi + \ov{\chi}) = f(\chi) \cdot \ov{f(\chi)} > 0.
\]
Hence in fact $f \in \Hom^{+}_{\Omega_{F_v}}(R_G,
(F_{v}^{c})^{\times})$ if and only if $f$ is positive on all
irreducible, symplectic characters of $G$. In particular, if $G$ has
no non-trivial irreducible symplectic characters (e.g. if $|G|$ is
odd), then we have
\[
\Hom^{+}_{\Omega_{F_v}}(R_G, (F_{v}^{c})^{\times}) =
\Hom_{\Omega_{F_v}}(R_G, (F_{v}^{c})^{\times}).
\]

We write $Z(F_vG)^{\times}_{+}$ for the image of
$\Hom^{+}_{\Omega_{F_v}}(R_G, (F_{v}^{c})^{\times})$ in
$Z(F_vG)^{\times}$ under the isomorphism
\[
\Hom_{\Omega_{F_v}}(R_G, (F_{v}^{c})^{\times}) \xrightarrow{\sim}
Z(F_vG)^{\times}.
\]

\qed
\end{definition}

\begin{proposition}  \label{P:agdet}
Let $F$ be a number field. For each place $v$ of $F$, we write
\begin{equation} \label{E:locdet}
\Det: (F^{c}_{v}G)^{\times} \to \Hom(R_G, (F^{c}_{v})^{\times}) \simeq
Z(F_v^cG)^{\times}
\end{equation}
for the determinant homomorphism afforded by Definition \ref{D:Det}.

(i) If $v$ is real, then \eqref{E:locdet} induces an isomorphism
\begin{equation} \label{E:rvdet}
\Det((F_vG)^{\times}) \simeq \Hom^{+}_{\Omega_{F_v}}(R_G,
(F_{v}^{c})^{\times}) \simeq Z(F_vG)^{\times}_{+}.
\end{equation}

(ii) If $v$ is finite or complex, then the map \eqref{E:locdet}
induces isomorphisms
\begin{align}
\Det((F_vG)^{\times}) &\simeq \Hom_{\Omega_{F_v}}(R_G,
(F^c_v)^{\times}) \simeq Z(F_vG)^{\times},  \label{E:vdeti}\\
\Det(\cH(F_vG)) &\simeq \Hom_{\Omega_{F_v}}(A_G, (F^c_v)^{\times}).  \label{E:vdetii}
\end{align}

(iii) If $v$ is finite of residue characteristic coprime to $|G|$, so
$O_{F_v}G$ is an $O_{F_v}$-maximal order in $F_vG$, then
\eqref{E:locdet} induces isomorphisms
\begin{align}
\Det((O_{F_v}G)^{\times}) &\simeq \Hom_{\Omega_{F_v}}(R_G,
(O_{F^c_v})^{\times}) \simeq Z(O_{F_v}G)^{\times}, \label{E:vdetiii}\\
\Det(\cH(O_{F_v}G)) &\simeq \Hom_{\Omega_{F_v}}(A_G,
(O_{F^c_v})^{\times}). \label{E:vdetiv}
\end{align}
\end{proposition}

\begin{proof}
The isomorphisms \eqref{E:rvdet}, \eqref{E:vdeti} and
\eqref{E:vdetiii} are standard and are explained in e.g. \cite[Chapter
  II, \S1]{Fr}. 

Suppose that $v$ is either finite or complex. Theorem
\ref{T:reskey}(a) and \eqref{E:vdeti} yield the following
commutative diagram:
\begin{equation} \label{E:ragcdi}
\begin{CD}
G @>{\subseteq}>> (F_vG)^{\times} @>>> \cH(F_vG) @>{\epi}>> H^1(F_v, G) 
\\
@VVV    @VV{\Det}V   @VV{\Det}V @VV{\epi}V  \\
G^{ab} @>{\subseteq}>> \Det((F_vG)^{\times}) @>>>
\Det(\cH(F_vG)) @>{\epi}>> 
H^1(F_v, G^{ab})  \\
@| @VV{\sim}V @VVV @|   \\
G^{ab} @>{\subseteq}>> \Hom_{\Omega_{F_v}}(R_G,(F_v^c)^{\times}) @>>>
\Hom_{\Omega_{F_v}}(A_G, (F^c_v)^{\times}) @>{\epi}>> H^1(F_v, G^{ab}), 
\end{CD}
\end{equation}

and this implies that the map
\[
\Det(\cH(F_vG)) \to \Hom_{\Omega_{F_v}}(A_G, (F^{c}_{v})^{\times})
\]
is an isomorphism, which proves \eqref{E:vdetii}.

Suppose now that $v$ is finite of residue characteristic coprime to
$|G|$. In order to establish \eqref{E:vdetiv}, we first observe that
applying the functor $\Hom(-, (O_{F^{c}_{v}})^{\times})$ to the exact
sequence \eqref{E:aug} yields a sequence
\begin{equation} \label{E:intaug}
0 \to G^{ab} \to \Hom(R_G, (O_{F^{c}_{v}})^{\times}) \to \Hom(A_G,
(O_{F^{c}_{v}})^{\times}) \to 1
\end{equation}
which is surjective on the right because $(O_{F^{c}_{v}})^{\times}$ is
divisible. Taking $\Omega_{F_v}$-cohomology of \eqref{E:intaug} yields 
\begin{align} \label{E:intaug2i}
0 &\to G^{ab} \to \Hom_{\Omega_{F_v}}(R_G, (O_{F^{c}_{v}})^{\times})
\to \Hom_{\Omega_{F_v}}(A_G, (O_{F^{c}_{v}})^{\times}) \to \notag
\\ 
&\to H^1(F_v, G^{ab}) \xrightarrow{f} H^1(F_v, \Hom(R_G,
(O_{F^{c}_{v}})^{\times})).
\end{align}
Now since $v \nmid |G|$, $Z(O_{F_v}G)$ is an $O_{F_v}$-maximal order
in (the split algebra) $Z(F_vG)$, and $Z(O_{F^{c}_{v}}G)^{\times}
\simeq \Hom(R_G, (O_{F^{c}_{v}})^{\times})$ (cf. \eqref{E:vdetiii}).
Suppose that $\pi \in \Ker(f)$. Then there exists $u \in
Z(O_{F^{c}_{v}}G)^{\times}$ such that $u^{\omega} \cdot u^{-1} =
\pi(\omega)$ for all $\omega \in \Omega_{F_v}$. This implies that
$u^{|G^{ab}|} \in Z(O_{F_v}G)^{\times}$. As $v \nmid |G^{ab}|$ and
$Z(O_{F_v}G)$ is a maximal order, it follows that $u \in
Z(O_{F^{nr}_{v}}G)^{\times}$, and so $\pi \in H^{1}_{nr}(F_v,
G^{ab})$. Hence there is an exact sequence
\begin{equation} \label{E:intaug2}
0 \to G^{ab} \to \Hom_{\Omega_{F_v}}(R_G, (O_{F^{c}_{v}})^{\times})
\to \Hom_{\Omega_{F_v}}(A_G, (O_{F^{c}_{v}})^{\times}) \to H^{1}_{nr}(F_v,
G^{ab}).
\end{equation}

We recall also (see the proof of Proposition \ref{P:group}) that the
natural map $H^{1}_{nr}(F_v, G) \to H^{1}_{nr}(F_v, G^{ab})$ is
surjective because the group $\Gal(F^{nr}_{v}/F_v)$ is profinite free
on a single generator. Theorem \ref{T:reskey}(b) together with
\eqref{E:vdetiii} and \eqref{E:intaug2} now yield the following
commutative diagram:
\begin{equation} \label{E:ragcdii}
\begin{CD}
G @>{\subseteq}>> (O_{F_v}G)^{\times} @>>> \cH(O_{F_v}G) @>{\epi}>> H_{nr}^{1}(F_v, G) 
\\
@VVV    @VV{\Det}V   @VV{\Det}V @VV{\epi}V  \\
G^{ab} @>{\subseteq}>> \Det((O_{F_v}G)^{\times}) @>>>
\Det(\cH(O_{F_v}G)) @>{\epi}>> H_{nr}^{1}(F_v, G^{ab}) \\
@| @VV{\sim}V @VVV @|   \\
G^{ab} @>{\subseteq}>> \Hom_{\Omega_{F_v}}(R_G,(O_{F_v^c})^{\times}) @>>>
\Hom_{\Omega_{F_v}}(A_G, (O_{F^c_v})^{\times}) @>>> H_{nr}^{1}(F_v, G^{ab}) . 
\end{CD}
\end{equation}
It follows from \eqref{E:ragcdii} that the third row of this
diagram is surjective on the right. Since $\Det(\cH(O_{F_v}G))$ is a
subgroup of $\Hom_{\Omega_{F_v}}(A_G, (O_{F^c_v})^{\times})$, we see
that the map
\[
\Det(\cH(O_{F_v}G))\to \Hom_{\Omega_{F_v}}(A_G,
(O_{F^{c}_{v}})^{\times})
\]
is an isomorphism. This establishes \eqref{E:vdetiv}.
\end{proof}

If on the other hand $v$ is finite and $v \mid |G|$, so $O_{F_v}G$ is
not an $O_{F_v}$-maximal order in $F_vG$, then we have
\[
\Det(\cH(O_{F_v}G)) \subseteq
\Hom_{\Omega_{F_v}}(A_G, (O^{c}_{F_v})^{\times}),
\]
but this inclusion is not in general an equality. If $\fa$ is any
integral ideal of $O_F$, set
\[
U_{\fa}(O_{F^c_v}):= (1 + \fa O_{F^c_v}) \cap
(O_{F^c_v})^{\times},
\]
and write $U_{\alpha}(O_{F^c_v})$ instead of $U_{\fa}(O_{F^c_v})$ when
$\fa = \alpha O_F$.  We shall need the following result of A. Siviero
(which is a variant of \cite[Theorem 2.14]{Mc1}) in Section
\ref{S:modray}.

\begin{proposition} (A. Siviero)  \label{P:intim}
Let $v$ be a finite place of $F$. Then if $N \in \bZ_{>0}$ is
divisible by a sufficiently large power of $|G|$, we have
\[
\Hom_{\Omega_{F_v}}(A_G, U_N(O_{F^c_v})) \subseteq
\Det(\cH(O_{F_{v}}G)) \subseteq \Hom_{\Omega_{F_v}}(A_G,
(O_{F^c_v})^{\times}).
\]
\end{proposition}

\begin{proof} This is shown in \cite[Theorem 5.1.10]{Siv} when $G$ is
  abelian, and the proof for arbitrary finite $G$ is quite similar. As
  \cite{Siv} is not widely accessible, we describe the argument.

If $v \nmid |G|$, then Proposition \ref{P:agdet}(iii) implies that we
have
\[
\Hom_{\Omega_{F_v}}(A_G, O^{\times}_{F^c_v}) =
\Det(\cH(O_{F_v}G)) = \Hom_{\Omega_{F_v}}(A_G,
(O_{F^c_v})^{\times}),
\]
and so it follows that the desired result holds in this case. We may
therefore suppose that $v \mid |G|$.

We first observe that the group
\[
\frac{\Hom_{\Omega_{F_v}}(A_G,
  (O_{F^c_v})^{\times})}{\Det((O_{F_v}G)^{\times}/G)}
\]
is annihilated by $|G^{ab}|[\Det(\cM_{v}^{\times}):
  \Det(O_{F_v}G)^{\times}]$, where $\cM_v$ denotes any
$O_{F_v}$-maximal order in $F_vG$ containing $O_{F_v}G$. Since $A_G$
is finitely generated, it follows that $\Det((O_{F_v}G)^{\times}/G)$
is of finite index in $\Hom_{\Omega_{F_v}}(A_G,
(O_{F^c_v})^{\times})$, and so is an open subgroup of
$\Hom_{\Omega_{F_v}}(A_G, (O_{F^c_v})^{\times})$. The result now
follows from the fact that, because $v \mid |G|$, the collection of
groups
\[
\left\{ \Hom_{\Omega_{F_v}}(A_G, U_{|G|^{n}}(O_{F^c_v})) \mid n \geq
0 \right\}
\] 
is a fundamental system of neighbourhoods of the identity of
$\Hom_{\Omega_{F_v}}(A_G, (O_{F^c_v})^{\times})$.
\end{proof}

\begin{remark} \label{R:intim}
When $G$ is abelian, it follows from \cite[Theorem 2.14]{Mc1} that we
may take $N = |G|^2$ in Proposition \ref{P:intim}. \qed
\end{remark}

We shall also require the following related result in Section \ref{S:neu}.

\begin{proposition}   \label{P:kill}
Let $\Gamma$ be a finite group with an action of $\Omega_F$. Suppose
that $v \mid |\Gamma|$ is a finite place of $F$, and write $\fp_v$ for
the maximal ideal of $O_{F_v}$. Then for all sufficiently large $n$,
we have
\[
\Hom_{\Omega_{F_v}}(A_{\Gamma}, U_{\fp_{v}^{n}}(O_{F^c_v})) \subseteq
\rag[\Hom_{\Omega_{F_v^c}}(R_{\Gamma}, (O_{F_v^c})^{\times})].
\]
\end{proposition}

\begin{proof}
The proof of this is very similar to that of Proposition
\ref{P:intim}. We observe that 
\[
|\Gamma^{ab}| \cdot \Hom_{\Omega_{F^c_v}}(A_{\Gamma},
(O_{F_v^c})^{\times})
\subseteq 
\rag[\Hom_{\Omega_{F_v^c}}(R_{\Gamma}, (O_{F_v^c})^{\times})],
\]
which implies that $\rag[\Hom_{\Omega_{F_v^c}}(R_{\Gamma},
  (O_{F_v^c})^{\times})]$ is an open subgroup of
$\Hom_{\Omega_{F^c_v}}(A_{\Gamma}, (O_{F_v^c})^{\times})$ because
$A_{\Gamma}$ is finitely generated. The desired result now follows
from the fact that the collection of groups $\{
\Hom_{\Omega_{F_v}}(A_{\Gamma}, U_{\fp_{v}^{n}}(O_{F^c_v})) \mid n\geq
0 \}$ is a fundamental system of neighbourhoods of the identity of
$\Hom_{\Omega_{F^c_v}}(A_{\Gamma}, (O_{F_v^c})^{\times})$.
\end{proof}


\section{Twisted forms and relative $K$-groups} \label{S:reshom}

Recall that $G$ is a finite group upon which $\Omega_F$ acts
trivially. In this section, we shall recall some basic facts
concerning categorical twisted forms and relative algebraic
$K$-groups. The reader may consult \cite{AB} and \cite[Chapter 15]{Sw}
for some of the details that we omit.

\subsection{Twisted forms}
Suppose that $R$ is a Dedekind domain with field of fractions $L$ of
characteristic zero. (For notational convenience, we shall sometimes
also allow ourselves to take $R = L$.) Let $\fA$ be any $R$-algebra
which is finitely generated as an $R$-module and which satisfies $\fA
\otimes_R L \simeq LG$.

\begin{definition} \label{D:twisted}
Let $\Lambda$ be any extension of $R$, and write $\cP(\fA)$ and
$\cP(\fA \otimes_R \Lambda)$ for the categories of finitely generated,
projective $\fA$ and $\fA \otimes_R \Lambda$-modules respectively. A
\textit{categorical $\Lambda$-twisted $\fA$-form} (or \textit{twisted
  form} for short) is an element of the fibre product category
$\cP(\fA) \times_{\cP(\fA \otimes_R \Lambda)} \cP(\fA)$, where the
fibre product is taken with respect to the functor $\cP(\fA) \to
\cP(\fA \otimes_R \Lambda)$ afforded by extension of scalars. In
concrete terms therefore, a twisted form consists of a triple
$(M,N;\xi)$, where $M$ and $N$ are finitely generated, projective
$\fA$-modules, and
\[
\xi: M \otimes_R \Lambda \xrightarrow{\sim} N \otimes_R \Lambda
\]
is an isomorphism of $\fA \otimes_R \Lambda$-modules. \qed
\end{definition}

\begin{example} 
If $F_{\pi}/F$ is any $G$-extension, and $\cL_{\pi} \subseteq F_{\pi}$
is any non-zero projective $O_FG$-module, then $(\cL_{\pi}, O_FG; \br_G)$ is a
categorical $F^c$-twisted $O_FG$-form. In particular, if $F_{\pi}/F$
is a tame $G$-extension, then $(O_\pi, O_FG; \br_G)$ is a categorical
$F^c$-twisted $O_FG$-form. Similarly, if $v$ is any place of
$F$, then (still assuming $F_{\pi}/F$ to be tame) $(O_{\pi,v}, O_{F_v}G;
\br_G)$ is a categorical $F^c_v$-twisted $O_{F_v}G$-form. We shall mainly
be concerned with twisted forms of these types in this paper.  \qed
\end{example}

We write $K_0(\fA, \Lambda)$ for the Grothendieck group associated to
the fibre product category $\cP(\fA) \times_{\cP(\fA \otimes_R
  \Lambda)} \cP(\fA)$, and we write $[M,N;\xi]$ for the isomorphism
class of the twisted form $(M,N;\xi)$ in $K_0(\fA, \Lambda)$. The
group $K_0(\fA, \Lambda)$ is often called \textit{the relative $K$-group with
respect to the homomorphism $\fA \to \Lambda$}. Recall (see
\cite[Theorem 15.5]{Sw}) that there is a long exact sequence of
relative algebraic $K$-theory:
\begin{equation} \label{E:RKES}
K_1(\fA) \rightarrow K_1(\fA \otimes_R \Lambda)
\xrightarrow{\partial^{1}_{\fA,\Lambda}} K_0(\fA, \Lambda)
\xrightarrow{\partial^{0}_{\fA, \Lambda}} K_0(\fA) \rightarrow
K_0(\fA \otimes_R \Lambda).
\end{equation}

The first and last arrows in this sequence are afforded by
extension of scalars from $R$ to $\Lambda$. The map
$\partial^{0}_{\fA, \Lambda}$ is defined by
\[
\partial^{0}_{\fA, \Lambda}([M,N;\lambda]) = [M] - [N].
\]
The map $\partial^{1}_{\fA, \Lambda}$ is defined by first recalling
that the group $K_1(\fA \otimes_R \Lambda)$ is generated by pairs of the form
$(V, \phi)$, where $V$ is a finitely generated, free,
$\fA \otimes_R \Lambda$-module, and $\phi: V \xrightarrow{\sim} V$ is an
$\fA \otimes_R \Lambda$-isomorphism. If $T$ is any projective $\fA$-submodule
of $V$ satisfying $T \otimes_{\fA} \Lambda \simeq V$, then we set
\[
\partial^{1}_{\fA, \Lambda}(V, \phi) = [T, T; \phi].
\]
It may be shown that this definition is independent of the choice of $T$.

We shall often ease notation and write e.g. $\partial^{0}$ rather than
$\partial^{0}_{\fA, \Lambda}$ when no confusion is likely to result.

\subsection{Idelic description and localisation} \cite[Chapter II, \S1]{Fr}.
Let us retain the notation established above, and suppose in addition
that we now work over a number field $F$.  The reduced norm map
\[
\nrd: (FG)^{\times} \to Z(FG)^{\times}
\]
induces isomorphisms
\begin{equation} \label{E:kiso}
K_1(FG) \simeq \nrd(K_1(FG)) \simeq \nrd((FG)^{\times}) \simeq
\Det((FG)^{\times}) \subseteq Z(FG)^{\times}
\end{equation}
and
\begin{equation} \label{E:kviso}
K_1(F_vG) \simeq \nrd(K_1(F_vG)) \simeq \nrd((F_vG)^{\times}) \simeq
\Det((F_vG)^{\times}) \subseteq Z(F_vG)^{\times}
\end{equation}
for each place $v$ of $F$. In general the natural map $K_1(\fA_v) \to
K_1(F_vG)$ is not injective, and so the reduced norm map
\[
\nrd: K_1(\fA_v) \to Z(\fA_v)^{\times}
\]
is not an isomorphism (although it is surjective if $\fA_v$ is an
$O_{F_v}$-maximal order in $F_vG$). If we write $K_1(\fA_v)'$ for the image of
$K_1(\fA_v)$ in $K_1(F_vG)$, then \eqref{E:kviso} induces isomorphisms
\begin{equation} \label{E:okviso}
K_1(\fA_v)' \simeq \nrd(K_1(\fA_v)') \simeq \nrd((\fA_v)^{\times})
\simeq \Det(\fA_{v}^{\times}).
\end{equation}
We shall make frequent use of the identifications \eqref{E:kiso},
\eqref{E:kviso} and \eqref{E:okviso} (as well as those afforded by
Proposition \ref{P:agdet}) in what follows, sometimes without explicit
mention.

For each place $v$ of $F$, we write
\[
\loc_v: K_1(FG) \to K_1(F_vG)
\]
for the obvious localisation map. 

\begin{definition} We define the group of ideles
$J(K_1(FG))$ of $K_1(FG)$ to be the restricted direct product over all
  places $v$ of $F$ of the groups $\Det(F_vG)^{\times} \simeq
  K_1(F_vG)$ with respect to the subgroups $\Det(O_{F_v}G)^{\times}$.
  We define the group of finite ideles $J_f(K_1(FG))$ in a similar
  manner but with the restricted direct product taken over all finite
  places $v$ of $F$.
\qed
\end{definition}

If $E$ is any extension of $F$, then the homomorphism
\[
\Det(FG)^{\times} \to J(K_1(FG)) \times \Det(EG)^{\times};\quad x \mapsto
((\loc_v(x))_v, x^{-1})
\]
induces a homomorphism
\[
\Delta_{\fA,E}: \Det(FG)^{\times} \to
\frac{J(K_1(FG))}{\prod_{v} \Det(\fA_v)^{\times}} \times
\Det(EG)^{\times}.
\]

\begin{theorem} \label{T:kdes}
(a) There is a natural isomorphism
\[
\Cl(\fA) \xrightarrow{\sim} \frac{J(K_1(FG))}{\Det(FG)^{\times}
    \prod_{v} \Det(\fA_v)^{\times}}.
\]

(b) There is a natural isomorphism
\[
h_{\fA,E}: K_0(\fA, E) \xrightarrow{\sim} \Coker(\Delta_{\fA,E}).
\]
\qed
\end{theorem}

\begin{proof} Part (a) is a well-known result of A. Fr\"ohlich (see
  e.g \cite[Chapter I]{Fr1}. Part (b) is proved in \cite[Theorem
    3.5]{AB}.
\end{proof}

\begin{remark} \label{R:repid}
If $[M,N;\xi] \in K_0(\fA, E)$ and $M$, $N$ are locally free
$\fA$-modules of rank one (which is the only case that we shall need in this
paper), then $h_{\fA,E}([M,N;\xi])$ may be described explicitly as
follows. 

For each place $v$ of $F$, we choose $\fA_v$-bases $m_v$ of $M_v$ and
$n_v$ of $N_v$. We also choose an $FG$ basis $n_{\infty}$ of $N_F$, as
well as an $FG$-module isomorphism $\theta: M_F \xrightarrow{\sim}
N_F$.  Then, for each $v$, we may write $n_v = \nu_v \cdot
n_{\infty}$, with $\nu_v \in (F_vG)^{\times}$. As
$\theta^{-1}(n_\infty)$ is an $FG$-basis of $M_F$, we may write $m_v =
\mu_v \cdot \theta^{-1}(n_{\infty})$, with $\mu_v \in
(F_vG)^{\times}$. Finally, writing $\theta_E$ for the map $M_E \to
N_E$ afforded by $\theta$ via extension of scalars from $F$ to $E$, we
have that $(\xi \circ \theta^{-1}_{E})(n_\infty) = \nu_{\infty} \cdot
n_{\infty}$ for some $\nu_{\infty} \in (EG)^{\times}$.  Then a
representative of $h_{\fA,E}([M,N;\xi])$ is given by the image of
$[(\mu_v \cdot \nu_{v}^{-1})_v, \nu_{\infty}]$ in $J(K_1(FG)) \times
K_1(EG)$, and a representative of $\partial^0(h_{\fA,E}([M,N;\xi]))
\in \Cl(\fA)$ is given by the image of $(\mu_v \cdot \nu_{v}^{-1})_v
\in J(K_1(FG))$.
\qed
\end{remark}

\begin{remark} \label{R:noinf}
As $\fA_v = F_vG$ when $v$ is infinite (by convention), we see that
\[
\frac{J(K_1(FG))}{\prod_{v} \Det(\fA_v)^{\times}}
\simeq
\frac{J_f(K_1(FG))}{\prod_{v \nmid \infty} \Det(\fA_v)^{\times}}.
\]
Hence the infinite places of $F$ in fact play no explicit role on the
right-hand sides of the isomorphisms given by Theorem \ref{T:kdes},
and so these isomorphisms may be formulated using the finite idele
group $J_f(K_1(FG))$ of $K_1(FG)$ instead of the full idele group
$J(K_1(FG))$.  \qed
\end{remark}

\begin{lemma} \label{L:klocdes}
Suppose that $v$ is a place of $F$ and that $E_v$ is any extension of
$F_v$. Then there is an isomorphism
\[
K_0(\fA_v, E_v) \simeq \Det(E_vG)^{\times}/ \Det(\fA_v)^{\times}.
\]
\end{lemma}

\begin{proof} This follows directly from the long exact sequence of
  relative $K$-theory \eqref{E:RKES} applied to $K_0(\fA_v, E_v)$,
  together with \eqref{E:kviso} and \eqref{E:okviso}.
\end{proof}

For each place $v$ of $F$, there is a localisation map on
relative $K$-groups:
\[
\lambda_v: K_0(\fA, E) \to K_0(\fA_v, E_v); \quad [M,N;\xi] \mapsto [M_v, N_v;
  \xi_v],
\]
where $\xi_v$ denotes the map obtained from $\xi$ via extension of
scalars from $E$ to $E_v$. It is not hard to check that, in terms of
the descriptions of $K_0(\fA, E)$ and $K_0(\fA_v, E_v)$ afforded by
Theorem \ref{T:kdes} and Lemma \ref{L:klocdes}, the map $\lambda_v$ is
that induced by the homomorphism (which we denote by the same symbol
$\lambda_v$) 
\[
\lambda_v: J(K_1(FG)) \times \Det(EG)^{\times}  \to \Det(E_vG)^{\times};
\quad [(x_v)_v, x_{\infty}] \mapsto [x_v \cdot \loc_v(x_{\infty})].
\]

\begin{definition}  \label{D:kiddes}
We define the idele group $J(K_0(\fA, E))$ of $K_0(\fA, E)$ to be
the restricted direct product over all places $v$ of $F$ of the groups
$K_0(\fA_v, E_v)$ with respect to the subgroups $K_0(\fA_v,
O_{E_v})$.

We define the group of finite ideles $J_f(K_0(\fA, F^c))$ in a similar
manner, but with the restricted direct product taken over all finite
places of $F$.
\qed
\end{definition}

\begin{proposition} \label{P:locinj}
(a) The homomorphism
\[
\lambda:= \prod_v \lambda_v: K_0(\fA, E) \to \prod_v K_0(\fA_v, E_v)
\]
is injective.

(b) If $F$ has no real places or if $G$ admits no irreducible
symplectic characters, then the homomorphism
\[
\lambda_f:= \prod_{v \nmid \infty} \lambda_v: K_0(\fA, E) \to
\prod_{v \nmid \infty} K_0(\fA_v, E_v) 
\]
is injective.

(c) The image of $\lambda$ lies in the idele group $J(K_0(\fA, E))$.
\end{proposition}

\begin{proof} 
(a) Suppose that $\alpha \in K_0(\fA, E)$ lies in the kernel
  of $\lambda$, and let 
\[
[(x_v)_v, x_\infty] \in J(K_1(FG)) \times \Det(EG)^{\times}
\] 
be a representative of $\alpha$. Then for each $v$, we have
\begin{equation} \label{E:locunit}
x_v \cdot \loc_v(x_\infty) \in \Det(\fA_v)^{\times} \subseteq
\Det(F_{v}G)^{\times}.
\end{equation}
Since $x_v \in \Det(F_vG)^{\times} \subseteq Z(F_vG)^{\times}$, we see
that $\loc_v(x_{\infty}) \in Z(F_vG)^{\times}$ for each $v$. Hence
$x_{\infty} \in Z(FG)^{\times}$, and so via the Hasse--Schilling norm
theorem (see \cite[Theorem 7.6]{SE} or \cite[Theorem 7.8]{CR1}) we
deduce that $x_{\infty} \in \Det(FG)^{\times}$. Hence $\alpha$ is also
represented by the idele
\[
[(\loc_v(x_{\infty}))_{v}, x^{-1}_{\infty}] \cdot [(x_{v})_{v},
  x_{\infty}] = [(x_{v} \cdot \loc_{v}(x_{\infty}))_{v},1],
\]
and now \eqref{E:locunit} and Theorem \ref{T:kdes}(b) imply that $\alpha
= 0$ in $K_0(\fA, E)$. Therefore $\lambda$ is injective, as claimed.

(b) The proof of this assertion is virtually identical to that of part
(a). Using the same notation as in the proof of part (a), we see that
$\loc_v(x_{\infty}) \in \Det(F_vG)^{\times} \simeq Z(F_vG)^{\times}$
for each finite place $v$ of $F$. This implies that $x_{\infty} \in
Z(FG)^{\times}$. Under our hypotheses, we have that $\Det(FG)^{\times}
\simeq Z(FG)^{\times}$, and so $x_{\infty} \in \Det(FG)^{\times}$. The
remainder of the argument proceeds exactly as in the proof of part
(a).

(c) If $\beta = [M,N;\xi] \in K_0(\fA, E)$, then for all but
finitely many places $v$, the isomorphism $\xi_v: M \otimes_{O_F}
E_v \xrightarrow{\sim} N \otimes_{O_F} E_v$ obtained from $\xi$
via extension of scalars from $E$ to $E_v$ restricts to an
isomorphism $M \otimes_{O_F} O_{E_v} \xrightarrow{\sim} N
\otimes_{O_F} O_{E_v}$. Hence, for all but finitely many $v$, we
have that $\lambda_v(\beta) \in K_0(\fA_v, O_{E_v})$, and so
$\lambda(\beta) \in J(K_0(\fA, E))$, as asserted.
\end{proof}


\section{Cohomological classes in relative $K$-groups} \label{S:cc}

Recall that $F$ is a number field and that $G$ is a finite group upon
which $\Omega_F$ acts trivially. In this section we shall explain how
the set of realisable classes $\cR(O_FG) \subseteq \Cl(O_FG)$ may be
studied via imposing local cohomological conditions on elements of the
relative $K$-group $K_0(O_FG, F^c)$.

\begin{definition} \label{D:psi}
We define maps $\bpsi$ and $\bpsi_v$ (for each place $v$ of
$F$) by
\[
\bpsi = \bpsi_G: H^1_t(F,G) \to K_0(O_FG, F^c); \quad [\pi] \mapsto
[O_{\pi}, O_FG; \br_G]
\]
and
\[
\bpsi_v = \bpsi_{G,v} :H^1_t(F_v,G) \to K_0(O_{F_v}G, F_v^c); \quad [\pi_v] \mapsto
     [O_{\pi_v}, O_{F_v}G; \br_G].
\]
We set
\[
\bcR(O_FG):= \Image(\bpsi).
\]
\qed
\end{definition}

\begin{definition} \label{D:psid}
We define the pointed set of ideles $J(H^1_t(F,G))$ of $H^1_t(F,G)$ to
be the restricted direct product over all places $v$ of $F$ of the
pointed sets $H^{1}_{t}(F_v, G)$ with respect to the pointed subsets
$H^{1}_{nr}(F_v, G)$, and we write
\[
\bpsi^{id}: J(H^1_t(F, G)) \to J(K_0(O_FG, F^c))
\]
for the map afforded by the maps $\bpsi_v: H^1_t (F_v, G) \to
K_0(O_{F_v}G, F_v^c)$. \qed
\end{definition}

In general, $\bcR(O_FG)$ is not a subgroup of $K_0(O_FG,
F^c)$. However, although $H^{1}_{nr}(F_v, G)$ is in general merely a
pointed set and not a group, the following result holds.

\begin{proposition} \label{P:unramgp}
Let $v$ be any place of $F$, and write $\bpsi_{v}^{nr}$ for the
restriction of $\bpsi_v$ to $H^{1}_{nr}(F_v,G)$. Then $\Image(\bpsi_{v}^{nr})$
is a subgroup of $K_0(O_{F_v}G, F_v^c)$.
\end{proposition}

\begin{proof} 
If $v$ is infinite, then $H^{1}_{nr}(F_v, G) = 0$, and so
$\Image(\bpsi_{v}^{nr}) = 0$.  For finite $v$, the result follows from
Proposition \ref{P:group} and Lemma \ref{L:klocdes}.
\end{proof}

\begin{definition} \label{D:cc}
We say that an element $x \in K_0(O_FG, F^c)$ is
\textit{cohomological} (respectively \textit{cohomological at $v$}) if
$x \in \Image(\bpsi)$ (respectively $\lambda_v(x) \in
\Image(\bpsi_v)$). We say that $x$ is \textit{locally cohomological}
if $x$ is cohomological at $v$ for all places $v$ of $F$. We write
\[
\LC(O_FG):= \lambda^{-1}(\Image(\bpsi^{id}))
\]
for the subset of $K_0(O_FG,F^c)$ consisting of locally cohomological
elements.
\qed
\end{definition}

The long exact sequence of relative $K$-theory \eqref{E:RKES} applied
to $K_0(O_FG, F^c)$ yields a long exact sequence
\begin{equation} \label{E:rkes}
K_1(O_FG) \to K_1(F^cG) \xrightarrow{\partial^1} K_0(O_FG, F^c)
\xrightarrow{\partial^0} \Cl(O_FG) \to 0,
\end{equation}
where $\Cl(O_FG)$ denotes the locally free class group of $O_FG$. We
set
\[
\psi := \partial^0 \circ \bpsi,
\]
and we write
\[
\cR(O_FG): = \Image(\psi).
\]

The second-named author has conjectured that that $\cR(O_FG)$ is
always a subgroup of $\Cl(O_FG)$, and he has proved that this is true
whenever $G$ is abelian (see \cite[Corollary 6.20]{Mc1}). The
following conjecture gives a precise characterisation of the image
$\bcR(O_FG)$ of $\bpsi$.

\begin{conjecture} \label{C:cc}
An element of $K_0(O_FG, F^c)$ is cohomological if and only if it is
locally cohomological. In other words, we have that
\[
\bcR(O_FG) = \LC(O_FG).
\]
\qed
\end{conjecture}

Let us now explain why Conjecture \ref{C:cc} implies that $\cR(O_FG)$
is a subgroup of $\Cl(O_FG)$.  In order to do this, we shall require
the following result which is equivalent to a theorem of the
second-named author when $G$ is abelian, and whose proof relies on
results contained in \cite{Mc1} and \cite{Mc2}. Before stating the
result, we remind the reader that $\prod_v \Image(\bpsi^{nr}_{v})$ is
not merely a pointed set, but is in fact a subgroup of $J(K_0(O_FG,
F^c))$ (see Proposition \ref{P:unramgp}).

\begin{theorem} \label{T:idgp}
Let 
\[
\ov{\bpsi^{id}}: J(H^1_t(F, G)) \to \frac{J(K_0(O_FG, F^c))}{\lambda[
\partial^{1} (K_1(F^cG))] \cdot \prod_v \Image(\bpsi^{nr}_{v})}
\]
denote the map of pointed sets given by the composition of the map
$\bpsi^{id}$ with the quotient homomorphism
\[
J(K_0(O_FG, F^c)) \to \frac{J(K_0(O_FG, F^c))}{\lambda
[\partial^{1} (K_1(F^cG))] \cdot \prod_v \Image(\bpsi^{nr}_{v})}.
\]
Then the image of $\ov{\bpsi^{id}}$ is in fact a group. Hence it
follows that
\[
\lambda[\partial^{1}(K_1(F^cG))] \cdot \Image(\bpsi^{id})
\]
is a subgroup of $J(K_0(O_FG, F^c))$.
\end{theorem}

This theorem will be proved in Section \ref{S:idgp}. It implies the
following result.

\begin{theorem} \label{T:congp}
If Conjecture \ref{C:cc} holds, then $\cR(O_FG)$ is a subgroup of
$\Cl(O_FG)$.
\end{theorem}

\begin{proof}
It follows from the exact sequence \eqref{E:rkes} that $\cR(O_FG)$ is
a subgroup of $\Cl(O_FG)$ if and only if $\partial^1(K_1(F^cG)) \cdot
\bcR(O_FG)$ is a subgroup of $K_0(O_FG, F^c)$. However, if Conjecture
\ref{C:cc} is true, then Theorem \ref{T:idgp} implies that 
\begin{equation} \label{E:later}
\partial^1(K_1(F^cG)) \cdot
\bcR(O_FG) = \partial^1(K_1(F^cG)) \cdot \LC(O_FG)
\end{equation} 
is the kernel of the homomorphism
\[
K_0(O_FG, F^c) \xrightarrow{\lambda} J(K_0(O_FG, F^c)) \to
\frac{J(K_0(O_FG, F^c))}{\lambda[\partial^{1} (K_1(F^cG))] \cdot
\Image(\bpsi^{id})},
\]
where the last arrow denotes the obvious quotient homomorphism. This
implies the desired result.
\end{proof}

We conclude this section with the following result on unramified
locally cohomological classes in $K_0(O_FG, F^c)$. This will be used
in the proofs of Theorem \ref{T:domco} and Theorem \ref{T:D} of the
Introduction (see Section \ref{S:odd} below).

\begin{proposition} \label{P:unram}
(a) Let $L$ be the maximal, abelian, everywhere unramified (including at
all infinite places) extension of $F$ of exponent $|G^{ab}|$, and
suppose that $y \in K_0(O_FG, F^c)$ lies in the kernel of the map
\[
\beta: K_0(O_FG, F^c) \xrightarrow{\lambda_F} J(K_0(O_FG, F^c)) \to
\frac{J(K_0(O_FG, F^c))}{\prod_v \Image(\bpsi^{nr}_{v})}.
\]
Then $y$ lies in the kernel of the extension of scalars map
\[
e_L: K_0(O_FG, F^c) \to K_0(O_LG, F^c).
\]
Hence, if $(h^+_F, |G^{ab}|) = 1$ (where $h^+_F$ denotes the narrow
class number of $F$), then $L = F$, and so $\beta$ is injective.

(b) Suppose that $G$ admits no non-trivial irreducible symplectic
characters, or that $F$ has no real places, and that $y \in K_0(O_FG,
F^c)$ lies in the kernel of the map
\[
\beta_f: K_0(O_FG, F^c) \xrightarrow{\lambda_{f,F}} J_f(K_0(O_FG, F^c)) \to
\frac{J_f(K_0(O_FG, F^c))}{\prod_{v \nmid \infty} \Image(\bpsi^{nr}_{v})}.
\]
Then $y$ lies in the kernel of the extension of scalars map
\[
e_M: K_0(O_FG, F^c) \to K_0(O_MG, F^c),
\]
where $M$ is the maximal, abelian, unramified (at all finite places)
extension of $F$ of exponent $|G^{ab}|$.

Hence if $(h_F, |G^{ab}|) = 1$ then $L = F$, and so $\beta_f$ is
injective.
\end{proposition}

\begin{proof}
(a) Suppose that $y = [(y_v), y_{\infty}]$ lies in the kernel of $\beta$,
and let $E/F$ be the smallest Galois extension such that $\Omega_E$
fixes $y_{\infty}$. For each place $v$ of $F$, let $w(v)$ be the place
of $E$ afforded by our fixed choice of embedding $F^c \to F^c_v$.

As $y$ lies in the kernel of $\beta$, we have that $y_v \cdot
\loc_v(y_{\infty}) \in \Image(\bpsi_{v}^{nr})$ for each place $v$. Hence,
for each $v$, $\loc_v(y_{\infty}) \in H(Z(F_vG))$ is an unramified
$G^{ab}$-resolvend over $F_v$ (cf. Proposition \ref{P:abres}). It
follows that, for each $v$, the extension $E_{w(v)}/F_v$ is unramified
and that $[E_{w(v)}:F_v]$ divides $|G^{ab}|$. This implies that $E/F$ is
unramified at all places $v$, and is of exponent dividing
$|G^{ab}|$. Hence $E \subseteq L$, and so $y_{\infty} \in \Det(LG)^{\times}$.

Now since $y_v \cdot \loc_v(y_{\infty}) \in \Image(\bpsi_{v}^{nr})$
for each place $v$, we see that in fact $y_v \cdot \loc_v(y_{\infty}) \in
\Det(O_{L_v}G)^{\times}$. Hence $e_L(y)$ is in the kernel of the
localisation map
\[
\lambda_L: K_0(O_LG, F^c) \to J(K_0(O_LG, F^c)),
\]
and since $\lambda_L$ is injective (see Proposition \ref{P:locinj}(a)) it
follows that $e_L(y) = 0$.

The final assertion now follows immediately.

(b) Virtually identical to the proof of (a), except that here, because
either $G$ admits no irreducible symplectic characters or $F$ has no
real places, we may appeal to the injectivity of the localisation map
$\lambda_{f,M}$ (see Proposition \ref{P:locinj}(b)) rather than that
of $\lambda_{M}$.
\end{proof} 


\section{Local extensions I} \label{S:local}

The goal of this section is to describe how resolvends of normal
integral bases of tamely ramified, non-archimedean local extensions
admit \textit{Stickelberger factorisations} (see Definition
\ref{D:stickfac}). This reflects the fact that every tamely ramified
$G$-extension of $F_v$ is a compositum of an unramified extension of
$F_v$ and a twist of a totally ramified extension of $F_v$. All of the
results in this section are based on unpublished notes of the
second-named author.

For each finite place $v$ of $F$, we fix a uniformiser $\vp_v$ of
$F_v$, and we write $q_v$ for the order of the residue field of
$F_v$. We fix a compatible set of roots of unity $\{ \zeta_m \}$, and
a compatible set $\{ \vp_{v}^{1/m} \}$ of roots of $\vp_v$. So, if $m$ and
$n$ are any two positive integers, then we have $(\zeta_{mn})^m =
\zeta_n$, and $(\vp_{v}^{1/mn})^{m} = \vp_{v}
^{1/n}$.

Recall that $F_{v}^{nr}$ (respectively $F_{v}^{t}$) denotes the
maximal unramified (respectively tamely ramified) extension of
$F_v$. Then
\[
F_{v}^{nr}=  \bigcup_{\stackrel{m \geq 1}{(m,q_v)=1}} F_v(\zeta_m),\quad
F_v^t = \bigcup_{\stackrel{m \geq 1}{(m,q_v)=1}} F_v(\zeta_m, \vp_{v}^{1/m}).
\]
The group $\Omega_{v}^{nr}:=
\Gal(F_{v}^{nr}/F_v)$ is topologically generated by a Frobenius
element $\phi_v$ which may be chosen to satisfy
\[
\phi_v(\zeta_m) = \zeta_{m}^{q_v}, \qquad \phi_v(\vp_{v}^{1/m}) = \vp_{v}^{1/m}
\]
for each integer $m$ coprime to $q_v$. Our choice of compatible roots
of unity also uniquely specifies a topological generator $\sigma_v$ of
$\Gal(F_v^t/F_{v}^{nr})$ by the conditions 
\[
\sigma_v(\vp_{v}^{1/m}) = \zeta_m \cdot \vp_{v}^{1/m}, \qquad
\sigma_v(\zeta_m) = \zeta_m
\]
for all integers $m$ coprime to $q_v$. The group
$\Omega^{t}_{v}:=\Gal(F_{v}^{t}/F_v)$ is topologically generated by
$\phi_v$ and $\sigma_v$, subject to the relation
\begin{equation} \label{E:tamerel}
\phi_v \cdot \sigma_v \cdot \phi_{v}^{-1} = \sigma_{v}^{q_v}.
\end{equation}

While reading the remainder of this section (especially Proposition
\ref{P:phi} below), it may be helpful for the reader to keep in mind
the statement and proof of the following well-known result which
provides some motivation for a number of subsequent constructions.

\begin{proposition} \label{P:kmot}
Set $L:=F_v$. Let $n$ be a positive integer with $(n, q_v) = 1$, and
suppose that $\mu_n \subseteq L$. Set $E = L(\vp_{v}^{1/n})$, $\Gamma
= \Gal(E/L) = \bZ/n\bZ$, and $\beta = \sum_{i=0}^{n-1}
\vp_{v}^{i/n}$. Then $O_E = O_L\Gamma \cdot \beta$.
\end{proposition}

\begin{proof}
We first observe that plainly $O_L\Gamma \cdot \beta \subseteq O_E$,
as $\beta \in O_E$.

Let $\chi$ denote the Kummer character of $\Gamma$, defined by
\[
\chi(\gamma) = \frac{\gamma(\vp_{v}^{1/n})}{\vp_{v}^{1/n}} \in \mu_n
\]
for each $\gamma \in \Gamma$. Then $\wh{\Gamma} = \langle \chi
\rangle$, and for each $0 \leq j \leq n-1$, we have
\begin{align*}
\left( \sum_{\gamma} \chi^j(\gamma)\gamma^{-1} \right) \cdot \beta &= 
\left( \sum_{\gamma} \chi^j(\gamma) \gamma^{-1} \right) \cdot \left(
\sum_{i=0}^{n-1} \vp_{v}^{i/n} \right) \\
&= \sum_{i=0}^{n-1} \left( \sum_{\gamma} \chi^j(\gamma) \cdot
\chi^{-i}(\gamma) \cdot \vp_{v}^{i/n} \right) \\
&= n \cdot \vp_{v}^{j/n}.
\end{align*}

As $n \in O_{L}^{\times}$, we therefore see that
$\{\vp_{v}^{j/n}\}_{j=0}^{n-1} \subseteq O_L\Gamma \cdot \beta$, which
implies that $O_E \subseteq O_L\Gamma \cdot \beta$. This implies the
desired result.
\end{proof}

\begin{definition} \label{D:admiss}
For each finite place $v$ of $F$, we define
\[
\Sigma_v(G) := \{ s \in G \mid s^{q_v} \in c(s) \}
\]
(recall that $c(s)$ denotes the conjugacy class of $s$ in
$G$). Plainly if $s \in \Sigma_v(G)$, then both $c(s)$ and $\langle s
\rangle$ are subsets of $\Sigma_v(G)$. Let us also remark that if $s
\in \Sigma_v(G)$, then the order $|s|$ of $s$ is coprime to
$q_v$. \qed
\end{definition}

\begin{definition} \label{D:ramphi}
If $s \in G$, we set
\[
\beta_s := \frac{1}{|s|} \sum_{i=0}^{|s|-1} \vp_{v}^{i/|s|};
\]
note that $\beta_s$ depends only upon $|s|$, and so in particular we
have 
\[
\beta_{s} = \beta_{g^{-1}sg}
\]
for every $g \in G$.  We define $\vphi_{v,s} \in \Map(G, O_{F_v^c})$
by setting
\[
\vphi_{v,s}(g) = 
\begin{cases}
\sigma_{v}^{i}(\beta_s)   &\text{if $g = s^i$;} \\
0     &\text{if $g \notin \langle s \rangle$.}
\end{cases}
\]
Then
\begin{equation} \label{E:locres}
\br_G(\vphi_{v,s}) = \sum_{i=0}^{|s|-1} \vphi_{v,s}(s^i) s^{-i} =
\sum_{i=0}^{|s|-1} \sigma_{v}^{i}(\beta_s) s^{-i}.
\end{equation}
We note that for each $g \in G$, we have
\begin{equation} \label{E:conjres}
\br_G(\vphi_{v,g^{-1} s g}) = g^{-1} \cdot \br_G(\vphi_{v,s}) \cdot g,
\end{equation}
and so
\begin{equation} \label{E:keydet}
\Det(\br_G(\vphi_{v,g^{-1} s g})) = \Det(\br_G(\vphi_{v,s})),
\end{equation}
i.e. the element $\Det(\br_G(\vphi_{v,s}))$ depends only upon the
conjugacy class $c(s)$ of $s$ in $G$.  We remark that it will be shown
later as a consequence of properties of the Stickelberger pairing that
$\Det(\br_G(\vphi_{v,s}))$ in fact determines the subgroup $\langle s
\rangle$ of $G$ up to conjugation (see Remark \ref{R:nrd} and
Proposition \ref{P:reseq}(b)).

We shall see that generators of inertia subgroups of tame Galois
$G$-extensions of $F_v$ lie in $\Sigma_v(G)$, and that the elements
$\vphi_{v,s}$ for $s \in G$ with $(|s|, q_v) =1$ may be used to
construct normal integral basis generators of tame (and of course
totally ramified) Galois $G$-extensions of $F_{v}^{nr} $.  \qed
\end{definition}

In order to ease notation, we shall now set $L:= F_v$ and $O:=O_L$,
and we shall drop the subscript $v$ from our notation for the rest of
this section.
\medskip

Suppose now that $L_\pi/L$ is a tamely ramified Galois $G$-extension
of $L$, corresponding to $\pi \in \Hom(\Omega^{t}, G)$. We are going
to describe the second-named author's decomposition of resolvends of
normal integral basis generators of $L_\pi/L$ (see \cite{Mc2} and also
\cite[Section 6]{B}). When $G$ is abelian, this decomposition is an
analogue of a version of Stickelberger's factorisation of Gauss sums.

Write $s:= \pi(\sigma)$, $t:= \pi(\phi)$; then $t \cdot s \cdot t^{-1}
= s^q$, and so $s \in \Sigma(G)$. We define $\pi_r, \pi_{nr}
\in \Map(\Omega^t,G)$ by setting
\begin{align}
&\pi_r(\sigma^m \phi^n) = \pi(\sigma^m) = s^m , \label{E:rmap} \\
&\pi_{nr}(\sigma^m \phi^n) = \pi(\phi^n) = t^n. \label{E:nrmap}
\end{align}
If $\omega_i \in \Omega^t$ ($i=1,2$) with $\omega_i = \sigma^{m_i}
\cdot \phi^{n_i}$, then a straightforward calculation using
\eqref{E:tamerel} shows that
\[
\omega_1 \cdot \omega_2 = \sigma^{m_1 + m_{2}q^{n_{1}}} \cdot \phi^{n_{1} + n_{2}}.
\]
This implies that $\pi_{nr} \in \Hom(\Omega^{nr},G)$. Plainly we have 
\begin{equation} \label{E:pifac}
\pi(\omega)= \pi_r(\omega) \cdot \pi_{nr}(\omega)
\end{equation}
for every $\omega = \sigma^m \cdot \phi^n \in \Omega^t$.  The map
$\pi_{nr} \in \Hom(\Omega^{nr}, G)$ corresponds to an unramified
Galois $G$-extension $L_{\pi_{nr}}$ of $L$ (see Remark \ref{R:dbend}
below for a more detailed discussion of this point). Since
$L_{\pi_{nr}}/L$ is unramified, $O_{\pi_{nr}}$ is a free
$O_LG$-module. Let $a_{nr}$ be any normal integral basis generator of
this extension. Note that $\br_G(a_{nr}) \in H(OG)$, because
$L_{\pi_{nr}}/L$ is unramified (see Corollary \ref{C:dual}(iv)).

\begin{definition} \label{D:gnr}
Let $G(\pi_{nr})$ denote the group $G$ with $\Omega^t$-action given
by
\[
\omega(g) = \pi_{nr}(\omega) \cdot g \cdot \pi_{nr}(\omega)^{-1}
\] 
for $\omega \in \Omega^t$ and $g \in G$. \qed
\end{definition}

\begin{lemma} \label{L:cocycle}
The map $\pi_r$ is a $G(\pi_{nr})$-valued 1-cocycle of
  $\Omega^t$.
\end{lemma}

\begin{proof} Suppose that $\omega_1, \omega_2 \in \Omega^t$. Then
  since $\pi_{nr} \in \Hom(\Omega^{nr}, G)$ and $\pi = \pi_r
  \cdot \pi_{nr}$, a straightforward calculation shows that
\[
\pi_r(\omega_1 \omega_2) = \pi_r(\omega_1) \cdot
\pi_{nr}(\omega_1) \cdot \pi_{r}(\omega_2) \cdot \pi_{nr}(\omega_1)^{-1},
\]
and this establishes the desired result.
\end{proof} 

\begin{definition} \label{D:gnr1}
We write $^{\pi_r} G(\pi_{nr})$ for the set $G$ endowed with the
following action of $\Omega^t$: for every $g \in G$ and $\omega \in
\Omega^t$ we have
\[
g^{\omega} = \pi_r(\omega) \cdot \pi_{nr}(\omega) \cdot g \cdot \pi_{nr}(\omega)
^{-1}.
\]
Lemma \ref{L:cocycle} implies that if $\omega_1, \omega_2 \in
\Omega^t$, then
\[
g^{(\omega_1 \omega_2)} = (g^{\omega_2})^{\omega_1}.
\]
 We set
\[
L_{\pi_r}(\pi_{nr}):= \Map_{\Omega^t}(^{\pi_r} G(\pi_{nr}), L^t).
\]
The algebra $(L^tG(\pi_{nr}))^{\Omega^t}$ acts on $L_{\pi_r}(\pi_{nr})$
via the rule
\[
(\alpha \cdot a)(h) = \sum_{g \in G} \alpha_g \cdot a (h \cdot g)
\]
for all $h \in G$ and $\alpha = \sum_{g \in G} \alpha_g \cdot g \in
(L^tG(\pi_{nr}))^{\Omega^t}$. \qed
\end{definition}

\begin{proposition} \label{P:phi}
(a) Recall that $s \in \Sigma(G)$. We have that $\vphi_s \in
  L_{\pi_r}(\pi_{nr})$.

(b) Set
\[
\fA(\pi_{nr}) = (O_{L^c}G(\pi_{nr}))^{\Omega^t},
\] 
and let
  $O_{\pi_r}(\pi_{nr})$ be the integral closure of $O_L$ in
  $L_{\pi_r}(\pi_{nr})$. Then
\[
\fA(\pi_{nr}) \cdot \vphi_s = O_{\pi_r}(\pi_{nr}).
\]

(c) For any $\alpha_r \in L_{\pi_r}(\pi_{nr})$ and $\omega \in
\Omega^t$, we have
\[
\br_G(\alpha_r)^{\omega} = \pi_{nr}(\omega)^{-1} \cdot
\br_G(\alpha_r) \cdot \pi(\omega).
\]
\end{proposition}

\begin{proof}
(a) Suppose that $\omega = \sigma^m \cdot \phi^n \in \Omega^t$. If $g
  \in G$  and $g \notin \langle s \rangle$, then we have that
\[
\vphi_s(g^{\omega}) = 0 = \vphi_s(g)^{\omega}.
\]
On the other hand, we also have
\begin{align*}
\vphi_s((s^i)^{\omega}) &= \vphi_s((s^i)^{\sigma^m \phi^{n}}) \\
&= \vphi_s(s^m \cdot t^n \cdot s^i \cdot t^{-n}) \\
&= \vphi_s(s^{m + iq^{n}}) \\
&= \sigma^{m + iq^{n}}(\beta_s) \\
&= (\sigma^m \cdot \phi^n) \cdot \sigma^i(\beta_s) \\
&= \vphi_s(s^i)^{\omega}.
\end{align*}
Hence $\vphi_s \in L_{\pi_r}(\pi_{nr})$, as claimed.

(b) The proof of this assertion is very similar to that of \cite[Lemma
  6.6]{B}, which is in turn an analogue of \cite[5.4]{Mc1}.

Set $H = \langle s \rangle$. Then $\Omega^t$ acts transitively on
$^{\pi_{r}} H(\pi_{nr}) \subseteq ^{\pi_{r}} G(\pi_{nr}) $, and so the
algebra
\[
L_{\pi_r}(\pi_{nr})^H:= \Map_{\Omega^t}(^{\pi_{r}}H(\pi_{nr}), L^t)
\]
may be identified with a subfield of $L^t$ via identifying $b \in
L_{\pi_r}(\pi_{nr})^H$ with $x_b = b(\1) \in L^t$. We have that
\[
x_{b}^{\sigma^m} = b(s^m), \quad x_{b}^{\phi} = x_b,
\]
and so it follows that $L_{\pi_r}(\pi_{nr})^H$ is the subfield of
$L^t$ consisting of those elements of $L^t$ that are fixed by both
$\phi$ and $\sigma^{|s|}$. This implies that $L_{\pi_r}(\pi_{nr})^H =
L[\vp^{1/|s|}]$ (which in general will not be normal over $L$), and that
the integral closure of $O_L$ in $L_{\pi_r}(\pi_{nr})^H$ is equal to
$O_L[\vp^{1/|s|}]$. Plainly $\beta_s \in O_L[\vp^{1/|s|}]$ (as $|s|$ is
invertible in $O_L$), and the element $\beta_s$ corresponds to the
element $\vphi_s|_H \in L_{\pi_r}(\pi_{nr})^H$.

If we set $\fA(\pi_{nr})_H:= (O_{L^t}H(\pi_{nr}))^{\Omega^t}$, then
for each integer $k$ with $0 \leq k \leq |s|-1$, it is not hard to
check that
\[
\left( \sum_{i=0}^{|s|-1} \zeta_{|s|}^{-ki} s^i \right)^{\phi}
=
 \sum_{i=0}^{|s|-1} \zeta_{|s|}^{-ki} s^i,
\]
and so we see that
\[
\sum_{i=0}^{|s|-1} \zeta_{|s|}^{-ki} s^i \in \fA(\pi_{nr})_H.
\]
A straightforward computation (cf. \cite[5.4]{Mc1}) also shows that
\[
\left( \sum_{i=0}^{|s|-1} \zeta_{|s|}^{-ki} s^i \right) \cdot
\beta_s = \vp^{k/|s|}.
\]
It therefore follows that $\fA(\pi_{nr})_H \cdot \beta_s = O_L[\vp^{1/|s|}]$,
and this in turn implies that
\[
\fA(\pi_{nr}) \cdot \vphi_s = O_{\pi_{r}}(\pi_{nr}),
\]
as asserted.

(c) We have
\begin{align*}
\br_G(\alpha_r)^{\omega} &= \sum_{g \in G} \alpha_r(g)^{\omega} \cdot
g^{-1} \\
&= \sum_{g \in G} \alpha_r(g^{\omega}) \cdot g^{-1} \\
&= \sum_{g \in G} \alpha_r( \pi_r(\omega) \cdot \pi_{nr}(\omega) \cdot
g \cdot \pi_{nr}^{-1}(\omega)) \cdot g^{-1} \\
&= \sum_{g \in G} \alpha_r(g) \cdot \pi_{nr}(\omega)^{-1} \cdot g^{-1} \cdot
\pi_r(\omega) \cdot \pi_{nr}(\omega) \\
&= \pi_{nr}(\omega)^{-1} \cdot \br_G(\alpha_r) \cdot \pi(\omega),
\end{align*}
as claimed.
\end{proof}

\begin{corollary} \label{C:genfac}
For any $\alpha_{r} \in L_{\pi_r}(\pi_{nr})$ and
  $\alpha_{nr} \in L_{\pi_{nr}}$, there is a unique $\alpha \in
  L_{\pi}$ such that
\[
\br_G(\alpha_{nr}) \cdot \br_G(\alpha_r) = \br_G(\alpha).
\]
\end{corollary}

\begin{proof} Proposition \ref{P:phi}(c) implies that, for any $\omega
  \in \Omega^t$, we have
\[
[\br_G(\alpha_{nr}) \cdot \br_G(\alpha_r)]^{\omega} =
\br_G(\alpha_{nr}) \cdot \br_G(\alpha_r) \cdot \pi(\omega),
\]
and so $\br_G(\alpha_{nr}) \cdot \br_G(\alpha_r) \in H(LG)$. As the
map $\br_G$ is bijective, it follows that there is a unique $\alpha
\in \Map(G, L^c)$ such that
\[
\br_G(\alpha_{nr}) \cdot \br_G(\alpha_r) = \br_G(\alpha),
\]
and that $\alpha \in L_{\pi}$.
\end{proof}

\begin{theorem} \label{T:stickfac}
If $a_{nr} \in L_{\pi_{nr}}$ is any normal integral basis generator of
$L_{\pi_{nr}}/L$, then the element $a \in L_{\pi}$ defined by
\begin{equation} \label{E:stickfac1}
\br_G(a_{nr}) \cdot \br_G(\vphi_s) = \br_G(a)
\end{equation}
is a normal integral basis generator of $L_{\pi}/L$.
\end{theorem}

\begin{proof}
The proof of this assertion is very similar to that of the analogous
result in the abelian case described in \cite[(5.7), page 283]{Mc1}.
We first observe that plainly $O_LG \cdot a \subseteq O_{\pi}$ because
$a_{nr} \in O_{\pi_{nr}}$ and $\vphi_s \in
O_{\pi_r}(\pi_{nr})$. Hence, to prove the desired result, it suffices
to show that
\[
\disc(O_LG \cdot a/O_L) = \disc(O_{\pi}/O_L).
\]
This will in turn follow if we show that
\[
\disc(O_{L^{nr}}G \cdot a / O_{L^{nr}}) = \disc(O_{\pi}/O_L) \cdot
O_{L^{nr}}.
\]

Recall (see \eqref{E:weddiso}) that we may write $L_{\pi} \simeq
\bigoplus_{G/\pi(\Omega^t)} L^{\pi}$, where $L^{\pi}$ is a field with
$\Gal(L^{\pi}/L) \simeq \pi(\Omega^t)$. Under this last isomorphism,
the inertia subgroup of $\Gal(L^{\pi}/L)$ is isomorphic to $\langle s
\rangle$. The standard formula for tame field
discriminants therefore yields
\[
\disc(O^{\pi}/O_L) = \vp^{(|s|-1)|\pi(\Omega^t)|/|s|} \cdot O_L
\]
and so we have 
\begin{equation} \label{E:disc1}
\disc(O_{\pi}/O) = \vp^{(|s|-1)|G|/|s|} \cdot O_L.
\end{equation}

Now $\br_G(a_{nr}) \in (O_{L^{nr}}G)^{\times}$, and we see from the
proof of Proposition \ref{P:phi}(b) that
\begin{align*}
O_{L^{nr}}G \cdot a &= O_{L^{nr}}G \cdot \vphi_s \\
&= O_{\pi_r}(\pi_{nr}) \otimes_{O_L} O_{L^{nr}} \\
&\simeq \bigoplus_{G/ \langle s \rangle} O_{L^{nr}} [ \vp^{1/|s|}].
\end{align*}
Since
\[
\disc(O_{L^{nr}}[\vp^{1/|s|}]/O_{L^{nr}}) = \vp^{|s|-1} \cdot O_{L^{nr}},
\]
it follows that
\begin{align*}
\disc(O_{L^{nr}}G \cdot a / O_{L^{nr}}) &= \vp^{(|s|-1)|G|/|s|} \cdot
O_{L^{nr}} \\
&= \disc(O_{\pi}/O) \cdot O_{L^{nr}},
\end{align*}
and this establishes the desired result.
\end{proof}

\begin{remark} \label{R:dbend}
We caution the reader that $L_{\pi_{nr}}$ is \textit{not} in general
equal to the maximal unramified subextension of $L_{\pi}/L$, even when
$L_{\pi}$ is a field.  Suppose, for example, that $L_{\pi}$ is a field, and write
$L_{0}$ for the maximal unramified subextension of $L_{\pi}/L$. Set $f
= [L_0:L]$. Then it is not hard to check that 
\begin {equation} \label{E:unramiso}
L_{\pi_{nr}} \simeq \prod_{i=1}^{|G|/f} L_0,
\end{equation}
and so $L_{\pi_{nr}}$ is a Galois algebra with `core field' $L_0$. If
$\alpha \in O_{L_0}$ is such that $O_{L_0} = O_L[\Gal(L_0/L)] \cdot
\alpha$, then we may take $a_{nr} = (\alpha, 0,\ldots, 0)$ under the
identification given by \eqref{E:unramiso}.

Suppose further that $L$ contains the $|s|$-th roots of unity, and
that $L_{\pi} = L_0 \cdot L(\vp^{1/|s|})$. To ease notation, write
$M:= L(\vp^{1/|s|})$, and set $H = \langle s \rangle$. Then a
calculation similar to (but simpler than) that given in the proof of
Proposition \ref{P:phi}(b) (see also Proposition \ref{P:kmot}) shows
that $O_M = O_L[H] \cdot \beta_s$, and it may be shown by computing
the coefficient of $\1_G$ on the left-hand side of \eqref{E:stickfac1}
that $a = \alpha \cdot \beta_s$, as is of course well-known.  \qed
\end{remark}

\begin{remark} \label{R:ramphi}
Suppose that $s \in G$ with $(|s|, q) = 1$. A straightforward
computation (cf. the proofs of Propositions \ref{P:kmot} and
\ref{P:phi}(b)) shows that for every $\omega \in \Omega_{L^{nr}}$, we
may write
\[
\br_G(\vphi_s)^{\omega} = \br_G(\vphi_s) \cdot \ti{\vphi}_s(\omega)
\]
where $[\ti{\vphi}_s] \in H^1_t(L^{nr}, G)$, and that $\vphi_s$ is a
normal integral basis generator of $L^{nr}_{\ti{\vphi}_s}/L^{nr}$. We
have that $[\ti{\vphi}_{s_1}] = [\ti{\vphi}_{s_2}]$ in $H^1_t(L^{nr},
G)$ if and only if $c(s_1) = c(s_2)$. It is easy to show that every element of
$H^1_t(L^{nr}, G)$ is of the form $[\ti{\vphi}_s]$ for some $s \in G$
with $(|s|, q) =1$ (cf. the proof of Proposition \ref{P:kmot} again).  \qed
\end{remark}
 
\begin{definition} \label{D:stickfac}
Let $a$ be any normal integral basis generator of $L_{\pi}/L$. Theorem
\ref{T:stickfac} implies that we may write
\begin{equation} \label{E:stickfac}
\br_G(a) = u \cdot \br_G(a_{nr}) \cdot \br_G(\vphi_s),
\end{equation}
where $u \in (OG)^{\times}$ and $a_{nr}$ is any normal integral basis
generator of $L_{\pi_{nr}}/L$. This may be viewed as being a
non-abelian analogue of a version of Stickelberger's factorisation of
abelian Gauss sums (see \cite[pages XXXV--XXXVI, and Theorems 135 and
  136]{Hi} and \cite[Introduction]{Mc1}), and so we call
\eqref{E:stickfac} a \textit{Stickelberger factorisation} of
$\br_G(a)$.  \qed
\end{definition}


\section{Local extensions II}  \label{S:local2}

Our goal in this section is to state certain results analogous to,
(but very much simpler than), those in Section \ref{S:local}, for
extensions of $F_v$ where $v$ is an infinite place of $F$. This
section may therefore be viewed as being a `supplement at infinity' to
Section \ref{S:local} (cf. \cite[Chapter I, \S3]{Fr1}). We remind the
reader that, if $v$ is infinite, by convention, we set $O_{F_v}G =
F_vG$ and $H^1_t(F_v, G) = H^1(F_v, G)$.

Suppose first that $v$ is a complex place of $F$. Then
\[
K_0(O_{F_v}G, F^{c}_{v}) = 0, \quad H^1(F_v, G) = 0,
\]
and we set $\Sigma_v(G) = \{ 1\}$. As this case is totally degenerate,
we therefore suppose henceforth in this section that $v$ is real. We
set $L = F_v \simeq \bR$, and for the remainder of this section, we
drop any further reference to $v$ from our notation.

Set $\Gal(L^c/L) = \langle \sigma \rangle$, and fix a primitive fourth
root of unity $\zeta_4 \in L^c$ (cf. the choice of compatible roots of
unity made at the beginning of Section \ref{S:local}), so $L^c =
L(\zeta_4)$.

Write
\begin{equation} \label{E:siginf}
\Sigma(G) := \{ s \in G \mid s^2 = e\}.
\end{equation}
(Note that this set is in fact independent of $v$.) For each $s \in
\Sigma(G)$, we set
\[
\beta_s = \frac{1}{2}(1 + \zeta_4).
\]
Define $\vphi_s \in \Map(G, L^c)$ by
\[
\vphi_{s}(g) = 
\begin{cases}
\sigma^{i}(\beta_s)   &\text{if $g = s^i$;} \\
0     &\text{if $g \notin \langle s \rangle$.}
\end{cases}
\]
Then it is easy to check that
\[
\br_G(\vphi_s) = \beta_s \cdot e + \sigma(\beta_s) \cdot s =
\frac{1}{2}[(1 + \zeta_4) \cdot e + ( 1 - \zeta_4) \cdot s].
\]

\begin{proposition}  \label{P:infres}
Suppose that $\pi \in \Hom(\Omega_L, G)$ with $\pi(\sigma) = s$. Then
$\vphi_s \in L_{\pi}$, and 
\[
L_{\pi} = LG \cdot \vphi_s.
\]
\end{proposition}

\begin{proof}
The first assertion follows directly from the definition of
$\vphi_s$. The second is an immediate consequence of the fact that
$\br_G(\vphi_s) \in (L^cG)^{\times}$, because 
\[
\frac{1}{2} \left( (1 + \zeta_4) \cdot e + ( 1 -
\zeta_4) \cdot s \right) \cdot \frac{1}{2}
\left( (1 - \zeta_4) \cdot e + ( 1 +
\zeta_4) \cdot s \right) = 1.
\]
\end{proof}

\begin{proposition} \label{P:infdet}
Suppose that $\chi \in R_G$, and write
\[
\chi \mid_{\langle s \rangle} = a \cdot \1 + b \cdot \ve,
\]
where $\ve$ denotes the unique non-trivial irreducible character of
$\langle s \rangle$. Then
\[
[\Det(\br_G(\vphi_s))](\chi) = (-1)^{b/2}.
\]

\end{proposition}

\begin{proof}
This follows via a straightforward computation:
\begin{align*}
[\Det(\br_G(\vphi_s))](\chi) &= \1(\br_G(\vphi_s))^{a} \cdot
\ve(\br_G(\vphi_s))^b \\
&= (\beta_s + \sigma(\beta_s))^{a} \cdot (\beta_s - \sigma(\beta_s))^b
\\
&= 1^a \cdot \zeta_{4}^{b} \\
&= (-1)^{b/2}.
\end{align*}
\end{proof}

\begin{remark} In terms of the Stickelberger pairing $\langle
  -,-\rangle_G$ which will be introduced in the next section,
  Proposition \ref{P:infdet} asserts that
\[
[\Det(\br_G(\vphi_s))](\chi) = (-1)^{\langle \chi, s \rangle_G}.
\]
\qed
\end{remark}


\section{The Stickelberger pairing} \label{S:stickpair}

\begin{definition} \label{D:stickpair} 
The \textit{Stickelberger pairing} is a $\bQ$-bilinear pairing
\begin{equation} \label{E:stickpair}
\langle -,- \rangle_G: \bQ R_G \times \bQ G \to \bQ
\end{equation}
that is defined as follows.

Let $\zeta_{|G|}$ be a fixed, primitive $|G|$-th root of unity
(cf. the conventions established at the beginning of Section
\ref{S:local}), and suppose first that $G$ is abelian. Then if $\chi
\in \Irr(G)$ and $g \in G$, we may write $\chi(g) = \zeta_{|G|}^r$ for some
integer $r$. We define
\[
\langle \chi, g \rangle_G = \left\{\frac{r}{|G|} \right\},
\]
where $\{ x\}$ denotes the fractional part of $x \in \bQ$, and we extend this
to a pairing on $\bQ R_G \times  \bQ G$ via linearity. For arbitrary
finite $G$, the Stickelberger pairing is defined via reduction to the
abelian case by setting 
\[
\langle \chi, g \rangle_G = \langle \chi \mid_{\langle g \rangle}, g
\rangle_{\langle g \rangle}.
\]
It is easy to check that both definitions agree when $G$ is abelian.
\qed
\end{definition}

We shall now explain a different way of expressing the Stickelberger
pairing using the standard inner product on $R_G$. In order to do
this, we must introduce some further notation.

For each $s \in G$, we set $m_s:= |G|/|s|$. We define a character
$\xi_s$ of $\langle s \rangle$ by $\xi_s(s^i) = \zeta_{|G|}^{im_s}$;
so $\xi_s$ is a generator of the group of irreducible characters of
$\langle s \rangle$. Then it follows from Definition \ref{D:stickpair}
that
\[
\langle \xi_{s}^{\alpha}, s^{\beta} \rangle_{\langle s \rangle} =  
\left\{ 
\frac{\alpha \beta}{|s|} 
\right\} .
\]
Define
\[
\Xi_s:= \frac{1}{|s|} \sum_{j=1}^{|s|-1} j \xi_{s}^{j}.
\]

\begin{proposition} \label{P:stickpair}
Let $(-,-)_G$ denote the standard inner product on $R_G$, and suppose
that $\chi \in R_G$, $s \in G$. Then we have
\[
( \chi, \Ind_{\langle s \rangle}^{G}(\Xi_s))_G = \langle \chi, s
\rangle_G.
\]
\end{proposition}

\begin{proof}
Suppose that
\[
\chi \mid_{\langle s \rangle} = \sum_{j =0}^{|s|-1} a_j \xi_{s}^{j},
\]
where $a_j \in \bZ$ for each $j$. Then we have
\begin{align*}
\langle \chi, s \rangle_G &= \sum_{j=0}^{|s|-1} a_j \langle \xi_{s}^{j}, s
\rangle_{\langle s \rangle} \\
&= \sum_{j =0}^{|s|-1} a_j \left\{\frac{j}{|s|} \right\} \\
&= \frac{1}{|s|} \sum_{j=0}^{|s|-1} a_j j.
\end{align*}

On the other hand, via Frobenius reciprocity, we have
\begin{align*}
( \chi, \Ind_{\langle s \rangle}^{G}(\Xi_s))_G &= ( \chi
  \mid_{\langle s \rangle}, \Xi(s))_{\langle s \rangle} \\
&= \left( \sum_{j=0}^{|s|-1} a_j \xi_s^j, \frac{1}{|s|} \sum_{j=0}^{|s|-1}
  j \xi_s^j \right)_{\langle s \rangle} \\
&= \frac{1}{|s|} \sum_{j=0}^{|s|-1} a_j j \\
&= \langle \chi, s \rangle_G,
\end{align*}
and this establishes the desired result.
\end{proof}

In order to apply Poposition \ref{P:stickpair}, we shall require the
following result concerning traces of sums of roots of unity.

\begin{lemma} \label{L:cyclem}
Let $n>1$ be an integer, and suppose that $\zeta$ is any primitive
$n$-th root of unity.  Write
\[
y := \sum_{i =1}^{n-1} i \cdot \zeta^i.
\]
Then
\[
\Tr_{\bQ(\zeta)/\bQ}(y) = -\frac{1}{2}n \phi(n),
\]
where $\phi$ is the Euler $\phi$-function. In particular,
$\Tr_{\bQ(\zeta)/\bQ}(y) \neq 0$.
\end{lemma}

\begin{proof}
Each $\zeta^i$ is a primitive $d$-th root of unity for some
divisor $d$ of $n$, and so it follows that
\[
y = \sum_{d \mid n} \sum_{\stackrel{1 \leq r \leq d-1}{(r,d) =1}}
\frac{nr}{d} \zeta^{nr/d}.
\]
If $d|n$, then applying M\"obius inversion to the identity $x^d - 1 =
\prod_{m \mid d} \Phi_{m}(x)$ (where $\Phi_{m}(x)$ denotes the $m$-th
cyclotomic polynomial) yields $\Phi_{m}(x) = \prod_{m \mid d}
(x^{m.}-1)^{\mu(d/m)}$, whence it is not hard to show that
$\Tr_{\bQ(\ve)/\bQ}(\ve) = \mu(d)$ for any primitive $d$-th root $\ve$
of unity. Hence $\Tr_{\bQ(\zeta)/\bQ}(\ve) = \phi(n) \mu(d)/\phi(d)$,
and so we have
\begin{align*}
\Tr_{\bQ(\zeta)/\bQ}(y) &=  
\sum_{d \mid n} \sum_{\stackrel{1 \leq r \leq d-1}{(r,d) =1}}
\frac{nr}{d} \Tr_{\bQ(\zeta)/\bQ}(\zeta^{nr/d})  \\
&= n \sum_{d \mid n} \frac{\mu(d)}{d} \frac{\phi(n)}{\phi(d)} s(d),
\end{align*}
where
\[
s(d) =
\begin{cases}
1 &\text{if $d = 1$;} \\
\sum_{\stackrel{1 \leq i \leq d-1}{(i,d)
    =1}} i &\text{if $d>1$}. \\
\end{cases}
\]
It is well known that
\[
s(d) = \frac{1}{2}d \phi(d)
\]
for any integer $d>1$ (see e.g. \cite[Theorem 7.7]{Bu07}). It therefore
follows that
\begin{align*}
\Tr_{\bQ(\zeta)/\bQ}(y) &= \frac{1}{2} n \phi(n) \sum_{\stackrel{d \mid n}{d>1}}
\mu(d) \\
&= -\frac{1}{2}n \phi(n),
\end{align*}
as claimed.
\end{proof}

We can now state the following corollary to Proposition
\ref{P:stickpair}.

\begin{corollary} \label{C:stick}
Suppose that $s_1$ and $s_2$ are elements of $G$. 

(i) If $c(s_1) = c(s_2)$, then $\langle \chi,
s_1 \rangle_G = \langle \chi, s_2 \rangle_G$ for all $\chi \in \bQ
R_G$.

(ii) If $\langle \chi, s_1 \rangle_G = \langle \chi, s_2 \rangle_G$
for all $\chi \in \bQ R_G$, then $\langle s_1 \rangle$ is conjugate to
$\langle s_2 \rangle$ in $G$.

(iii) We have that $\langle \chi, s_1 \rangle_G = 0$ for all $\chi \in
\bQ R_G$ if and only if $s_1 = e$.
\end{corollary}

\begin{proof}
(i) Let $\chi \in R_G$ and $s \in G$. It follows from the definition of
the Stickelberger pairing that for fixed $\chi$, the value of $\langle
\chi, s \rangle_G$ depends only upon the conjugacy class $c(s)$ of $s$
in $G$. Hence, if $c(s_1) = c(s_2)$, then $\langle \chi, s_1 \rangle_G
= \langle \chi, s_2 \rangle_G$ for all $\chi \in \bQ R_G$.

(ii) To show this we use Proposition \ref{P:stickpair}. We first note
that a straightforward computation shows that the degree of the
virtual character $\Ind_{\langle s \rangle}^{G}(\Xi_{s})$ is equal to
$|G|(|s|-1)/2|s|$, and so we see that $\Ind_{\langle s \rangle}^{G}
(\Xi_{s})$ determines $|s|$. Next, we remark that 
If $\{ t_i \}$ is a set of representatives of
$G/\langle s \rangle$, then for each $g \in G$, we have
\begin{equation} \label{E:indform}
[\Ind^{G}_{\langle s \rangle}(\Xi_{s})](g) = \sum_{t^{-1}_{i}g t_i
  \in \langle s \rangle} \xi_s(t^{-1}_{i}g t_i),
\end{equation}
and so  the character $\Ind^{G}_{\langle s \rangle}(\Xi_{s})$
vanishes on all elements of $G$ that are not conjugate to an element
of $\langle s \rangle$.

Proposition \ref{P:stickpair} implies that under our hypotheses,
$\Ind^{G}_{\langle s_1 \rangle}(\Xi_{s_1}) = \Ind^{G}_{\langle s_2
  \rangle}(\Xi_{s_2})$. Hence, to prove the desired result, it
suffices to show that $[\Ind^{G}_{\langle s_1
    \rangle}(\Xi_{s_1})](s_1) \neq 0$, because then
\[
[\Ind^{G}_{\langle s_2 \rangle}(\Xi_{s_1})](s_1) = [\Ind^{G}_{\langle
    s_1 \rangle}(\Xi_{s_1})](s_1) \neq 0,
\]
which implies (since $|s_1| = |s_2|$) that $s_1$ is conjugate to a
generator of $\langle s_2 \rangle$.

Now if $s_{1}^{a}$ is any generator of $\langle s_1 \rangle$, then
$\xi_{s_1}(s_{1}^{a})$ is a primitive $|s_1|$-th root of unity, and we have
\[
\xi_{s_1}(s_{1}^{a}) = \sum_{i=1}^{|s_{1}|-1} i
\xi_{s_{1}}(s_{1}^{a})^{i}.
\]
Hence if $\zeta$ denotes any primitive $|s_1|$-th root of unity,
Lemma \ref{L:cyclem} implies that
\[
\Tr_{\bQ(\zeta)\/\bQ}(\xi_{s_1}(s_{1}^{a})) = -\frac{1}{2} |s_1|
\phi(|s_1|).
\]
It follows from \eqref{E:indform} that
$\Tr_{\bQ(\zeta)/\bQ}[\Ind_{s_1}^{G}(\Xi_{s_1})](s_1)$ is equal to a
non-zero multiple of $-|s_1| \phi(|s_1|)/2$, and so is non-zero. This
in turn implies that $[\Ind_{s_1}^{G}(\Xi_{s_1})](s_1)$ is also
non-zero, thereby establishing the desired result.

(iii) Proposition \ref{P:stickpair} implies that $\langle \chi, s_1
\rangle_G = 0$ for all $\chi \in \bQ R_G$ if and only if
$(\Ind_{\langle s_1 \rangle}^{G}(\Xi_{s_1}), \chi)_G = 0$ for all
$\chi \in \bQ R_G$. The latter condition holds if and only if 
$\Ind_{\langle s_1 \rangle}^{G}(\Xi_{s_1}) = 0$ and this happens if
and only if $s_1 = e$.
\end{proof}

\begin{remark} \label{remstick} 
(a) The converse to Corollary \ref{C:stick}(i) does not hold in general,
e.g. it fails for the dihedral group $D_{2p}$ of order $2p$, where $p>3$
is a prime. (See \cite[Chapter 3]{Siv} or \cite{Siv1} for an
explicit description of the Stickelberger pairing in this case.)

(b) Let $\chi_1,\ldots, \chi_d$ (respectively $c_1,\ldots, c_d$) be the
set of irreducible characters (respectively conjugacy classes) of
$G$. We refer the reader to \cite{BFKMST} for computations and
conjectures concerning the rank of the $d \times d$-matrix
$[\langle \chi_i, c_j \rangle_G]$ associated to the Stickelberger
pairing $\langle -,- \rangle_G$ when $G$ is cyclic.
\qed
\end{remark}


\section{The Stickelberger map and transpose
  homomorphisms} \label{S:stickhom} 

\subsection{The Stickelberger map}

\begin{definition}
The \textit{Stickelberger map}
\begin{equation} \label{E:stickmap}
\Theta =\Theta_G: \bQ R_G \to \bQ G
\end{equation}
is defined by
\[
\Theta(\chi) = \sum_{g \in G} \langle \chi, g \rangle_G \cdot g.
\]
\qed
\end{definition}

We write $G(-1)$ for the set $G$ endowed with an action of $\Omega_F$
via the inverse cyclotomic character. Note that in general, for
non-abelian $G$, this $\Omega_F$-action is not an action on $G$ via
group automorphisms; it is only an action on the set $G$. However, it
does induce an action on the additive group $\bQ G(-1)$, which is all
that we shall require.

The following proposition summarises some basic properties of the
Stickelberger map.

\begin{proposition} \label{P:stickmap}
(a) We have that $\Theta(\chi) \in Z(\bQ G)$ for all $\chi \in
  R_G$, i.e. in fact
\[
\Theta: \bQ R_G \to Z(\bQ G).
\]

(b) Suppose that $\chi \in R_G$. Then $\Theta(\chi) \in \bZ G$ if and
only if $\chi \in A_G$. Hence $\Theta$ induces a homomorphism $A_G \to
\bZ G$.

(c) The map
\[
\Theta: \bQ R_G \to \bQ G(-1)
\]
is $\Omega_F$-equivariant.
\end{proposition}

\begin{proof} The proofs of these assertions for arbitrary $G$ are
  essentially the same as those in the case of abelian $G$.  See
  \cite[Propositions 4.3 and 4.5]{Mc1}.

(a) It follows from the definition of the Stickelberger pairing that
 if $\chi \in R_G$ and $g \in G$, then $\langle \chi, g \rangle_G$
 is determined by the conjugacy class $c(g)$ of $g$ in $G$. This
 implies that $\Theta(R_G) \subseteq Z(\bQ G)$, as claimed.

(b) Suppose that $\chi \in R_G$ and $g \in G$. Write
\[
\chi \mid_{\langle g \rangle} = \sum_{\eta} a_{\eta} \eta,
\]
where the sum is over irreducible characters of $\langle g \rangle$,
and set $\zeta_{|g|}:= \zeta_{|G|}^{|G|/|g|}$. Then
\begin{align*}
(\det(\chi))(g) &= \det(\chi \mid_{\langle g \rangle})(g) \\
&= \prod_{\eta} \eta(g)^{a_{\eta}} \\
&= \prod_{\eta} \zeta_{|g|}^{|g| \langle a_{\eta} \eta, g \rangle_{\langle g \rangle}} \\
&= \zeta_{|g|}^{|g| \sum_{\eta}\langle a_{\eta} \eta,
    g\rangle_{\langle g \rangle}} \\
&= \zeta_{|g|}^{|g| \langle \chi, g \rangle_{G}}.
\end{align*}
It now follows that $\langle \chi, g \rangle_G \in \bZ$
for all $g \in G$ if and only if $\chi \in \Ker(\det) = A_G$, as
required.

(c) Let $\kappa$ denote the cyclotomic character of $\Omega_F$, and
suppose that $\chi \in R_G$ is of degree one. Then, for each $g \in
G$ and $\omega \in \Omega_F$, we have
\begin{equation*} 
\chi^{\omega}(g) = \chi(g^{\kappa(\omega)}), 
\end{equation*}
and so 
\begin{equation} \label{E:d37}
\langle \chi^{\omega}, g \rangle_G = \langle \chi,
g^{\kappa(\omega)} \rangle_G.
\end{equation}
It follows via bilinearity that \eqref{E:d37} holds for all $\chi \in
R_G$ and all $g \in G$. Hence, if we view $\Theta(\chi)$ as being an
element of $\bQ G(-1)$, then
\begin{align*}
\Theta(\chi^{\omega}) &= \sum_{g \in G} \langle \chi^{\omega}, g
\rangle_G \cdot g \\
&= \sum_{g \in G} \langle \chi, g^{\kappa(\omega)} \rangle_G \cdot g \\
&= \sum_{g \in G} \langle \chi, g \rangle_G \cdot
g^{\kappa^{-1}(\omega)}\\
&= \Theta(\chi)^{\omega}.
\end{align*}
\end{proof}

\subsection{Transpose Stickelberger homomorphisms}
We see from Proposition \ref{P:stickmap} that dualising the
homomorphism
\[
\Theta: A_G \to Z(\bZ G)
\]
and twisting by the inverse cyclotomic character yields an
$\Omega_F$-equivariant \textit{transpose Stickelberger
  homomorphism}
\begin{equation} \label{E:mstickhom}
\Theta^t: \Hom(Z(\bZ G(-1)), (F^c)^{\times}) \to \Hom(A_G, (F^c)^{\times}).
\end{equation}
Composing \eqref{E:mstickhom} with the sequence of
homomorphisms
\begin{equation} \label{E:mstickhomi}
\Hom(A_G, (F^c)^{\times})  \xrightarrow{\sim}
Z(F^cG)^{\times}/G^{ab} \to
\frac{\Det(F^cG)^{\times}}{\Det(O_FG)^{\times}} \to K_0(O_FG, F^c),
\end{equation}
(where the first arrow is given by \eqref{E:zres}, the second via (the
inverse of) \eqref{E:nrdiso}, and the third is via the homomorphism
$\partial^{1}$ of \eqref{E:rkes}) yields a homomorphism
\begin{equation} \label{E:stickhom}
K\Theta^t: \Hom(Z(\bZ G(-1)), (F^c)^{\times}) \to
K_0(O_FG, F^c).
\end{equation}

Hence, if we write $\cC(G(-1))$ for the set of conjugacy classes of
$G$ endowed with $\Omega_F$-action via the inverse cyclotomic
character, and set
\begin{align*}
\Lambda(O_FG):= \Hom_{\Omega_F}(Z(\bZ G(-1)), O_{F^{c}}) &=
\Map_{\Omega_F}(\cC(G(-1)), O_{F^{c}})\\
&= Z(O_{F^{c}}[G(-1)])^{\Omega_F};\\
\Lambda(FG):= \Hom_{\Omega_F}(Z(\bZ G(-1)), F^c) &= 
\Map_{\Omega_F}(\cC(G(-1)), F^c)\\
&= Z(F^c[G(-1)])^{\Omega_F},
\end{align*}
then $K\Theta^t$ induces a homomorphism (which we denote by the same
symbol):
\[
K\Theta^t: \Lambda(FG)^{\times} \to K_0(O_FG, F^c).
\]

For each place $v$ of $F$, we may apply the discussion above
with $F$ replaced by $F_v$ to obtain  local versions
\begin{equation}
\Theta^t_v: \Hom(Z(\bZ G(-1)), (F_v^c)^{\times}) \to \Hom(A_G, (F_v^c)^{\times})
\end{equation}
and
\begin{equation} \label{E:locstick}
K\Theta_{v}^{t}: \Lambda(F_vG)^{\times} \to K_0(O_{F_v}G, F_v^c)
\end{equation}
of the maps $\Theta^t$ and $K\Theta^t$ respectively. The homomorphism
$\Theta^t$ commutes with local completion, and $K\Theta^t$ commutes
with the localisation maps
\[
\lambda_v: K_0(O_FG, F^c) \to K_0(O_{F_v}G, F_v^c).
\]

\begin{definition}
We define the group of ideles $J(\Lambda(FG))$ of $\Lambda(FG)$ to be
the restricted direct product over all places $v$ of $F$ of the groups
$\Lambda(F_vG)^{\times}$ with respect to the subgroups
$\Lambda(O_{F_v}G)^{\times}$. \qed
\end{definition}

For all finite places $v$ of $F$ not dividing the order of $G$, as
$O_{F_{v}}G$ is an $O_{F_v}$-maximal order in $F_vG$, we have that
(cf. Proposition \ref{P:agdet}(ii))
\[
\Theta_v^t(\Lambda(O_{F_v}G)) \subseteq \Hom_{\Omega_{F_v}}(A_G,
  (O_{F_v^c})^{\times}) = \Det(\cH(O_{F_v}G)),
\]
and so 
\[
K\Theta_v^t(\Lambda(O_{F_v}G)) \subseteq K_0(O_{F_v}G, O_{F_v^c}).
\]
It follows that the homomorphisms $\Theta^t_v$ combine to yield an
idelic transpose Stickelberger homomorphism
\begin{equation} \label{E:idstick}
K\Theta^t: J(\Lambda(FG)) \to J(K_0(O_FG, F^c)).
\end{equation}

We shall see in the next subsection that the idelic homomorphism
$K\Theta^t$ is closely related to the homomorphism
\[
\bpsi^{id}: J(H^{1}_{t}(F,G)) \to J(K_0(O_FG, F^c))
\]
of Definition \ref{D:psid}.

\subsection{Prime $\bF$-elements}
\begin{definition} \label{D:fprime}
Let $v$ be a place of $F$. For each element $s \neq e$ of
$\Sigma_v(G)$ (see Definition \ref{D:admiss} and \eqref{E:siginf}),
define $f_{v,s} \in \Lambda(F_vG)^{\times}$ by
\begin{equation}
f_{v,s}(c) =
\begin{cases}
-1 &\text{if $v$ is real and $c = c(s)$;} \\
\vp_v, &\text{if $v$ is finite and $c=c(s)$;} \\ 
1, &\text{otherwise.}
\end{cases}
\end{equation}
Observe that $f_{v,s}$ is $\Omega_{F_v}$-equivariant because $s \in
\Sigma_v(G)$ and so $\Omega_{F_v}$ fixes $c(s)$ when $s$ is viewed as
an element of $G(-1)$. The element $f_{v,s}$ depends only upon the
conjugacy class $c(s)$ of $s$. For all places $v$ of $F$, we define $f_{v,e} \in
(\Lambda(F_vG))^{\times}$ to be the constant function $f_{v,e} = 1$.

Write
\begin{equation*}
\bF_v:= \{f_{v,s}\, \mid \, \text{$s \in \Sigma_v(G)$}\},
\end{equation*}
and define the subset $\bF \subset J(\Lambda(FG))$ of prime
$\bF$-elements by
\begin{equation*}
f \in \bF \iff \text{$f \in J(\Lambda(FG))$ and $f_v \in \bF_v$ for
all places  $v$ of $F$.}
\end{equation*}
Following \cite[Definition 7.1]{B}, we define the \textit{support}
$\Supp(f)$ of $f \in \bF$ to be set of all places $v$ of $F$ for which
$f_v \neq 1$. We say that $f$ is \textit{full} if, for each $s \in G$
there is a place $v$ with $f_v = f_{v,s}$.  \qed
\end{definition}

Our interest in the set $\bF$, as well as the relationship between
$K\Theta^t$ and $\bpsi^{id}$, is explained by the following result.

\begin{proposition} \label{P:reseq} 
Let $v$ be a place of $F$.

(a) For each $s \in \Sigma_v(G)$, we have
\[
\Det(\br_G(\vphi_{v,s})) = K\Theta^t_v(f_{v,s})
\]
in $K_0(O_{F_v}G, F_v^c)$.

(b) Suppose that $s_1, s_2 \in \Sigma_v(G)$ with
\begin{equation} \label{E:eqdet}
\Det(\br_G(\vphi_{v,s_1})) = \Det(\br_G(\vphi_{v,s_2})).
\end{equation}
Then $\langle s_1 \rangle$ is conjugate in $G$ to $\langle s_2
\rangle$.

(c) Suppose that $v$ is finite. Let $\pi_1, \pi_2 \in
\Hom(\Omega_{F_v},G)$ with $[\pi_i] \in H^{1}_{t}(F_v, G)$ for each
$i$, and set $s_i = \pi_i(\sigma_v)$ (cf. \eqref{E:rmap}). Let $a_i$ be
a normal integral basis generator of $F_{v,\pi_{i}}/F_v$, and let
\[
\br_G(a_i) = u_i \cdot \br_G(a_{i,nr}) \cdot \br_G(\vphi_{s_{i}})
\]
be a Stickelberger factorisation of $\br_G(a_i)$ (see Definition
\ref{D:stickfac}). Suppose that
\begin{equation} \label{E:detunit}
\Det(\br_G(a_1)) \cdot \Det(\br_G(a_2))^{-1} \in
\Det((O_{F_{v}^{c}}G)^{\times}).
\end{equation}
Then
\[
\Det(\br_G(\vphi_{s_1})) = \Det(\br_G(\vphi_{s_2}))
\]
and for some integer $m$ and some $h \in G$, the equality
\[
\pi_1(\omega) = h \cdot \pi_2(\omega)^m \cdot h^{-1}
\]
holds for all $\omega \in I_v$.
\end{proposition}

\begin{proof} 
(a) The proof of this assertion is very similar to that of
  \cite[Proposition 5.4]{Mc1}.

It suffices to show that the equality
\[
\Det(\br_G(\vphi_{v,s})) = \Theta^t_v(f_{v,s})
\]
holds in $\Hom(A_G, (F^c_v)^{\times})$.

Let $\chi \in R_G$, and write
\[
\chi \mid_{<s>} = \sum_{\eta} a_{\eta} \eta,
\]
where the sum is over irreducible characters $\eta$ of $\langle s
\rangle$. 

Suppose first that $v$ is finite. Using \eqref{E:locres}, we see that
(cf. \cite[Proposition 5.4]{Mc1})
\begin{align}  \label{E:phivalue}
[\Det(\br_G(\vphi_{v,s}))](\chi) &= \prod_{\eta} \left( \sum_{i=0}^{|s|-1}
\sigma_{v}^{i}(\beta_s) \eta(s^{-i}) \right)^{a_{\eta}}  \notag \\
&= \vp_{v}^{\langle \sum_{\eta}a_{\eta} \eta, s \rangle_{\langle s \rangle}}  \notag \\
&= \vp_{v}^{\langle \chi, s \rangle_{G}},
\end{align}
and so it follows that
\[
[\Det(\br_G(\vphi_{v,s}))](\alpha) = \vp_{v}^{\langle \alpha, s \rangle_{G}}
\]
for all $\alpha \in A_G$.

If $v$ is real, then then the proof of Proposition \ref{P:infdet}
shows directly that
\[
[\Det(\br_G(\vphi_{v,s}))](\chi) = (-1)^{\langle \chi, s \rangle_{G}},
\]
and so we have
\[
[\Det(\br_G(\vphi_{v,s}))](\alpha) = (-1)^{\langle \alpha, s \rangle_{G}}
\]
for all $\alpha \in A_G$ in this case also.

Now suppose that $v$ is either finite or real. If $\alpha \in A_G$,
then we have
\begin{align*}
(\Theta_v^t(f_{v,s}))(\alpha) &= f_{v,s}(\Theta(\alpha)) \\
&= f_{v,s} \left( \sum_{g \in G} \langle \alpha, g \rangle_G \cdot g \right) \\
&= \prod_{g \in G} f_{v,s}(g)^{\langle \alpha, g \rangle_{G}} \\
&= 
\begin{cases}
\vp_{v}^{\langle \alpha, s \rangle_{G}}, &\text{if $v$ is finite;} \\
(-1)^{\langle \alpha, s \rangle_{G}}, &\text{if $v$ is real.}
\end{cases}
\end{align*}

The desired result now follows.

(b) The proof of (a) above shows that if \eqref{E:eqdet} holds, then
\[
\langle \chi, s_1 \rangle_G = \langle \chi, s_2 \rangle_G
\]
for every $\chi \in R_G$. It therefore follows from Corollary
\ref{C:stick} that $\langle s_1 \rangle$ is conjugate in $G$ to
$\langle s_2 \rangle$.

(c) Observe that \eqref{E:detunit} holds if and only if
\begin{equation} \label{E:phiunit}
\Det(\br_{G}(\vphi_{s_{1}})) \cdot \Det(\br_{G}(\vphi_{s_{2}})^{-1})
\in \Det((O_{F_{v}^{c}}G)^{\times}),
\end{equation}
and the proof of part (a) (see \eqref{E:phivalue}) implies that
\eqref{E:phiunit} holds if and only if 
\[
\Det(\br_G(\vphi_{s_1})) = \Det(\br_G(\vphi_{s_2})).
\] 
Part (b) therefore implies that $\langle s_1 \rangle$ and
$\langle s_2 \rangle$ are conjugate. Hence
\[
s_1 = h \cdot s_{2}^{m} \cdot h^{-1}
\]
for some $m \in \bZ$ and $h \in G$, and so
\[
\br_G(\vphi_{s_1}) = h \cdot \br_{G}(\vphi_{s^{m}_{2}}) \cdot h^{-1}
\]
(see \eqref{E:conjres}). 

For any $\omega \in \Omega_{F_{v}^{nr}}$, we have
\[
\pi_i(\omega) = \br_G(a_i)^{-1} \cdot \br_G(a_i)^{\omega} =
\br_G(\vphi_{s_i})^{-1} \cdot \br_G(\vphi_{s_i})^{\omega}.
\]
Applying the map $F^{c}_{v}G \to F^{c}_{v}G$ defined by $\sum_g a_g g
\mapsto \sum_g a_g g^m$ to this equality (when $i =
2$) yields
\[
\pi_2(\omega)^m = \br_G(\vphi_{s^{m}_{2}})^{-1} \cdot
\br_G(\vphi_{s^{m}_{2}})^{\omega}.
\]
The final assertion now follows.
\end{proof}

\subsection{The Stickelberger pairing revisited} \label{SS:Stickred}
In this subsection we shall briefly describe an alternative definition
of the Stickelberger pairing that involves a direct connection with
resolvends of local normal integral basis generators. This will not be
used in the sequel.

Let $v$ be a finite place of $F$. There is a natural pairing
\begin{equation} \label{E:modpair}
\{-,-\}_{G,v}: \Irr(G) \times H^1(F^{nr}_{v},G) \to \bQ/\bZ;\quad (\chi,
   [\pi]) \mapsto [v(\Det(\br_G(a(\pi)))(\chi))],
\end{equation}
where $a(\pi)$ is any normal basis generator of
$F^{nr}_{v,\pi}/F^{nr}_{v}$. Recall that every element of
$H^1_t(F^{nr}_{v},G)$ is of the form $\ti{\vphi}_{v,s}$ for some $s
\in G$ with $v \nmid |s|$ (see Remark \ref{R:ramphi}).  The
restriction of $\{-,-\}_{G,v}$ to $\Irr(G) \times H^1_t(F^{nr}_{v}, G)$
yields a refined pairing
\begin{equation} \label{E:refmodpair}
\{-,-\}_{G,v}^{(1)}: \Irr(G) \times H^1_t(F^{nr}_{v},G) \to \bQ;\quad (\chi,
\ti{\vphi}_{v,s}) \mapsto v(\Det(\br_G(\vphi_{v,s}))(\chi)).
\end{equation}
This leads to the following definition.

\begin{definition} \label{D:modstick}
Suppose that $v$ is finite and that $v \nmid |G|$. We define a pairing
\begin{equation} \label{E:modstick}
[-,-]_{G,v}: \Irr(G) \times G \to \bQ;\quad (\chi, g) \mapsto
v(\Det(\br_G(\vphi_{v,g}))(\chi)),
\end{equation}
and we extend this to a pairing on $\bQ R_G \times \bQ G$ via linearity.
\qed
\end{definition}

\begin{proposition} \label{P:funcstick}
Suppose that $v$ is finite and that $v \nmid |G|$. Then for each $\chi
\in \Irr(G)$ and $g \in G$, we have
\begin{equation} \label{E:funcstick}
[\chi, g]_{G,v} = [\chi \mid_{\langle g \rangle}, g]_{\langle g\rangle, v}.
\end{equation}
\end{proposition}

\begin{proof}
Set $H:= \langle g \rangle$. The property \eqref{E:funcstick} is a
direct consequence of the fact that the restriction map $R_G \to R_H$
induces a homomorphism $\Hom(R_H, (F^c_v)^{\times}) \to \Hom(R_G,
(F^c_v)^{\times})$ such that the following diagram commutes:
\[
\begin{CD}
(F_v^cH)^{\times} @>{\subseteq}>> (F^c_vG)^{\times} \\
@VV{\Det}V                    @VV{\Det}V \\
\Hom(R_H, (F^c_v)^{\times}) @>>> \Hom(R_G, (F^c_v)^{\times})\\
\end{CD}
\]
(see e.g. \cite[p. 436]{Fr76} or \cite[p. 118]{Fr1}).
\end{proof}

\begin{proposition} \label{P:eqstick}
Suppose that $v$ is finite and that $v \nmid |G|$. Then for each $\chi \in \Irr(G)$
and $g \in G$, we have
\begin{equation} \label{E:eqstick}
[\chi, g]_{G,v} = \langle \chi, g \rangle_G.
\end{equation}
In particular, $[-,-]_{G,v}$ is independent of our choice of $v$.
\end{proposition}

\begin{proof} 
Proposition \ref{P:funcstick} implies that we may assume that $G$ is
cyclic. The equality \eqref{E:eqstick} may then be established via an
argument identical to that used in the proof of Proposition
\ref{P:reseq}(a) (see also \cite[Proposition 5.4]{Mc1}).
\end{proof}


\section{Modified ray class groups}  \label{S:modray}

\begin{definition}
Let $\fa$ be an integral ideal of $O_F$. For each finite place $v$ of
$F$, recall that
\[
U_{\fa}(O_{F^c_v}):= (1 + \fa O_{F^c_v}) \cap
(O_{F^c_v})^{\times}.
\]
We define
\begin{equation*}
U_{\fa}'(\Lambda(O_{F_v}G)) \subseteq \Lambda(F_vG)^{\times} =
\Map_{\Omega_{F_v}}(\cC(G(-1)), (F_{v}^{c})^{\times})
\end{equation*}
by
\begin{equation*}
U_{\fa}'(\Lambda(O_{F_v}G)):= \left\{ g_v \in \Lambda(F_vG)^{\times} \mid
g_v(c) \in U_{\fa}(O_{F^c_v}) \quad \forall c \neq 1 \right\}
\end{equation*}
(with $g_v(1)$ allowed to be arbitrary).

Set
\begin{equation*}
U_{\fa}'(\Lambda(O_FG)):= \left( \prod_v U_{\fa}'(\Lambda(O_{F_v}G))
\right) \cap J(\Lambda(FG)).
\end{equation*}
\qed
\end{definition}

\begin{definition}
For each real place $v$ of $F$, we define
\[
\Lambda(F_vG)^{\times}_{+}:= \{ g_v \in \Lambda(F_vG)^{\times} \mid g_v(c)
\in \bR^{\times}_{>0}\, \text{for all $c \in \cC(G(-1))$} \}
\]
(with $g_v(1)$ allowed to be arbitrary).

If $v$ is complex, we set $\Lambda(F_vG)_{+}^{\times} :=
\Lambda(F_vG)^{\times}$.  We define
\[
U'_{\infty}(\Lambda(O_FG)):= \left( \prod_{v | \infty}
\Lambda(FG)^{\times} \right) \cap J(\Lambda(FG)),
\]
and
\[
U'_{\infty}(\Lambda(O_FG))_{+}:= \left( \prod_{v | \infty}
\Lambda(FG)_{+}^{\times} \right) \cap J(\Lambda(FG)).
\]
\qed
\end{definition}

\begin{definition}
The \textit{modified ray class group modulo} $\fa$ of $\Lambda(O_FG)$ is
defined by
\begin{equation*}
\Cl_{\fa}'(\Lambda(O_FG)):= \frac{J(\Lambda(FG))}{\Lambda(FG)^{\times} \cdot
  U_{\fa}'(\Lambda(O_FG)) \cdot U_{\infty}'(\Lambda(O_FG))}.
\end{equation*}

The \textit{modified narrow ray class group modulo} $\fa$ is
defined by
\begin{equation*}
\nCl(\Lambda(O_FG)):= \frac{J(\Lambda(FG))}{\Lambda(FG)^{\times} \cdot
  U_{\fa}'(\Lambda(O_FG)) \cdot U_{\infty}'(\Lambda(O_FG))_{+}}.
\end{equation*}

We refer to the elements of $\Cl'_{\fa}(\Lambda(O_FG))$ (respectively
$\nCl(\Lambda(O_FG))$) as the \textit{modified ray classes}
(respectively \textit{modified narrow ray classes}) of $\Lambda(O_FG)$
modulo $\fa$.
\qed
\end{definition}

\begin{remark}
Fix a set of representatives $T$ of $\Omega_F\backslash \cC(G(-1))$, and
for each $t \in T$, let $F(t)$ be the smallest extension of $F$ such
that $\Omega_{F(t)}$ fixes $t$. Then the Wedderburn decomposition of
$\Lambda(FG)$ is given by
\begin{equation} \label{E:wiso} 
\Lambda(FG) = \Map_{\Omega_F}(\cC(G(-1)), F^c) \simeq \prod_{t \in T}F(t),
\end{equation}
where the isomorphism is induced by evaluation on the elements of $T$.

The group $\Cl_{\fa}'(\Lambda(O_FG))$ (respectively
$\nCl(\Lambda(O_FG))$) above is finite, and is isomorphic to
the product of the ray class groups $\Cl_{\fa}(O_{F(t)})$
(respectively the narrow ray class groups $\Cl^{+}_{\fa}(O_{F(t)})$)
modulo $\fa$ of the Wedderburn components $F(t)$ of $\Lambda(FG)$ with
$t \neq 1$. There is a natural surjection
\[
\nCl(\Lambda(O_FG)) \to \Cl_{\fa}'(\Lambda(O_FG))
\]
with kernel an elementary abelian $2$-group.

If $|G|$ is odd, then (as no non-trivial element of $G$ is conjugate
to its inverse), $F(t)$ has no real places when $t \neq 1$, and so
$\Cl_{\fa}(O_{F(t)}) = \Cl^{+}_{\fa}(O_{F(t)})$. Hence
we have
\[
\nCl(\Lambda(O_FG)) = \Cl_{\fa}(\Lambda(O_FG))
\]
whenever $G$ is of odd order.
\qed
\end{remark}

\begin{proposition} \label{P:raysurj}
Let $\fa$ be any integral ideal of $O_F$. Then the inclusion $\bF \to
J(\Lambda(FG))$ induces a surjection $\bF \to
\nCl(\Lambda(O_FG))$. In particular, each modified narrow ray
class modulo $\fa$ of $\Lambda(O_FG)$ contains infinitely many
elements of $\bF$.
\end{proposition}

\begin{proof}
Let $I(\Lambda(O_FG))$ denote the group of fractional ideals of
$\Lambda(O_FG)$. Then via the Wedderburn decomposition \eqref{E:wiso}
of $\Lambda(FG)$, we see that each fractional ideal $\fB$ in
$\Lambda(O_FG)$ may be written in the form $\fB = (\fB_t)_{t \in T}$,
where each $\fB_t$ is a fractional ideal of $O_{F(t)}$. For each
conjugacy class $t \in T$, let $o(t)$ denote the $\Omega_F$-orbit of
$t$ in $\cC(G(-1))$, and write $|t|$ for the order of any element of
$t$.

For each idele $\nu \in J(\Lambda(FG))$, let 
\[
\co(\nu):= [ \co(\nu)_{t}]_{t \in T} \in I(\Lambda(O_{F}G)) \simeq
\prod_{t \in T} I(O_{F(t)})
\] 
denote the ideal obtained by taking the idele content of
$\nu$. If $v$ is a place of $F$, we view $\bF_v$ as being a
subset of $\bF$ via the obvious embedding $\Lambda(F_vG)^{\times}
\subseteq J(\Lambda(FG))$, and we set
\begin{equation*}
\fF_v:= \{ \co(f_v) \mid f_v \in \bF_v \}.
\end{equation*}

Now suppose that $v$ is finite, and consider the ideal
\[
\co(f_{v,s}) = [\co(f_{v,s})_t]_{t \in T}
\]
in $I(\Lambda(O_FG))$.  If $c(s) \notin o(t)$, then it follows from the
definition of $f_{v,s}$ that $\co(f_{v,s})_t = O_{F(t)}$. Suppose that
$c(s) \in o(t)$. Since $s \in \Sigma_v(G)$, it follows that $v(|s|) =
0$ and that $\Omega_{F_v}$ fixes $c(s)$. Hence $F_v(t) = F_v$, and so we
see that $\co(f_{v,s})_t$ is a prime ideal of $O_{F(t)}$ of degree one
lying above $v$ (cf. \cite[pages 287--289]{Mc1}). Furthermore, if $t
\in T$ and if $v$ is a finite place of $F$ that is totally split in
$F(t)$, then $f_{v,s} \in \bF_v$ for all $c(s) \in o(t)$.

We therefore deduce that if $v$ is finite, the set $\fF_v$ consists
precisely of the invertible prime ideals $\fp = (\fp_t)_{t \in T}$ of
$\Lambda(O_FG)$ with $\fp_{t_1}$ a prime of degree one above $v$ in
$F(t_1)$ for some $t_1 \in T$ with $v(|t_1|) = 0$ and $\fp_t =
O_{F(t)}$ for all $t \neq t_1$. For every $t \in T$, the narrow ray
class modulo $\fa$ of $F(t)$ contains infinitely many primes of degree
one, and this implies that $\bF$ surjects onto $\nCl(\Lambda(O_FG))$
as claimed.
\end{proof}

Our next result describes a transpose Stickelberger homomorphism on
modified narrow ray class groups $\nCl(\Lambda(O_FG))$ for a
suitable choice of $\fa$. Before stating it, we remind the reader that
Proposition \ref{P:unramgp} implies that $\prod_v
\Image(\bpsi^{nr}_{v})$ is a subgroup of $J(K_0(O_FG, F^c))$.

\begin{proposition} \label{P:raytheta}
Let $N$ be an integer, and set $\fa:= N \cdot O_F$. Then if $N$ is
divisible by a sufficiently high power of $|G|$, the idelic transpose
Stickelberger homomorphism
\[
K\Theta^t: J(\Lambda(FG)) \to J(K_0(O_FG, F^c))
\]
induces a homomorphism
\[
\Theta^{t}_{\fa}: \nCl(\Lambda(O_FG)) \to 
\frac{J(K_0(O_FG, F^c))}{\lambda[
  \partial^{1} (K_1(F^cG))] \cdot \prod_v \Image(\bpsi^{nr}_{v})}.
\]
\end{proposition}

\begin{proof}
To show this, we first observe that Proposition \ref{P:intim} implies
that if $N$ is divisible by a sufficiently high power of $|G|$ and $v$
is any finite place of $F$, then we have
\[
\Theta_v^t(U'_{\fa}(\Lambda(O_{F_v}G))) \subseteq
\Det((O_{F_v}G)^{\times}/G) \subseteq \Det(\cH(O_{F_v}G)) = \Image(\bpsi^{nr}_{v}),
\]
and so it follows that 
\[
K\Theta^t(U'_{\fa}(\Lambda(O_FG))) \subseteq \prod_v \Image(\bpsi^{nr}_{v})
\]
in $J(K_0(O_FG, F^c))$.

Suppose that $v$ is a real place of $F$ and that $h \in
\Lambda(F_vG)^{\times}_{+}$. Then for each $\chi \in R_G$, we have
(recalling that $\langle \chi, e \rangle_G = 0$)
\[
\Theta_{v}^{t}(h)(\chi) = \prod_{g \in G} h(c(g))^{\langle \chi, g \rangle_{G}}
> 0,
\]
and so $\Theta_{v}^{t}(h) \in \Hom^{+}_{\Omega_{F_v}}(R_G,
(F_v^c)^{\times})$.  This implies that $K\Theta^t(h) = 1$ in
$K_0(O_{F_v}G, F^c_v)$. Hence $K\Theta^t(U_{\infty}'(\Lambda(O_FG))) =
1$ in $J(K_0(O_FG, F^c))$.
 
It now follows that  $K\Theta^t$ induces a homomorphism
\[
\Theta^{t}_{\fa}:\nCl(\Lambda(O_FG)) \to \frac{J(K_0(O_FG, F^c))}{\lambda[
  \partial^{1} (K_1(F^cG))] \cdot \prod_v \Image(\bpsi^{nr}_{v})},
\]
as claimed.
\end{proof}


\section{Proof of Theorem \ref{T:idgp}} \label{S:idgp}

In this section we shall prove Theorem \ref{T:idgp}. Recall that we
wish to show that if 
\[
\ov{\bpsi^{id}}: J(H^1_t(F, G)) \to \frac{J(K_0(O_FG, F^c))}{\lambda[
 \partial^{1} (K_1(F^cG))] \cdot \prod_v \Image(\bpsi^{nr}_{v})}
\]
denotes the map of pointed sets given by the composition of the map $\bpsi^{id}$
with the quotient homomorphism
\[
q_1:J(K_0(O_FG, F^c)) \to \frac{J(K_0(O_FG, F^c))}{\lambda[
 \partial^{1} (K_1(F^cG))] \cdot \prod_v \Image(\bpsi^{nr}_{v})},
\]
then the image of $\ov{\bpsi^{id}}$ is in fact a group.

To show this, we choose an ideal $\fa = N\cdot O_F$ as in Proposition
\ref{P:raytheta}, and we consider the following diagram:
\begin{equation} \label{E:light}
\begin{CD}
@.  @.  J(H^{1}_{t}(F,G)) \\
@.   @.   @V{\bpsi^{id}}VV \\
\bF @>{\subset}>>   J(\Lambda(FG)) @>{K\Theta^t}>> J(K_0(O_FG, F^c)) \\
@V{q_2}VV      @V{q_2}VV                    @V{q_1}VV      \\
\nCl(\Lambda(O_FG))@=  \nCl(\Lambda(O_FG))
@>{\Theta^{t}_{\fa}}>>
\dfrac{J(K_0(O_FG, F^c))}{\lambda[\partial^{1} (K_1(F^cG))] 
\cdot \prod_v \Image(\bpsi^{nr}_{v})} \\
\end{CD}
\end{equation}
  
Here $q_2$ denotes the obvious quotient map. Proposition
\ref{P:raytheta} shows that the right-hand square commutes, and
Proposition \ref{P:raysurj} shows that the left-most vertical arrow is
surjective.

It follows from Proposition \ref{P:reseq}(a) that
\begin{align*}
q_1 [K\Theta^t(\bF)] &= q_1[\bpsi^{id}(J(H^1_t(F,G)))] \\
&= \Image{\ov{\bpsi^{id}}}.
\end{align*}
On the other hand, we also have that
\[
q_1[K\Theta^t(\bF)] =
\Theta^{t}_{\fa}(\nCl(\Lambda(O_FG))),
\]
which is a group. It therefore follows that $\Image(\ov{\bpsi^{id}})$
is indeed a group, as claimed. 

This completes the proof of Theorem \ref{T:idgp}.
\qed


\section{Realisable classes from field extensions} \label{S:ker}

In this section, after first proving that the kernel of $\bpsi$ is
finite, we explain how a slightly weaker form of Conjecture
\ref{C:conj} implies that every element of $\cR(O_FG)$ may be realised
by the ring of integers of a tame field (as opposed to merely a Galois
algebra) $G$-extension of $F$.

Recall that $G'$ denotes the derived subgroup of $G$, and note that we
may view $H^1(F, G')$ and $H^1(F_v, G')$ as being pointed subsets of
$H^1(F, G)$ and $H^1(F_v, G)$ respectively via taking Galois
cohomology of the exact sequence of groups
\[
0 \to G' \to G \to G^{ab} \to 0.
\]
Recall also that we write $H^{1}_{fnr}(F, G')$ for the set of
isomorphism classes of $G'$-Galois $F$-algebras that are unramified at
all finite places of $F$.

\begin{proposition} \label{P:kernel}
(a) Let $v$ be a finite place of $F$. Then $\Ker(\bpsi_v) \subseteq
  H^{1}_{nr}(F_v, G')$.

(b) Suppose that $[\pi] \in \Ker(\bpsi)$. Then $[\pi] \in H^{1}_{fnr}(F,G')
  \subseteq H^1(F,G)$. We have that $\Ker(\bpsi)$ is finite.

(c) Suppose that $F/\bQ$ is at most tamely ramified at all primes
dividing $|G|$. Then
$H^{1}_{nr}(F, G') \subseteq \Ker(\bpsi)$.

(d) Suppose that $G$ has no irreducible symplectic characters or that
$F$ has no real places. Suppose also that $F/\bQ$ is at most tamely
ramified at all primes dividing $|G|$. Then $\Ker(\bpsi) = H^{1}_{fnr}(F, G')$.
\end{proposition}

\begin{proof}
(a) Let $v$ be a finite place of $F$.  Suppose that $[\pi_v] \in
  H^1_t(F_v, G)$, and that $O_{\pi_v} = O_{F_v}G \cdot a_v$.  Recall
  (see Sections \ref{S:reshom} and \ref{S:cc}) that we have
\[
\bpsi_v: H^{1}_{t}(F_v, G) \to K_0(O_{F_{v}}G, F_{v}^{c}) \simeq
\frac{\Det(F_{v}^{c}G)^{\times}}{\Det(O_{F_v}G)^{\times}},
\]
and that $\bpsi_v([\pi_v]) = [\Det(\br_G(a_v))]$ (see also Definition
\ref{D:Det} and Remark \ref{R:nrd}). It follows that $\bpsi_v([\pi_v])
= 0$ if and only if $\Det(\br_G(a_v)) \in \Det(O_{F_{v}}G)^{\times}$. 

Hence, if $\bpsi_v([\pi_v]) = 0$, then for each $\omega \in
\Omega_{F_{v}}$, we have
\[
\Det(\br_G(a_v)^{-1}) \cdot \Det(\br_G(a_v))^{\omega} = 1,
\]
and so we deduce from \eqref{E:resdiag} that $[\pi_v]$ lies in the
kernel of the natural map $H^1(F_v, G) \to H^1(F_v, G^{ab})$ of
pointed sets. This implies that $[\pi_v] \in H^1(F_v, G')$.  Finally,
we see from \eqref{E:stickfac} and Proposition \ref{P:reseq}(c) that
$\Det(\br_G(a_v)) \in \Det((O_{F_{v}}G)^{\times})$ only if $[\pi_v]
\in H^{1}_{nr}(F_v, G)$. We now conclude that if $[\pi_v] \in
\Ker(\bpsi_v)$, then $[\pi_v] \in H^{1}_{nr}(F_v, G')$.  This
establishes part (a).

(b) Suppose that $[\pi] \in H^1(F,G)$ satisfies $\bpsi([\pi]) =
0$. Then $\bpsi_v(\loc_v([\pi])) = 0$ for each place $v$, and so it
follows from part (a) that $\loc_v([\pi]) \in H^{1}_{nr}(F_v, G')$ for
all finite places $v$ of $F$. Therefore $[\pi] \in H^{1}(F,G')$, and
$\pi$ is unramified at each finite place of $F$, i.e. $[\pi] \in
H^{1}_{fnr}(F, G')$. As there are only finitely many unramified
extensions of $F$ of bounded degree, it follows that $H^{1}_{fnr}(F,
G')$ is finite, and so $\Ker(\bpsi)$ is finite, as claimed.

(c) Suppose that $[\pi] \in H^{1}_{nr}(F, G') \subseteq H^{1}_{t}(F,
G)$, and write $O_{\pi_v} = O_{F_v}G \cdot a_v$ for each finite place
$v$ of $F$. As $\pi$ is unramified at $v$, it follows that
$\Det(\br_G(a_v)) \in \Det(O_{F^{nr}_{v}}G)^{\times}$. Since
$\loc_v([\pi])$ lies in the kernel of the natural map $H^1(F_v, G) \to
H^1(F_v, G^{ab})$, we see from the diagram \eqref{E:resdiag} that the
image of $\Det(\br_G(a_v))$ in $Z(F_vG)^{\times} \backslash
\cH(Z(F_vG))$ is trivial, and so in fact $\Det(\br_G(a_v)) \in
   [\Det(O_{F^{nr}_{v}}G)^{\times}]^{\Omega_{F_v}}$. Note that
   $\Det(\br_G(a_v))$ is defined over the finite, unramified extension
   $F_{v}^{\pi_v}$ of $F_v$ (see \eqref{E:weddcomp}). Let $L$ denote
   an arbitrary finite, unramified extension of $F_v$.

If $v \nmid |G|$, then $O_{L}G$ is an $O_{L}$-maximal order in $LG$,
and we have (see \eqref{E:vdetiii})
\begin{align*}
[\Det(O_LG)^{\times}]^{\Omega_{F_v}} &\simeq [\Hom_{\Omega_L}(R_{G},
  (O_{F^{c}_{v}})^{\times})]^{\Omega_{F_{v}}} \\
&\simeq \Hom_{\Omega_{F_v}}(R_{G},  (O_{F^{c}_{v}})^{\times}) \\
&\simeq \Det(O_{F_v}G)^{\times}.
\end{align*}
If $v \mid |G|$, then because $F/\bQ$ is at most tamely ramified at
all primes dividing $|G|$, it follows from M. J. Taylor's fixed point
theorem for group determinants (see e.g.  \cite[Chapter VIII]{T}) that
\[
[\Det(O_LG)^{\times}]^{\Omega_{F_v}} = \Det(O_{F_v}G)^{\times}.
\]
Hence, for each finite place $v$ of $F$, we see that $\Det(\br_G(a_v))
\in \Det(O_{F_{v}}G)^{\times}$, and so $\bpsi_v([\pi_v]) = 0$
(cf. part (a) above).

Since $H^{1}_{nr}(F_v, G) = 0$ for all infinite places of $F$, it
follows that $\bpsi_v([\pi_v]) = 0$ for all places $v$ of $F$. This in
turn implies that $\lambda(\bpsi([\pi])) = 0$. As the localisation map
$\lambda$ is injective (see Proposition \ref{P:locinj}(a)), it follows
that $\bpsi([\pi]) = 0$. Hence $H^{1}_{nr}(F, G') \subseteq
\Ker(\bpsi)$, as claimed.

(d) The proof of this assertion is very similar to that of part (c)
above, and so here we shall be brief. Suppose that $[\pi] \in
H^{1}_{fnr}(F, G')$. Arguing exactly as in part (c), we see that
$\bpsi_v([\pi]_v) = 0$ for all finite places $v$ of $F$, which in turn
implies that $\lambda_f(\bpsi([\pi])) = 0$. Under our hypotheses,
Proposition \ref{P:locinj}(b) implies that the localisation map
$\lambda_f$ is injective, and so $\bpsi([\pi]) = 0$. Hence we see that
$H^{1}_{fnr}(F, G') \subseteq \Ker(\bpsi)$, and so it follows from
part (b) above that in fact $H^{1}_{fnr}(F, G') = \Ker(\bpsi)$, as
asserted.
\end{proof}

\begin{definition}   \label{D:ram}
Suppose that $x \in \LC(O_FG)$ (see Definition \ref{D:cc}). We say that
  $x$ is \textit{unramified} (respectively \textit{ramified}) at a
  place $v$ of $F$ if $\lambda_v(x) \in \Image(H^{1}_{nr}(F_v,
  G))$ (respectively if $\lambda_v(x) \notin \Image(H^{1}_{nr}(F_v,
  G))$). 

If $S$ is any finite set of places of $F$, we denote the set
of $x \in \LC(O_FG)$ that are unramified at all places in $S$ by
$\LC(O_FG)_S$. \qed
\end{definition}

Before stating our next result, it will be helpful to introduce the
following notation. Suppose that $x \in \LC(O_FG)$ and let $[(x_v)_v,
  x_\infty] \in J(K_1(FG)) \times \Det(F^cG)^{\times}$ be a
representative of $x$. Then $\lambda(x) \in J(K_0(O_FG, F^c))$ is
represented by the element $(x_v \cdot \loc_v(x_{\infty})) \in \prod_v
\Det(F^{c}_{v}G)^{\times}$. Hence it follows from Theorem \ref{T:stickfac} and
Proposition \ref{P:reseq}(a) that we have an equality
\begin{equation} \label{E:dec}
[(x_v \cdot \loc_v(x_\infty))] = [a(x)] \cdot K\Theta^t(f(x))
\end{equation}
in $J(K_0(O_FG, F^c))$, where $a(x) = (a(x)_v) \in \prod_v
\Det(\cH(O_{F_v}G))$ and $f(x) \in \bF$.

\begin{definition} \label{D:fullram}
We say that $x \in \LC(O_FG)$ is \textit{fully ramified} if $f(x)$ is
full (see Definition \ref{D:fprime}---note in particular that this
does \textit{not} mean that $x$ is ramified at all places of $F$,
which would of course be absurd!).  \qed
\end{definition}

Let us also recall that $\partial^{0}(x) \in \Cl(O_FG)$ is represented
by the idele $(x_v)_v \in J(K_1(FG))$ (see Remark \ref{R:repid}).

\begin{proposition} \label{P:finen}
Suppose that $S$ is any finite set of places of $F$, and that
$x \in \LC(O_FG)$. Then there exist infinitely many $y \in
\LC(O_FG)_S$ with $\partial^{0}(y) = \partial^{0}(x)$ in $\Cl(O_FG)$.
Hence we have
\begin{equation} \label{E:sreal}
\partial^{0}(LC(O_FG)) = \partial^{0}(LC(O_FG)_S).
\end{equation}
\end{proposition}

\begin{proof}
Let $\fa$ be an ideal of $F$ chosen as in Proposition \ref{P:raytheta}
(so $\fa$ is divisible by a sufficiently high power of $|G|$ for the
homomorphism $\Theta^{t}_{\fa}$ to be defined). Proposition
\ref{P:raysurj} implies that there are infinitely many choices of $g
\in \bF$ such that $\Supp(g)$ is disjoint from $S$ and $g$ lies in the
same modified narrow ray class modulo $\fa$ as $f(x)$, i.e.
\[
f(x) \equiv g \pmod {\Lambda(FG)^{\times} \cdot
  U_{\fa}'(\Lambda(O_FG)) \cdot U_{\infty}'(\Lambda(O_FG))_{+}}.
\]
Hence for any such $g$, we have
\[
K\Theta^t(f(x)) = K\Theta^t(\beta \cdot b \cdot g)
\]
where $\beta \in \Lambda(FG)^{\times}$ and $b = (b_v) \in
U_{\fa}'(\Lambda(O_FG)) \cdot U_{\infty}'(\Lambda(O_FG))_{+}$. Now
$K\Theta^t(\beta) \in \partial^1(K_1(F^cG))$ (see \eqref{E:mstickhom},
\eqref{E:mstickhomi}, and \eqref{E:stickhom}), while $K\Theta^t(b)$
lies in the image of $\prod_v \Det(\cH(O_{F_v}G))$ in
$J(K_0(O_FG,F^c))$, by virtue of our choice of $\fa$. We therefore see
from \eqref{E:dec} that we have the equality
\[
[(x_v \cdot \loc_v(x_{\infty}))] \cdot K\Theta^t(\beta)^{-1} = [a(x)]
\cdot K\Theta^t(b) \cdot K\Theta^t(g)
\]
in $J(K_0(O_FG, F^c))$. Then the class 
\[
y = [(x_v \cdot \loc_v(x_{\infty}))] \cdot
K\Theta^t(\beta)^{-1}
\]
in $J(K_0(O_FG, F^c))$ satisfies the desired conditions.

The final assertion follows immediately from the exact sequence
\eqref{E:rkes}.
\end{proof}

\begin{proposition}  \label{P:full}
Suppose that $S$ is any finite set of places of $F$, and that $x \in
\LC(O_FG)$. Then there exist infinitely many $y \in \LC(O_FG)_S$ such
that $y$ is fully ramified and $\partial^0(y) = \partial^0(x)$ in
$\Cl(O_FG)$.
\end{proposition}

\begin{proof}
This is a generalisation of \cite[Proposition 6.14]{Mc}, and it may be
proved in the same way as \cite[Proposition 7.4]{B}.

We begin by constructing a full element $h$ of $\bF$ as follows.  Let
$M/F$ be a finite Galois extension such that $\Omega_M$ acts trivially
on $\cC(G(-1))$. For each $s \in G$, choose a place $v(s)$ of $F$ that
splits completely in $M/F$; the Chebotarev density theorem implies
that this may be done so that the places $v(s)$ are distinct and
disjoint from $S$. Then the element $h = \prod_{s \in G} f_{v(s), s}$
is full.

Next, we choose an ideal $\fa$ of $F$ as in Proposition
\ref{P:raytheta} and observe that Proposition \ref{P:raysurj} implies
that there are infinitely many choices of $g \in \bF$ with $\Supp(g)$
disjoint from $S \cup \Supp(h)$ such that $g$ lies in the same
modified narrow ray class of $\Lambda(O_FG)$ modulo $\fa$ as $f(x)
\cdot h^{-1}$. Then, for any such $g$, we have that
\[
f(x) \equiv g \cdot h \pmod{\Lambda(FG)^{\times} \cdot
  U_{\fa}(\Lambda(O_FG)) \cdot U_{\infty}'(\Lambda(O_FG))_{+}},
\]
and $g \cdot h \in \bF$ is full. Now exactly as in the proof of
Proposition \ref{P:finen} we may replace $f(x)$ by $g \cdot h$ in
\eqref{E:dec}, changing the other terms in the equality as needed, to
obtain $y \in K_0(O_FG, F^c)$ satisfying the stated conditions.
\end{proof} 

\begin{theorem}  \label{T:fieldres}
Let $S$ be any finite set of places of $F$, and suppose that
Conjecture \ref{C:conj} holds for $\LC(O_FG)_S$, i.e. that
\begin{equation}  \label{E:inc}
\LC(O_FG)_S \subseteq K\cR(O_FG) = \Image(\bpsi).
\end{equation} 
Then $\cR(O_FG)$ is a subgroup of $\Cl(O_FG)$. If $c \in
\cR(O_FG)$, then there exist infinitely many $[\pi] \in H^1_t(F, G)$
such that $F_{\pi}$ is a field and $(O_{\pi}) = c$. The extensions
$F_{\pi}/F$ may be chosen to have ramification disjoint from $S$.
\end{theorem}

\begin{proof}
To prove the first assertion, it suffices to show that, under the
given hypotheses, we have
\begin{equation} \label{E:now}
\partial^{0}(\LC(O_FG)) = \cR(O_FG)
\end{equation}
(cf. the proof of Theorem \ref{T:congp}, especially \eqref{E:later}).

  We plainly have $\cR(O_FG) \subseteq \partial^0(\LC(O_FG))$. Suppose
  that $x \in \LC(O_FG)$, and set $c_x = \partial^0(x)$. Then
  Proposition \ref{P:full} implies that there exists $y \in
  \LC(O_FG)_S$ with $\partial^0(y) = c_x$. By hypothesis, we have $y
  \in \Image(\bpsi)$, and so $\partial^0(y) = c_x \in \cR(O_FG)$. This
  implies that $\partial^0(\LC(O_FG)) \subseteq \cR(O_FG)$. Hence
  \eqref{E:now} holds, and so $\cR(O_FG)$ is a subgroup of
  $\Cl(O_FG)$, as claimed.

Next, we observe that if $c \in \cR(O_FG)$, then \eqref{E:now} and
Proposition \ref{P:full} imply that there are infinitely many $x \in
\LC(O_FG)_S$ such that $x$ is fully ramified and $\partial^0(x) =
c$. For each such $x$, our hypotheses imply that there exists $\pi_x
\in \Hom(\Omega_F, G)$ with $[\pi_x] \in H^1_t(F, G)$ and
$\bpsi([\pi_x]) = x$. The set of primes that ramify in $F_{\pi_x}/F$
is equal to $\Supp(f(x))$, and so $F_{\pi_x}/F$ has ramification
disjoint from $S$. As $f(x)$ is full, we see that for each
non-identity element $s \in G$, there is a place $v(s) \in
\Supp(f(x))$ such that $\pi_x(\sigma_{v(s)}) \in c(s)$
(cf. \eqref{E:rmap} and Proposition \ref{P:reseq} (a) and (b)). Hence
$\Image(\pi_x)$ has non-trivial intersection with every conjugacy
class of $G$ and so is equal to the whole of $G$, by a lemma of Jordan
(see \cite[p. 435, Theorem 4']{Se2003}). Therefore $\pi_x$ is
surjective, and so $F_{\pi_x}$ is a field. This establishes the
result.
\end{proof}


\section{Abelian groups} \label{S:ab}

In this section we shall prove that Conjecture \ref{C:cc} holds for
abelian groups. We shall also show that the map $\bpsi$ is injective
in this case.

Let $G$ be abelian, and suppose that $L$ is any finite extension of
$F$ or of $F_v$ for some place $v$ of $F$. As $G$ is abelian, the
reduced norm map induces isomorphisms
\begin{equation} \label{E:goodab}
(LG)^{\times} \simeq \Det(LG)^{\times}, \quad (O_LG)^{\times} \simeq
\Det(O_LG)^{\times},\quad (L^cG)^{\times} \simeq \Det(L^cG)^{\times}.
\end{equation}

For each finite place $v$ of $F$, Lemma \ref{L:klocdes} and
\eqref{E:goodab} imply that there are isomorphisms
\[
K_0(O_{F_v}G, F^{c}_{v}) \simeq
\frac{\Det(F^{c}_{v}G)^{\times}}{\Det(O_{F_v}G)^{\times}} \simeq
\frac{(F^{c}_{v}G)^{\times}}{(O_{F_v}G)^{\times}}.
\]

\begin{proposition} \label{P:lab}
Let $G$ be abelian, and suppose that $v$ is a finite place of
$F$. Then the map $\bpsi_v$ is injective.
\end{proposition}

\begin{proof}
Suppose that $[\pi_{v,i}] \in H^1_t(F_v, G)$ ($i = 1, 2$), with
$O_{\pi_{v,i}} = O_{F_v}G \cdot a_{v,i}$. Then $\bpsi_v([\pi_{v,i}]) =
[\br_G(a_{v,i})]$ in $(F^c_vG)^{\times}/(O_{F_v}G)^{\times}$. Hence if
$\bpsi([\pi_{v,1}]) = \bpsi([\pi_{v,2}])$, then we have
$\br_G(a_{v,1}) \cdot \br_G(a_{v,2})^{-1} \in
(O_{F_v}G)^{\times}$. This implies that $[\pi_{1,v}] = [\pi_{2,v}]$ in
$H^1_t(F_v, G)$, and so it follows that $\bpsi_v$ is injective, as
claimed.
\end{proof}

Again because $G$ is abelian, the pointed set of resolvends $H_t(LG)$
is an abelian group, and the exact sequences \eqref{E:res3} and
\eqref{E:res4} show that there is an isomorphism
\begin{equation} \label{E:abg}
\tau: H^1_t(L, G) \xrightarrow{\sim} \frac{H_t(LG)}{(LG)^{\times}}
\end{equation}
defined as follows: if $[\pi] \in H^1_t(L, G)$ with $L_{\pi} = LG
\cdot b_{\pi}$, then $\tau([\pi]) = [\br_G(b_{\pi})]$.

Note also that Theorem \ref{T:kdes}(b) and \eqref{E:goodab} imply that
$K_0(O_FG, F^c)$ is isomorphic to the cokernel of the homomorphism
\[
\Delta_{O_FG,F^c}: (FG)^{\times} \to
\frac{J(FG)}{\prod_v(O_{F_v}G)^{\times}} \times (F^cG)^{\times}
\]
induced by 
\[
(FG)^{\times} \to J(FG) \times (F^cG)^{\times};\quad x \mapsto
((\loc_v(x))_v, x^{-1}).
\]

\begin{theorem} \label{T:ab}
Conjecture \ref{C:cc} is true when $G$ is abelian.
\end{theorem}

\begin{proof}
Suppose that $x \in \LC(O_FG)$, and let $[(x_v)_v, x_{\infty}] \in
J(FG) \times (F^cG)^{\times}$ be a representative of $x$. We shall
explain how to construct an element $[\pi] \in H^1_t(F, G)$ such that
$\lambda_v(x) = \lambda_v(\bpsi([\pi]))$ for all finite places $v$ of
$F$. Since $G$ is abelian, and therefore admits no irreducible
symplectic characters, this will imply that $x = \bpsi([\pi])$ (see
Proposition \ref{P:locinj}(b)).

For each $v$, we have that $x_v \cdot \loc_v(x_{\infty}) \in
H_t(F_vG)$.  As $x_v \in (F_vG)^{\times}$, this implies that
$\loc_v(x_{\infty}) \in H_t(F_vG)$ for each $v$. It follows from
Proposition \ref{P:reshasse} that $x_{\infty} \in H(FG)$, and we see
in addition that in fact $x_{\infty} \in H_t(FG)$. Hence $x_{\infty}$
is the resolvend of a normal basis generator of a tame extension
$F_{\pi}/F$. Set $\pi_v:= \loc_v(\pi)$. Then for each finite $v$, we have
\[
\tau(\bpsi^{-1}_{v}(\lambda_v(x))) = [\loc_v(x_{\infty})] =
\tau([\pi_v])
\]
in $H_t(F_vG)/(F_vG)^{\times}$, which in turn implies that
\[
\lambda_v(x) = \bpsi_v([\pi_v]) = \lambda_v(\bpsi([\pi])).
\]
Hence $x = \bpsi([\pi])$, as required.
\end{proof}

\begin{proposition} \label{P:gab}
If $G$ is abelian, then the map $\bpsi$ is injective.
\end{proposition}

\begin{proof}
Let $[\pi] \in H^1_t(F_v, G)$, and suppose that $[(x_v)_v, x_{\infty}]
\in J(K_1(FG)) \times (F^cG)^{\times}$ is a representative of
$\bpsi([\pi])$. Then it follows from the proof of Theorem \ref{T:ab}
that $\tau([\pi]) = x_{\infty}$ in $H_t(FG)/(FG)^{\times}$. Since
$\tau$ is an isomorphism, we deduce that $\bpsi$ is injective.
\end{proof}


\section{Neukirch's Lifting Theorem}   \label{S:neu}

Our main purpose in this section is to describe certain results,
mainly from \cite{Neu}, that will be used in the proof of Theorem
\ref{T:D}. We refer the reader to \cite{Neu} or \cite[IX.5]{NSW} for
full details regarding these topics.

Let $D$ be an arbitrary finite group. Consider the category $\cD$ of
homomorphisms $\eta: \cG \to D$ of arbitrary profinite groups $\cG$
into $D$ in which a morphism between two objects $\eta_1: \cG_1 \to D$
and $\eta_2: \cG_2 \to D$ is defined to be a homomorphism $\nu: \cG_1
\to \cG_2$ such that $\eta_1 = \eta_2 \circ \nu$. We say that two such
morphisms $\nu_i: \cG_1 \to \cG_2$ ($i = 1, 2$) are
\textit{equivalent} if there is an element $k \in \Ker(\eta_2)$ such
that $\nu_1(\omega) = k \cdot \nu_2(\omega) \cdot k^{-1}$ for all
$\omega \in \cG_1$. Write $\cHom_{D}(\cG_1, \cG_2)$ for the set of
equivalence classes of homomorphisms $\cG_1 \to \cG_2$, and
$\cHom_{D}(\cG_1, \cG_2)_{\epi}$ for the subset of $\cHom_{D}(\cG_1,
\cG_2)$ consisting of equivalence classes of surjective
homomorphisms. 

Suppose now that we have an exact sequence
\[
0 \to B \to G \xrightarrow{q} D \to 0
\]
with $B$ abelian, and that $L$ is a number field or a local field. Let
$h: \Omega_L \to D$ be a fixed homomorphism. We view $\Omega_L
\xrightarrow{h} D$ and $G \xrightarrow{q} D$ as being elements of
$\cD$.  The group $D$ acts on $B$ via inner automorphisms, and this in
turn induces an action of $\Omega_L$ on $B$ via $h$. We write $L(B)$
for the smallest extension of $L$ such that $\Omega_{L(B)}$ fixes $B$
(i.e. $L(B)$ is the field of definition of $B$).

It may be shown that the group $H^1(L, B)$ acts on $\cHom_D(\Omega_L,
G)$ in the following way. Let $z \in Z^1(L, B)$ be any $1$-cocycle
representing $[z] \in H^1(L, B)$, and let $\nu \in \Hom(\Omega_L, G)$
be any homomorphism, representing an element $[\nu] \in
\cHom_D(\Omega_L, G)$.
Define $z \cdot \nu: \Omega_L \to G$ by
\[
(z \cdot \nu)(\omega) = z(\omega) \cdot \nu(\omega)
\]
for all $\omega \in \Omega_L$. It is not hard to check that 
\[
h = q \circ (z \cdot \nu),
\]
and that the element $[z \cdot \nu] \in \cHom_{D}(\Omega_L, G)$ is
independent of the choices of $z$ and $\nu$. It may also be shown that
$\cHom_D(\Omega_L, G)$ is a principal homogeneous space over $H^1(L,
B)$.

For a number field $F$, and a finite place $v$ of $F$, we let
$\cHom_{D}(\Omega_{F_v}, G)_{nr}$ denote the set of classes of
homomorphisms $\Omega_{F_v} \to G$ that are trivial on $I_v$. We write
$J_f(\cHom_{D}(\Omega_F, G))$ for the restricted direct product over
all finite places of $F$ of the sets $\cHom_{D}(\Omega_{F_v}, G)$ with
respect to the subsets $\cHom_{D}(\Omega_{F_v}, G)_{nr}$.

Now we can state Neukirch's Lifting Theorem.

\begin{theorem}  \label{T:nl}
Let $F$ be a number field and let $h: \Omega_F \to D$ be a fixed,
surjective homomorphism. Suppose that
\[
0 \to B \to G \xrightarrow{q} D \to 0
\]
is an exact sequence for which $B$ is a simple
$\Omega_F$-module. (This implies that $l \cdot B = 0$ for a unique
prime $l$.) Assume that the field of definition $F(B)$ of $B$ contains
no non-trivial $l$-th roots of unity, and that
$J_f(\cHom_{D}(\Omega_{F}, G)) \neq \emptyset$. Let $S$ be any finite
set of finite places of $F$. Then the natural map
\[
\cHom_{D}(\Omega_F, G)_{\epi} \to \prod_{v \in S}\cHom_{D}(\Omega_{F_v}, G)
\]
is surjective. 
\end{theorem}

\begin{proof}
This is \cite[Main Theorem, p. 148]{Neu}. 
\end{proof}

The following result implies that $\cHom_{D}(\Omega_{F_v}, G) \neq
\emptyset$ for all but finitely many $v$.

\begin{proposition}  \label{P:nrlift}
(\cite[Lemma 5]{Neu})
Let $F$ be a number field, and let $v$ be a finite place of
$F$. Suppose that $\cG_1 \to \cG_2$ is a surjective homomorphism of
arbitrary profinite groups, and that there exists an unramified
homomorphism $h_v: \Omega_{F_v} \to \cG_2$. Then
$\cHom_{\cG_{2}}(\Omega_{F_v}, \cG_1)_{nr} \neq \emptyset$, and so
$\cHom_{\cG_{2}}(\Omega_{F_v}, \cG_1) \neq \emptyset$ also.
\end{proposition}

\begin{proof}
If $h_v$ is unramified, then $h_v$ factors through
$\Omega_{F_v}/I_v \simeq \hat{\bZ}$, and a map $\hat{\bZ} \to \cG_2$ may
always be lifted to a map $\hat{\bZ} \to \cG_1$ by lifting the image of a
topological generator of $\hat{\bZ}$. 
\end{proof}

We now turn to two results of a local-global nature that will play a
role in the proof of Theorem \ref{T:domco}. In order to describe them, we
let $\Gamma$ be a finite abelian group equipped with an action of
$\Omega_F$ such that $\Gamma$ is a simple $\Omega_F$-module. Then $l
\cdot \Gamma = 0$ for a unique prime $l$. Write $F(\Gamma)$ for the
field of definition of $\Gamma$.

\begin{theorem}  \label{T:ncl}
Let $M/F$ be a Galois extension with $F(\Gamma) \subseteq M$ and
$\mu_{l} \nsubseteq M$, and let $\cN /M$ be a finite abelian
extension. Let $S$ be a finite set of finite places of $F$, and
suppose given an element $y_v \in H^1(F_v, \Gamma)$ for each $v \in
S$. Then there exists an element $z \in H^1(F, \Gamma)$ satisfying the
following local conditions:

(i) $z_v = y_v$ for each $v \in S$.

(ii) If $v \notin S$, then $z_v$ is cyclic (i.e. is trivialised
by a cyclic extension of $F_v$), and if $z_v$ is ramified, then
$v$ splits completely in $\cN/F$.
\end{theorem}

\begin{proof} 
This is \cite[Theorem 1]{Neu}.
\end{proof}

In order to state our next result, we introduce the following
notation. 

\begin{definition} \label{D:ohell}
Let $T := \{ v_1,\ldots, v_r \}$ be any finite set of finite places of
$F$ containing all places that ramify in $F(\Gamma)/F$ and all places
above $l$. Let $\fp_i$ denote the prime ideal of $F$ corresponding to
$v_i$. Proposition \ref{P:kill} implies that we may choose an integer
$N = N(T)$ such that for each $1\leq i \leq r$ and for every place $w$
of $F(\Gamma)$ lying above $v_i$, we have
\[
\Hom_{\Omega_{F(\Gamma)_w}}(A_{\Gamma}, U_{\fp^{N}_{i}}(O_{F(\Gamma)^c_w})) \subseteq
\rag[\Hom_{\Omega_{F(\Gamma)_w}}(R_{\Gamma}, O^{\times}_{F(\Gamma)^c_w})].
\]
Set 
\[
\fa = \fa(T) = \prod_{i=1}^{r} \fp_i.
\]
Let $F(\fa^N)$ denote the ray class field of $F$ modulo $\fa^N$.  \qed
\end{definition}

\begin{theorem} \label{T:bcl}
Let $v \notin T$ be any finite place of $F$ that splits completely in
$F(\fa^N)$, and suppose that $s$ is any non-trivial element of
$\Gamma$. Then there is an element $b = b(v;s) \in H^1(F, \Gamma)$
satisfying the following local conditions:

(i) $\loc_{v_i}(b) = 0$ for $1 \leq i \leq r$;

(ii) $b \mid_{I_v} = \ti{\vphi}_{v,s}$ (see Remark \ref{R:ramphi});

(iii) $b$ is unramified away from $v$.
\end{theorem}

\begin{proof}
Let $\fp$ be the prime ideal of $F$ corresponding to $v$. Our
hypotheses on $v$ imply that $\fp$ is principal, with $\fp \equiv 1
\pmod{\fa^N}$.  Set $M := F(\Gamma)$. As $\Gamma$ is abelian, we have
that $\cH(M\Gamma) \simeq \Hom_{\Omega_{M}}(A_{\Gamma},
(M^c)^{\times})$ (cf.  \eqref{E:zres}). Let $\vpi$ be a generator of
$\fp$, and define $\rho \in \Hom_{\Omega_M}(A_{\Gamma}, (M^c)^{\times})$ by
\[
\rho(\alpha) = \vpi^{\langle \alpha, s \rangle_{\Gamma}}.
\]
(This homomorphism is $\Omega_M$-equivariant because $\Omega_M$ fixes
$\Gamma$.) Then $\rho$ is the reduced resolvend of a normal basis
generator of an extension $M_{\pi(\rho)}/M$ corresponding to
$[\pi(\rho)] \in H^1(M, \Gamma)$. Since $\fp \equiv 1 \pmod{\fa^N}$,
for each place $w$ of $M$ lying above a place $v_i$ in $T$, we have
\[
\loc_w(\rho) \in \Hom_{\Omega_{M_w}}(A_{\Gamma}, U_{\fp^{N}_{i}}(O_{M^c_w})) \subseteq
\rag[\Hom_{\Omega_{M_w}}(R_{\Gamma}, O^{\times}_{M^c_w})],
\]
and so it follows that $\loc_{w}(\pi(\rho)) = 0$ (see
\eqref{E:globcoh}). In particular, $\pi(\rho)$ is unramified at all
places above $T$.

For all places $w'$ of $M$ not lying above $T$ or $v$ we have that
\[
\loc_{w'}(\rho) \in \Hom_{\Omega_{M_{w'}}}(A_{\Gamma},
O^{\times}_{M^{c}_{w'}}),
\]
and so $\pi(\rho)$ is unramified at $w'$. This implies that
$\pi(\rho)$ is unramified away from $v$, since we have already seen
that $\pi(\rho)$ does not ramify at any place above $T$. It is also
easy to see that
\[
b \mid_{I_{w(v)}} = \ti{\vphi}_{w(v),s}
\]
for any place $w(v)$ of $M$ lying above $v$ (cf. the proof of
Proposition \ref{P:reseq}(a)).

As $\vpi \in F$, we have that $\pi(\rho) \in H^1(M,
\Gamma)^{\Gal(M/F)}$. Since $\Gamma^{\Omega_{F}} = 0$ (because
$\Gamma$ is a simple $\Omega_F$-module), the restriction map $H^1(F,
\Gamma) \to H^1(M, \Gamma)$ is injective and induces an isomorphism
$H^1(F, \Gamma) \simeq H^1(M, \Gamma)^{\Gal(M/F)}$.  Hence $\pi(\rho)$
is the image of an element $b \in H^1(F, \Gamma)$ satisfying the
conditions (i), (ii) and (iii) of the theorem.
\end{proof}


\section{Soluble groups}   \label{S:odd}

In this section we shall use Neukirch's Lifting Theorem to prove a
result (see Theorem \ref{T:domco} below) that implies Theorem
\ref{T:D} of the Introduction. In order to describe this result, it
will be helpful to formulate the following definition.

\begin{definition} \label{D:R}
Let $S$ be any finite (possibly empty) set of places of $F$.  We shall
say that $\LC(O_FG)_S$ \textit{satisfies Property R} if the following
holds.  Suppose given any fully ramified $x \in \LC(O_FG)_S$. For each
finite place $v$ of $F$, suppose also given a homomorphism $\pi_{v,x}
\in \Hom(\Omega_{F_v}, G)$ such that $[\pi_{v,x}] \in H^1_t(F_v, G)$
and $\lambda_v(x) = \bpsi_v([\pi_{v,x}])$. (Note that in general, such
a choice of $\pi_{v,x}$ is not unique.) Then there exists $\Pi \in
\Hom(\Omega_F, G)$ with $[\Pi] \in H^1_t(F, G)$ such that

(a) $x = \bpsi([\Pi])$;

(b) $\Pi \mid_{I_v} = \pi_{v,x} \mid_{I_v}$ for each finite place $v$
of $F$.

(So in particular, $x$ is cohomological.)
\qed
\end{definition}

\begin{proposition} \label{P:rab}
If $G$ is abelian, then $\LC(O_FG)$ satisfies Property R.
\end{proposition}

\begin{proof}
We shall in fact prove a slightly stronger result. Suppose that $G$
is abelian, and let $x \in \LC(O_FG)$. (Note that we do not assume
that $x$ is fully ramified.) Then Theorem \ref{T:ab} implies that $x$
is cohomological. As $G$ is abelian, the maps
$\bpsi$ and $\bpsi_v$ are injective (see Propositions
\ref{P:lab} and \ref{P:gab}). Hence it follows that there is a unique $[\Pi] \in
H^1_t(F, G)$ such that $x = \bpsi([\Pi])$, and a unique $[\pi_{v,x}]
\in H^1_t(F_v, G)$ such that $\lambda_v(x) = \bpsi_v([\pi_{v,x}])$.
We therefore see that
\[
\lambda_v(x) =  \bpsi_v([\Pi_v]) = \bpsi([\pi_{v,x}]),
\]
and so $\Pi_v = \pi_{v,x}$. This implies that $\LC(O_FG)$ satisfies
Property R.
\end{proof}

\begin{theorem} \label{T:rgreat}
Suppose that $\LC(O_FG)_S$ satisfies Property R. Then $\cR(O_FG)$ is a
subgroup of $\Cl(O_FG)$.  If $c \in \cR(O_FG)$, then there exist
infinitely many $[\pi] \in H^1_t(F, G)$ such that $F_{\pi}$ is a field
and $(O_{\pi}) = c$. The extensions $F_{\pi}/F$ may be chosen to have
ramification disjoint from $S$.
\end{theorem}

\begin{proof}
This is an immediate consequence of Theorem \ref{T:fieldres}.
\end{proof}

Our proof of Theorem \ref{T:D} rests on the following result.

\begin{theorem}  \label{T:domco}
Suppose that there is an exact sequence
\[
0 \to B \to G \to D \to 0,
\]
where $B$ is an abelian minimal normal subgroup of $G$ with $l \cdot B
= 0$ for an odd prime $l$. Let $S$ be any finite set of finite places
of $F$ containing all places dividing $|G|$.
Assume that the following conditions hold:

(i) The set $\LC(O_FD)_S$ satisfies Property R;

(ii) We have $(|G|, h_F) = 1$, where $h_F$ denotes the class number of
$F$;

(iii) Either $G$ admits no irreducible symplectic characters, or $F$
has no real places;

(iv) The field $F$ contains no non-trivial $l$-th roots of unity.

Then $\LC(O_FG)_S$ satisfies Property R.
\end{theorem}

\begin{proof}
We shall establish this result in several steps, one of which
crucially involves Neukirch's Lifting Theorem (see Theorem
\ref{T:nl}).

Suppose that $x \in \LC(O_FG)_S$ is fully ramified. For each finite
place $v$ of $F$, choose $\pi_{v,x} \in \Hom(\Omega_{F_v}, G)$ such
that $[\pi_{v,x}] \in H^1_t(F_v, G)$ with
\[
\lambda_v(x) = \bpsi_{v}([\pi_{v,x}]).
\]
The choice of $\pi_{v,x}$ is not unique. However, if $a(\pi_{v,x})$ is
any normal integral basis generator of $F_{\pi_{v,x}}/F_v$, with
Stickelberger factorisation (see Definition \ref{D:stickfac})
\begin{equation} \label{E:sf}
\br_G(a(\pi_{v,x})) = u(a(\pi_{v,x})) \cdot \br_G(a_{nr}(\pi_{v,x})) \cdot
\br_G(\vphi(\pi_{v,x})),
\end{equation}
then Proposition \ref{P:reseq}(c) implies that
$\Det(\br_G(\vphi(\pi_{v,x})))$ is independent of the choice of
$\pi_{v,x}$. Hence, if $\vphi(\pi_{v,x}) = \vphi_{v,s}$, say, then it
follows from Proposition \ref{P:reseq}(b) that the subgroup $\langle s
\rangle$ of $G$ (up to conjugation) and the determinant
$\Det(\br_G(\vphi_{v,s}))$ of the resolvend $\br_G(\vphi_{v,s})$ do
not depend upon the choice of $\pi_{v,x}$.

We write $q: G \to D$ for the obvious quotient map, and we use
the same symbol $q$ for the induced maps  
\begin{align*}
&K_0(O_FG, F^c) \to K_0(O_FD, F^c), \quad H^1(F, G) \to
H^1(F, D),\\
&H^1(F_v, G) \to H^1(F_v, D).
\end{align*}
Set
\[
\ov{x}:= q(x),\quad \pi_{v, \ov{x}}:= q(\pi_{v,x}).
\]
Then $\ov{x} \in \LC(O_FD)_S$ with
\[
\lambda_v(\ov{x}) = \bpsi_{D,v}(\pi_{v,\ov{x}})
\]
for each finite place $v$ of $F$, and $\ov{x}$ is fully ramified.

By hypothesis, $\LC(O_FD)_S$ satisfies Property R, and so there exists
$\rho \in \Hom(\Omega_F, D)$ with $[\rho] \in H^1_t(F, D)$ such that
\begin{equation} \label{E:d1}
\ov{x} = \bpsi_D([\rho])
\end{equation}
and
\begin{equation} \label{E:d2}
\rho \mid_{I_v} = \pi_{v,\ov{x}} \mid_{I_v}
\end{equation}
for each finite place $v$ of $F$. Hence, for each such $v$, we have
that
\[
\Det(\br_D(\vphi(\rho_v))) = \Det(\br_D(\vphi(\pi_{v,\ov{x}}))),
\]
using the notation established in \eqref{E:sf} above concerning
Stickelberger factorisations. As $\ov{x}$ is fully ramified, we see
from the proof of Theorem \ref{T:fieldres} that $\rho$ is surjective,
and so $F_{\rho}$ is a field.  We also see that, as $\ov{x} \in
\LC(O_FD)_S$, the extension $F_{\rho}/F$ is unramified at all places
dividing $|D|$. Furthermore, if $v \mid l$ (so $v \in S$), then since
$\pi_{v,x}$ is unramified, the same is true of $\pi_{v,\ov{x}}$, and
so $F_{\rho}/F$ is also unramified at $v$. Hence, as $F \cap \mu_{l} =
\{1\}$ by hypothesis, it follows that $F_{\rho} \cap \mu_{l} = \{1\}$
also.

For each finite place $v$ of $F$, we are now going to use the fact
that $x \in \LC(O_FG)$ to construct a lift $\ti{\rho}_v \in
\Hom(\Omega_{F_v}, G)$ of $\rho_v$ such that $[\ti{\rho}_v] \in
H^1_t(F_v, G)$ with
\begin{equation} \label{E:rhopi}
\ti{\rho}_{v} \mid_{I_v} = \pi_{v,x} \mid_{I_v}.
\end{equation}

To do this, we first observe that if $\vphi(\pi_{v,x}) =
\vphi_{v,s}$, then $\vphi(\pi_{v,\ov{x}}) = \vphi_{v,\ov{s}}$, where
$\ov{s} = q(s)$, and so we have
\[
\vphi(\rho_v) = \vphi(\pi_{v,\ov{x}}) = \vphi_{v,\ov{s}}
\]
(see \eqref{E:d2}).

Next, we write
\[
\rho_v = \rho_{v,r} \cdot \rho_{v,nr},
\]
with $[\rho_{v,nr}] \in H^{1}_{nr}(F_v, D)$ (see
\eqref{E:pifac}). Since $\rho_{v,nr}$ is unramified, Proposition
\ref{P:nrlift} implies that $[\rho_{v,nr}]$ may be lifted to
$[\ti{\rho}_{v,nr}] \in H^{1}_{nr}(F_v, G)$. Let $a(\ti{\rho}_{v,nr})$
be a normal integral basis generator of
$F_{\ti{\rho}_{v,nr}}/F_v$. Then $\br_G(a(\ti{\rho}_{v,nr})) \cdot
\br_G(\vphi_{v,s})$ is the resolvend of a normal integral basis
generator of a tame Galois $G$-extension $F_{\ti{\rho}_v}/F_v$ such
that $q([\ti{\rho}_v]) = \rho_v$ (cf. Corollary \ref{C:genfac} and
Theorem \ref{T:stickfac}). As $\vphi(\pi_{v,x}) =
\vphi_{v,s}$, we see from the construction of $\ti{\rho}$ that
\[
\ti{\rho}_{v}\mid_{I_v} = \pi_{v,x} \mid_{I_v} = \ti{\vphi}_{v,s},
\]
where $[\ti{\vphi}_{v,s}] \in H^{1}_{t}(I_v, G)$ is defined in Remark
\ref{R:ramphi}. The map $\ti{\rho}_v$ is our desired lift of
$\rho_v$.

We are now ready to apply the results contained in Section
\ref{S:neu}.  Consider the following diagram:
\[
\begin{CD}
0 @>>> B @>>> G @>{q}>> D @>>> 0 \\
@.         @.         @.             @AA{\rho}A  @. \\
@.         @.          @.            \Omega_F   @.\\
\end{CD}
\]
The group $D$ acts on $B$ via inner automorphisms, and we view $B$ as
being an $\Omega_F$-module via $\rho$. Then $B$ is a simple
$\Omega_F$-module because $B$ is a minimal normal subgroup of $G$ and
$\rho$ is surjective. The field of definition $F(B)$ of $B$ is
contained in the field $F_{\rho}$, and so in particular $F(B)$ contains no
non-trivial $l$-th roots of unity. We are going to construct an
element $\Pi \in \cHom_{D}(\Omega_F, G)$ such that
\[
\Pi \mid_{I_v} = \pi_{v,x} \mid_{I_v}
\]
for each finite place $v$ of $F$. This will be accomplished in the
following three steps:
\smallskip

I. We begin by observing that our construction above of a lift
$\ti{\rho}_v$ of $\rho_v$ for each finite $v$ shows that
$J_f(\cHom_{D}(\Omega_F, G))$ is non-empty.  Let $\cS$ be the set of
finite places $v$ of $F$ at which $x$ is ramified or $v \mid
|G|$. Theorem \ref{T:nl} implies that there exists $\Pi_1 \in
\cHom_{D}(\Omega_F, G)$ such that $\Pi_{1,v} = \ti{\rho}_v$ for all $v
\in \cS$. Observe that $\Pi_1$ is unramified at all $v \mid |G|$
because $\ti{\rho}_v$ is unramified at these places (see
\eqref{E:rhopi}). Note also that $\Pi_1$ may well be ramified outside
$\cS$.
\smallskip

II. Recall that $\cHom_{D}(\Omega_F, G)$ (respectively
$\cHom_{D}(\Omega_{F_v}, G)$ for each finite $v$) is a principal
homogeneous space over $H^1(F, B)$ (respectively $H^1(F_v, B)$). Let
$\cS_1$ denote the set of finite places $v \notin \cS$ of $F$ at which
$\Pi_1$ is ramified. For each $v \in \cS_1$, choose $y_v \in H^1(F_v,
B)$ so that $y_v \cdot \Pi_{1,v} \in \cHom_{D}(\Omega_{F_v}, G)$ is
unramified.

Now apply Definition \ref{D:ohell} (with $\Gamma = B$ and $T = \cS$) to
obtain an ideal $\fa = \fa(\cS)$ and an integer $N = N(\cS)$ as described
there. Theorem \ref{T:ncl} implies that there exists an element $z \in
H^1(F, B)$ such that

(z1) $z_v = y_v$ for all $v \in \cS_1$;

(z2) $z_v = 1$ for all $v \in \cS$;

(z3) If $v \notin \cS \cup \cS_1$, then $z_v$ is cyclic, and if $z_v$ is
ramified, then $v$ splits completely in $(F(B) \cdot F(\fa^N))/F$, where $F(\fa^N)$
denotes the ray class field of $F$ modulo $\fa^N$. 

Set $\Pi_2:= z \cdot \Pi_1 \in \cHom_{D}(\Omega_F, G)$.  Note that, as
$z$ might possibly be ramified, the homomorphism $\Pi_2$ might be
ramified outside $\cS$. We shall eliminate any such potential
ramification in the third and final step.
\smallskip

III. Let $\cS_z$ be the set of places of $F$ at which $z$ is ramified
(so $\cS \cap \cS_z = \emptyset$). We see from (z3) that each $v \in
\cS_z$ is totally split in $F(\fa^N)/F$. Hence Theorem \ref{T:bcl}
implies that for each $v \in \cS_z$, we may choose $b(v) \in H^1(F,
B)$ such that

(b1) $b(v)_w = 1$ for all $w \in \cS$;

(b2) $b(v) \mid_{I_v} = z^{-1}_{v} \mid_{I_v}$;

(b3) $b(v)$ is unramified away from $v$.

Set 
\[
\Pi := \left[ \left(\prod_{v \in S_{z}} b(v)\right) \cdot z\right] \cdot \Pi_2.
\]
Then it follows directly from the construction of $\Pi$ that we have
\begin{equation} \label{E:tgod}
\Pi \mid_{I_v} = \pi_{v,x} \mid_{I_v}
\end{equation}
for all finite places $v$ of $F$.
\smallskip

We claim that 
\[
x = \bpsi(\Pi).
\]
To show this,  let $\tau = \bpsi(\Pi)^{-1} \cdot x$. We see from \eqref{E:tgod}
that 
\[
\lambda_v(\tau) \in \Image(\bpsi^{nr}_{v})
\]
for every finite place $v$ of $F$. As either $G$ admits no irreducible
symplectic characters or $F$ has no real places, and as $(h_F, |G|) =
1$ by hypothesis, Proposition \ref{P:unram}(b) implies that $\tau =
0$. Hence $x = \Psi(\Pi)$, as claimed.

This completes the proof that $\LC(O_FG)_S$ satisfies Property R.
\end{proof}

Theorem \ref{T:domco} (in conjunction with Proposition \ref{P:rab})
yields an abundant supply of groups $G$ for which $\LC(O_FG)_S$
satisfies Property R (for a suitable choice of $S$), and therefore
also for which Theorem \ref{T:rgreat} holds.  Here is an example of
this.

\begin{theorem} \label{T:rodd}
Let $G$ be of odd order. Suppose that $(|G|, h_F)=1$ and that $F$
contains no non-trivial $|G|$-th roots of unity. Let $S$ be any finite
set of finite places of $F$ containing all places dividing $|G|$. Then
$\LC(O_FG)_S$ satisfies Property R.
\end{theorem}

\begin{proof}
We shall establish this result by induction on the order of
$G$. We first note that Proposition \ref{P:rab} implies that the
theorem holds if $G$ is abelian.

Suppose now that $G$ is an arbitrary finite group of odd order. As
$|G|$ is odd, a well known theorem of Feit and Thompson (see
\cite{FT63}) implies that $G$ is soluble. Hence $G$ has an abelian
minimal normal subgroup $B$ such that $l \cdot B = 0$ for some odd
prime $l$ (see e.g. \cite[Theorem 5.24]{Rot}), and there is an exact
sequence
\[
0 \to B \to G \to D \to 0
\]
with $D$ soluble. As $|G|$ is odd, $G$ admits no non-trivial
irreducible symplectic characters. We may therefore suppose by
induction on the order of $G$ that $\LC(O_FD)_S$ satisfies Property
R. The desired result now follows from Theorem \ref{T:domco}.
\end{proof}

\begin{remark} \label{R:even}
It follows from Theorem \ref{T:ab} that in Theorem \ref{T:domco}, we
may take $D$ to be a finite abelian group of arbitrary order (subject
of course to the obvious constraint that the number field $F$ is such
that all other conditions of Theorem \ref{T:domco} are
satisfied). This enables one to show that Property R holds for many
non-abelian groups of even order (e.g. $S_3$). However, if for example
$G$ is a non-abelian $2$-group (e.g. $H_8$), then because $\mu_2
\subseteq F$ for any number field $F$, we can no longer appeal to
Neukirch's Lifting Theorem, and our proof of Theorem \ref{T:domco}
fails. It appears very likely that new ideas are needed to establish
Property R in such cases (cf. also the remarks contained in the final
paragraph of \cite[Introduction]{Neu}, where a similar difficulty is
briefly discussed in the context of the inverse Galois problem for
finite groups). \qed
\end{remark}

We can now prove Theorem \ref{T:D} of the Introduction.

\begin{theorem} \label{T:odd}
Let $G$ be of odd order and suppose that $(|G|, h_F)=1$, where $h_F$
denotes the class number of $F$. Suppose also that $F$ contains no
non-trivial $|G|$-th roots of unity. Then $\cR(O_FG)$ is a subgroup of
$\Cl(O_FG)$.  If $c \in \cR(O_FG)$, then there exist infinitely many
$[\pi] \in H^1_t(F, G)$ such that $F_{\pi}$ is a field and $(O_{\pi})
= c$. The extensions $F_{\pi}/F$ may be chosen to have ramification
disjoint from any finite set $S$ of places of $F$.
\end{theorem}

\begin{proof}
This is an immediate consequence of Theorems \ref{T:rodd} and
\ref{T:rgreat}.
\end{proof}


\end{document}